\documentclass{amsart}

\usepackage{amsmath,amssymb}
\usepackage{mathbbol}
\usepackage{upgreek}
\usepackage{stmaryrd}
\usepackage[setpagesize=false]{hyperref}
\usepackage{url}
\usepackage{adjustbox}
\usepackage{tikz}
\usepackage{tikz-cd}
\usetikzlibrary{calc}
\usepackage{amsthm}
\numberwithin{equation}{section}
\theoremstyle{plain}
\newtheorem{C}{}[section] 
\newtheorem{lemma}[C]{Lemma}
\newtheorem{theorem}[C]{Theorem}
\newtheorem{corollary}[C]{Corollary}
\theoremstyle{definition}
\newtheorem{definition}[C]{Definition}
\theoremstyle{remark}
\newtheorem{remark}[C]{Remark}
\newtheorem{example}[C]{Example}
\newcommand{\bfk}{\mathbb{k}}           
\newcommand{\id}{\mathrm{id}}
\newcommand{\op}{\mathrm{op}}
\newcommand{\Hom}{\mathrm{Hom}}
\newcommand{\Nat}{\mathrm{Nat}}
\newcommand{\Fun}{\mathrm{Fun}}
\newcommand{\Rex}{\mathrm{Rex}}
\newcommand{\Proj}{\mathrm{Proj}}
\newcommand{\radj}{\mathrm{r.a.}}     
\newcommand{\rradj}{\mathrm{r.r.a.}}  
\newcommand{\lmod}[1]{{#1}\mathrm{\text{-}mod}}
\newcommand{\rmod}[1]{\mathrm{mod\text{-}}#1}
\newcommand{\bimod}[2]{#1\mathrm{\text{-}mod\text{-}}#2}
\newcommand{\lproj}[1]{{#1}\mathrm{\text{-}proj}}
\newcommand{\unitobj}{\mathbb{1}}
\newcommand{\eval}{\mathrm{ev}}
\newcommand{\coev}{\mathrm{coev}}
\newcommand{\rev}{\mathrm{rev}}
\newcommand{\piv}{\mathrm{piv}}
\newcommand{\DLD}{\mathbb{D}}     
\newcommand{\DRD}{\overline{\DLD}} 
\newcommand{\catactl}{\ogreaterthan} 
\newcommand{\catactr}{\olessthan}    
\newcommand{\Vect}{\mathbf{Vect}}
\newcommand{\vectactl}{\otimes}       
\newcommand{\trace}{\mathtt{t}}
\newcommand{\close}{\mathtt{cl}}
\newcommand{\TwTr}{\mathrm{Tr}}

\newcommand{\Nak}{\mathbb{N}}
\usepackage{ifthen}
\newcommand{\nakac}[1]{\widetilde{#1}}
\newcommand{\HH}{\mathrm{HH}}
\begin{document}

\title{Modified traces and the Nakayama functor}

\author[T. Shibata]{Taiki Shibata}
\address[T. Shibata]{Department of Applied Mathematics,
  Okayama University of Science \\
  1-1 Ridai-cho, Kita-ku Okayama-shi, Okayama 700-0005, Japan.}
\email{shibata@xmath.ous.ac.jp}
\author[K. Shimizu]{Kenichi Shimizu}
\address[K. Shimizu]{Department of Mathematical Sciences,
  Shibaura Institute of Technology \\
  307 Fukasaku, Minuma-ku, Saitama-shi, Saitama 337-8570, Japan.}
\email{kshimizu@shibaura-it.ac.jp}

\date{}


\begin{abstract}
  We organize the modified trace theory with the use of the Nakayama functor of finite abelian categories. For a linear right exact functor $\Sigma$ on a finite abelian category $\mathcal{M}$, we introduce the notion of a $\Sigma$-twisted trace on the class $\mathrm{Proj}(\mathcal{M})$ of projective objects of $\mathcal{M}$.
  In our framework, there is a one-to-one correspondence between the set of $\Sigma$-twisted traces on $\mathrm{Proj}(\mathcal{M})$ and the set of natural transformations from $\Sigma$ to the Nakayama functor of $\mathcal{M}$. Non-degeneracy and compatibility with the module structure (when $\mathcal{M}$ is a module category over a finite tensor category) of a $\Sigma$-twisted trace can be written down in terms of the corresponding natural transformation.
  As an application of this principal, we give existence and uniqueness criteria for modified traces. In particular, a unimodular pivotal finite tensor category admits a non-zero two-sided modified trace if and only if it is spherical. Also, a ribbon finite tensor category admits such a trace if and only if it is unimodular.

\end{abstract}

\maketitle

\tableofcontents

\section{Introduction}

The linear algebraic notion of the trace was extended to pivotal tensor categories and have been used in some constructions of topological invariants in low-dimensional topology. Recently, to explore finer invariants from non-semisimple categories, {\em modified traces} are introduced and investigated in, {\it e.g.}, \cite{2018arXiv180100321B,2018arXiv180901122F,2018arXiv180900499G,MR2803849,MR3068949,MR2480500,MR2998839,MR4079628}. They are not 
defined all of endomorphisms in a tensor category, but only for a morphisms in a specified tensor ideal.

In this paper, we initiate a new approach to the modified trace theory.
We restrict ourselves to the case of {\em finite abelian categories}, {\it i.e.}, $\bfk$-linear categories that are equivalent to the category $\lmod{A}$ of finite-dimensional modules over a finite-dimensional algebra $A$. We establish our framework by using the (right exact) Nakayama functor introduced by Fuchs, Schaumann and Schweigert \cite{MR4042867}. For a finite abelian category $\mathcal{M}$, it is defined by
\begin{equation}
  \label{eq:intro-Nakayama-def}
  \Nak_{\mathcal{M}}: \mathcal{M} \to \mathcal{M}, \quad
  M \mapsto \int^{X \in \mathcal{C}} \Hom_{\mathcal{M}}(M, X)^* \vectactl X
  \quad (M \in \mathcal{M}).
\end{equation}

First of all, we explain where our main idea comes from.
For full generality, we rather consider {\em twisted} traces.
Let $\mathcal{M}$ be a finite abelian category over a field $\bfk$, and let $\Sigma$ be a $\bfk$-linear right exact endofunctor on $\mathcal{M}$.
By a {\em $\Sigma$-twisted trace} on a full subcategory $\mathcal{P} \subset \mathcal{M}$, we mean a family of $\bfk$-linear maps
\begin{equation*}
  \trace_{\bullet} = \{ \trace_{P} : \Hom_{\mathcal{M}}(P, \Sigma(P)) \to \bfk \}_{P \in \mathcal{P}}
\end{equation*}
satisfying the $\Sigma$-cyclicity condition:
\begin{equation*}
  \trace_{P}(g f) = \trace_{Q}(\Sigma(f) g)
  \quad (f : P \to Q, g : Q \to \Sigma(P), P, Q \in \mathcal{P}). 
\end{equation*}
A $\Sigma$-twisted trace $\trace_{\bullet}$ on $\mathcal{P}$ is said to be {\em non-degenerate} if the pairing
\begin{equation*}
  \Hom_{\mathcal{M}}(M, \Sigma(P)) \times \Hom_{\mathcal{M}}(P, M) \to \bfk,
  \quad (f, g) \mapsto \trace_{P}(f g)
\end{equation*}
is non-degenerate for all objects $M \in \mathcal{M}$ and $P \in \mathcal{P}$.

Now we consider the case where $\mathcal{M} = \lmod{A}$ for some finite-dimensional algebra $A$ and $\mathcal{P}$ is the full subcategory of projective objects of $\lmod{A}$. Then, by the Eilenberg-Watts equivalence, we may view $\Sigma$ as a finite-dimensional $A$-bimodule. With the use of the idea of universal traces \cite[Subsection 2.4.6]{2018arXiv180100321B}, we find that the class of $\Sigma$-twisted traces on $\mathcal{P}$ is in bijection with the space
\begin{equation}
  \label{eq:intro-HH0-dual}
  \HH_0(\Sigma)^* := \{ \lambda \in \Hom_{\bfk}(\Sigma, \bfk)
  \mid \text{$\lambda(s a) = \lambda(a s)$ for all $a \in A$ and $s \in \Sigma$} \}.
\end{equation}
The point is that $\HH_0(\Sigma)^*$ is naturally identified with the space of $A$-bimodule homomorphisms from $\Sigma$ to $A^*$ (see Subsection~\ref{subsec:non-degen-twist}). Since $\Nak_{\lmod{A}} \cong A^{*} \otimes_A(-)$ \cite{MR4042867}, this result is interpreted category-theoretically as follows: For a finite abelian category $\mathcal{M}$ and a $\bfk$-linear right exact endofunctor $\Sigma$ on $\mathcal{M}$, there is a bijection
\begin{equation}
  \label{eq:intro-trace-bij}
  \Nat(\Sigma, \Nak_{\mathcal{M}})
  \cong \{ \text{$\Sigma$-twisted traces on $\Proj(\mathcal{M})$} \},
\end{equation}
where $\Proj(\mathcal{M})$ is the full subcategory of projective objects of $\mathcal{M}$.

This observation suggests that a property of a $\Sigma$-twisted trace on $\Proj(\mathcal{M})$ can be described by properties of the corresponding natural transformation.
For instance, we show that a $\Sigma$-twisted trace is non-degenerate if and only if the corresponding natural transformation is invertible (Theorem~\ref{thm:tw-tr-A-mod-non-deg}).

Modified traces are required to be compatible with the tensor product or the action of a tensor category. In our theory, such a compatibility of a trace can be described in terms of the corresponding natural transformation (Theorem \ref{thm:tw-tr-module-compati}).
Unlike the existing theory of modified traces, we work in the {\em non-pivotal} setting.
Let $\mathcal{C}$ be a rigid monoidal category (which is not necessarily pivotal), and let $\mathcal{M}$ be a finite abelian category on which $\mathcal{C}$ acts linearly from the right (see Section~\ref{sec:compati-module} for the precise setting).
Suppose that $\Sigma : \mathcal{M} \to \mathcal{M}$ is a $\bfk$-linear right exact functor equipped with a structure
\begin{equation*}
  \Psi^{\Sigma}_{M,X} : \Sigma(M) \catactr X^{**} \to \Sigma(M \catactr X)
  \quad (M \in \mathcal{M}, X \in \mathcal{C})
\end{equation*}
of a `twisted' right $\mathcal{C}$-module functor (which is defined in a similar way as an ordinary module functor).
We say that a $\Sigma$-twisted trace $\trace_{\bullet}$ on $\Proj(\mathcal{M})$ is {\em compatible with the right $\mathcal{C}$-module structure} if the equation
\begin{equation*}
  \trace_{P}(\close_{P, \Sigma(P)|X}(f)) = \trace_{P \catactr X}(\Psi^{\Sigma}_{P,X} \circ f)
\end{equation*}
holds for all $P \in \Proj(\mathcal{M})$, $X \in \mathcal{C}$ and $f \in \Hom_{\mathcal{M}}(P \catactr X, \Sigma(P) \catactr X^{**})$, where
\begin{equation*}
  \close_{M,N | X} : \Hom_{\mathcal{M}}(M \catactr X, N \catactr X^{**}) \to \Hom_{\mathcal{M}}(M, N)
  \quad (M, N \in \mathcal{M}, X \in \mathcal{C})
\end{equation*}
is the operator defined by `closing' the $X$-strand in string diagrams.
An important fact we shall recall here is that the Nakayama functor $\Nak_{\mathcal{M}}$ also has a structure of a twisted right $\mathcal{C}$-module functor \cite{MR4042867}. Now we can characterize the module compatibility of a twisted trace in terms of the corresponding natural transformation as follows:

\begin{theorem}[{\it cf}. Theorem \ref{thm:tw-tr-module-compati}]
  \label{thm:intro-1}
  Suppose that a $\Sigma$-twisted modified trace $\trace_{\bullet}$ corresponds to a natural transformation $\xi: \Sigma \to \Nak_{\mathcal{M}}$ via the bijection \eqref{eq:intro-trace-bij}. Then $\trace_{\bullet}$ is compatible with the right $\mathcal{C}$-module structure if and only if $\xi$ is a morphism of twisted $\mathcal{C}$-module functors.
\end{theorem}

Now we suppose that $\mathcal{C}$ has a pivotal structure.
Then we can regard $\id_{\mathcal{M}}$ as a twisted right $\mathcal{C}$-module functor through the pivotal structure. By a {\em right modified trace} on $\Proj(\mathcal{M})$, we mean a $\Sigma$-twisted trace on $\Proj(\mathcal{M})$ with $\Sigma = \id_{\mathcal{M}}$ that is compatible with the right $\mathcal{C}$-module structure (this definition agrees with those considered in \cite{2018arXiv180100321B,MR2803849,MR3068949,MR2998839} in the case where $\mathcal{M} = \mathcal{C}$).
By this theorem and known results on the Nakayama functor, we obtain some existence and uniqueness results on right modified traces. In particular, we have:

\begin{theorem}[{\it cf}. Theorem~\ref{thm:FTC-mod-tr-1}]
  For a pivotal finite tensor category $\mathcal{C}$,
  a non-zero right modified trace on $\Proj(\mathcal{C})$ exists if and only if $\mathcal{C}$ is unimodular in the sense of \cite{MR2097289}.
  Furthermore, such a trace is unique up to scalar multiple if it exists.
\end{theorem}

The `if' part of this theorem follows from \cite[Corollary 5.6]{2018arXiv180900499G}.

One can define a left modified trace and a two-sided modified trace for a left $\mathcal{C}$-module category and a $\mathcal{C}$-bimodule category, respectively, in an analogous way as right modified traces. The relation between the Nakayama functor and the categorical Radford $S^4$-formula \cite{MR2097289,MR4042867} yields:

\begin{theorem}
  \label{thm:intro-2}
  For a pivotal finite tensor category $\mathcal{C}$, a non-zero two-sided modified trace on $\Proj(\mathcal{C})$ exists if and only if the pivotal structure of $\mathcal{C}$ is spherical in the sense of \cite{2013arXiv1312.7188D}.
\end{theorem}

By \cite[Lemma 5.9]{2017arXiv170709691S}, a ribbon finite tensor category is spherical if and only if it is unimodular. For applications in topological field theories, the following result could be important:

\begin{theorem}
  \label{thm:intro-3}
  For a ribbon finite tensor category $\mathcal{C}$, a non-zero two-sided modified trace on $\Proj(\mathcal{C})$ exists if and only if $\mathcal{C}$ is unimodular.
\end{theorem}

We have mainly considered the case where $\mathcal{M} = \mathcal{C}$ in the above.
In this paper, we also examine the case where $\mathcal{C} = \lmod{H}$ and $\mathcal{M} = \lmod{A}$ for some finite-dimensional Hopf algebra $H$ and finite-dimensional right $H$-comodule algebra $A$.
In this case, every module-compatible modified trace on $\Proj(\lmod{A})$ comes from a {\em grouplike-cointegral} on $A$ recently introduced in \cite{2018arXiv181007114K} and investigated in \cite{2018arXiv181007114K,2019arXiv190400376S}.
This correspondence extends a relation between modified traces on $\lproj{H}$ and cointegrals on $H$ established in \cite{2018arXiv180100321B}.
We also exhibit some concrete examples.

\begin{remark}
  Schweigert and Woike give a result similar to our main theorem and establish some applications of modified traces in the context of topological field theories in \cite{arXiv:2103.15772}.
\end{remark}

\subsection*{Organization of the paper}

Section~\ref{sec:preliminaries} fixes our convention and provides some basic materials on finite abelian categories. The Nakayama functor will also be introduced.

In Section~\ref{sec:twist-trace-HH},
after introducing fundamental notions in this paper,
we consider the case where $\mathcal{M} = \lmod{A}$ for some finite-dimensional algebra $A$.
Let $\Sigma$ be a $\bfk$-linear right exact endofunctor on $\lmod{A}$ and regard it as an $A$-bimodule through the Eilenberg-Watts equivalence.
By the way of \cite[Subsection 2.4.6]{2018arXiv180100321B}, we construct a family of $\bfk$-linear maps
\begin{equation*}
  \mathtt{T}_{\bullet}
  = \{ \mathtt{T}_{P}: \Hom_{A}(P, \Sigma(P)) \to \HH_0(\Sigma) \}_{P \in \Proj(\lmod{A})}
\end{equation*}
that generalizes the Hattori-Stallings trace. Here, $\HH_0(\Sigma)$ is the 0-th Hochschild homology of the bimodule $\Sigma$.
It is shown that $\mathtt{T}_{\bullet}$ has a certain universal property (Theorem~\ref{thm:initial-twisted-trace}). This gives a bijection between the class of $\Sigma$-twisted traces on $\Proj(\lmod{A})$ and the dual space of $\HH_0(\Sigma)$, which may be identified with the right-hand side of \eqref{eq:intro-HH0-dual} (Corollary \ref{cor:tw-tr-A-mod}). We show that a $\Sigma$-twisted trace $\trace_{\bullet}$ on $\Proj(\lmod{A})$ is non-degenerate if and only if the corresponding linear form $\lambda$ on $\HH_0(\Sigma)$ is non-degenerate in a certain sense (Theorem~\ref{thm:tw-tr-A-mod-non-deg}).

By the discussion of Section~\ref{sec:twist-trace-HH}, we see that there is a bijection \eqref{eq:intro-trace-bij} for every finite abelian category $\mathcal{M}$. The aim of Section~\ref{sec:cat-theo-refo} is to reformulate this bijection in purely category-theoretical terms for the use of later sections.
According to \cite{MR4042867}, the Nakayama functor $\Nak_{\mathcal{M}}$ for $\mathcal{M} = \lmod{A}$ is isomorphic to the functor $A^{*} \otimes_A (-)$. This means that $A^{*} \otimes_A M$ for $M \in \lmod{A}$ has a universal property as a coend as in \eqref{eq:intro-Nakayama-def}. To accomplish the aim of this section, we first write down the universal dinatural transformation for $A^{*} \otimes_A M$ explicitly. We also introduce the {\em canonical Nakayama-twisted trace} $\nakac{\trace}_{\bullet}$ by using the universal property of the Nakayama functor. It turns out that the bijection \eqref{eq:intro-trace-bij} is written in terms of $\nakac{\trace}_{\bullet}$. The inverse of the bijection \eqref{eq:intro-trace-bij} is also described without referencing an algebra $A$ such that $\mathcal{M} \approx \lmod{A}$.

In Section~\ref{sec:compati-module}, we consider the case where $\mathcal{M}$ is a finite abelian category on which a rigid monoidal category $\mathcal{C}$ acts linearly from the right. Suppose that $\Sigma : \mathcal{M} \to \mathcal{M}$ is a $\bfk$-linear right exact right $\mathcal{C}$-module functor from $\mathcal{M}$ to $\mathcal{M}_{\DLD}$, where $(-)_{\DLD}$ means the twist of the action by the double dual functor $\DLD = (-)^{**}$. Then, as in the above, the {\em compatibility} of a $\Sigma$-twisted trace with the right $\mathcal{C}$-module structure is defined. The main result of this section is that a $\Sigma$-twisted trace is module-compatible if and only if the corresponding natural transformation $\Sigma \to \Nak_{\mathcal{M}}$ is a morphism of right $\mathcal{C}$-module functors from $\mathcal{M}$ to $\mathcal{M}_{\DLD}$ (Theorem~\ref{thm:tw-tr-module-compati}).
By the computation using the universal property of the Nakayama functor, 
we see that the canonical Nakayama-twisted trace is module-compatible.
The general case follows from this result.

In Section~\ref{sec:applications-to-ftc}, we apply our results to finite tensor categories and their modules (in this section, the base field $\bfk$ is assumed to be algebraically closed for technical reasons).
Let $\mathcal{C}$ be a finite tensor category, and let $\mathcal{M}$ be a finite right $\mathcal{C}$-module category.
An easy, but important observation is that if $\mathcal{M}$ is indecomposable, then the Nakayama functor $\Nak_{\mathcal{M}}$ is a simple object of the category $\Rex_{\mathcal{C}}(\mathcal{M}, \mathcal{M}_{\DLD})$ of $\bfk$-linear right exact right $\mathcal{C}$-module functors from $\mathcal{M}$ to $\mathcal{M}_{\DLD}$.
Most of `uniqueness' results follow from this observation and Schur's lemma.
According to \cite{MR4042867}, the Nakayama functor of $\mathcal{C}$ is given by $\Nak_{\mathcal{C}}(X) = \alpha_{\mathcal{C}} \otimes X^{**}$ ($X \in \mathcal{C}$), where $\alpha_{\mathcal{C}}$ is the dual of the distinguished invertible object \cite{MR2097289}.
This fact relates the existence of a non-zero modified trace on $\Proj(\mathcal{C})$ and the condition that $\alpha_{\mathcal{C}} \cong \unitobj$ (unimodularity).
Based on these observations, we obtain some results on modified traces as summarized into Theorems~\ref{thm:intro-1}--\ref{thm:intro-3} in the above.

In Section~\ref{sec:comod-algebras}, we examine the case where $\mathcal{C} = \lmod{H}$ and $\mathcal{M} = \lmod{A}$ for some finite-dimensional Hopf algebra $H$ and finite-dimensional right $H$-comodule algebra $A$. If this is the case, the category $\Rex_{\mathcal{C}}(\mathcal{M}, \mathcal{M}_{\DLD})$, to which $\Nak_{\mathcal{M}}$ belongs, can be identified with the category of finite-dimensional $H$-equivariant $A^{(S^2)}$-$A$-bimodules in the sense of \cite{MR2331768}, where $A^{(S^2)}$ is the $H$-comodule algebra obtained from $A$ by twisting its coaction by the square of the antipode of $H$ (Lemma~\ref{lem:H-eq-EW-2}). Under this identification, the Nakayama functor corresponds to the $A$-bimodule $A^*$ with the right $H$-comodule structure given as in \cite{MR3537815} (Theorem~\ref{thm:Nakayama-as-equivariant-bimodule}). From this result, it turns out that twisted module-compatible traces (and in particular modified traces) are given in terms of linear forms on $A$ satisfying an equation similar to the axiom for cointegrals on Hopf algebras (Theorems~\ref{thm:mod-tr-cointegral-1} and~\ref{thm:mod-tr-cointegral-2}).
Finally, as concrete examples, we check the existence of non-zero modified traces for some comodule algebras discussed in Mombelli \cite{MR2678630}.

\subsection*{Acknowledgements}
We are grateful to the referees for careful reading of the manuscript.
The first author (T.S.) is supported by JSPS KAKENHI Grant Number JP19K14517.
The second author (K.S.) is supported by JSPS KAKENHI Grant Number JP20K03520.

\section{Preliminaries}
\label{sec:preliminaries}

\subsection{Monoidal categories}

A {\em monoidal category} \cite[VII.1]{MR1712872} is a category $\mathcal{C}$ equipped with a functor $\otimes : \mathcal{C} \times \mathcal{C} \to \mathcal{C}$ (called the tensor product), an object $\unitobj \in \mathcal{C}$ (called the unit object) and natural isomorphisms
$(X \otimes Y) \otimes Z \cong X \otimes (Y \otimes Z)$ and
$\unitobj \otimes X \cong X \cong X \otimes \unitobj$
($X,Y,Z \in \mathcal{C}$)
satisfying the pentagon and the triangle axioms. By the Mac Lane coherence theorem, we may, and do, assume that every monoidal category is strict.

Let $\mathcal{C}$ and $\mathcal{D}$ be monoidal categories.
A {\em monoidal functor} \cite[XI.2]{MR1712872} from $\mathcal{C}$ to $\mathcal{D}$ is a triple $(F, \varphi, \psi)$ consisting of a functor $F : \mathcal{C} \to \mathcal{D}$, a natural transformation $\varphi_{X,Y} : F(X) \otimes F(Y) \to F(X \otimes Y)$ ($X, Y \in \mathcal{C}$) and a morphism $\psi : \unitobj \to F(\unitobj)$ subject to certain axioms. A monoidal functor $(F, \varphi, \psi)$ is said to {\em strong} ({\it resp}. {\em strict}) if $\varphi$ and $\psi$ are isomorphisms ({\it resp}. identities).

A {\em left dual object} of an object $X$ of a monoidal category $\mathcal{C}$ is a triple $(L, \varepsilon, \eta)$ consisting of an object $L \in \mathcal{C}$ and morphisms $\varepsilon : L \otimes X \to \unitobj$ and $\eta : \unitobj \to X \otimes L$ in $\mathcal{C}$ such that the equations
$(\varepsilon \otimes \id_L) (\id_L \otimes \eta) = \id_L$ and
$(\id_X \otimes \varepsilon) (\eta \otimes \id_X) = \id_X$
hold ({\it cf}. \cite{MR3242743}). A {\em right dual object} of $X$ is a triple $(R, \varepsilon, \eta)$ such that $(X, \varepsilon, \eta)$ is a left dual object of $R$.

A monoidal category is said to be {\em rigid} if every its object has a left dual and a right dual object. Now suppose that $\mathcal{C}$ is a rigid monoidal category. Given an object $X \in \mathcal{C}$, we usually denote by
\begin{equation*}
  (X^*, \eval_X : X^* \otimes X \to \unitobj, \coev_X : \unitobj \to X \otimes X^*)
\end{equation*}
a (fixed) left dual object of $X$. The assignment $X \mapsto X^*$ extends to a strong monoidal functor from $\mathcal{C}^{\op}$ to $\mathcal{C}^{\rev}$, which we call the {\em left duality functor}. Here, $\mathcal{C}^{\rev}$ is the category $\mathcal{C}$ equipped with the reversed tensor product $X \otimes^{\rev} Y = Y \otimes X$.

A right dual object of $X \in \mathcal{C}$ will be denoted by $({}^* \! X, \eval_X', \coev_X')$.
The assignment $X \mapsto {}^*\!X$ also extends to a strong monoidal functor from $\mathcal{C}^{\op}$ to $\mathcal{C}^{\rev}$, which we call the {\em right duality functor}. By replacing $\mathcal{C}$ with an equivalent one, we may assume that the left and the right duality are strict monoidal functors and mutually inverse to each other (see, {\it e.g.}, \cite[Lemma 5.4]{MR3314297}). Thus we have $(X \otimes Y)^* = Y^* \otimes X^*$,
${}^*(X^*) = X = ({}^*\!X)^*$, etc.

\subsection{Module categories}

Let $\mathcal{C}$ be a monoidal category. A {\em right $\mathcal{C}$-module category} \cite{MR3934626,MR3242743,MR2794862} is a category $\mathcal{M}$ equipped with a functor $\catactr : \mathcal{M} \times \mathcal{C} \to \mathcal{M}$, called the action, and natural transformations
\begin{equation}
  \label{eq:mod-cat-associator}
  (M \catactr X) \catactr Y \cong M \catactr (X \otimes Y)
  \quad \text{and} \quad M \catactr \unitobj \cong M
  \quad (M \in \mathcal{M}, X, Y \in \mathcal{C})
\end{equation}
satisfying certain axioms similar to those for a monoidal category.
A right $\mathcal{C}$-module category is said to be {\em strict} if the natural isomorphisms \eqref{eq:mod-cat-associator} are the identity. A {\em left module category} and a {\em bimodule category} are defined analogously.
It is known that an analogue of Mac Lane's strictness theorem for monoidal categories holds for module categories (see \cite[Theorem 1.3.8]{MR2794862}). Thus we may, and do, assume that all module categories are strict in this paper.

Below, following \cite{MR3934626}, we introduce technical terms for action-preserving functors between module categories. We only mention the case of right module categories; see \cite{MR3934626} for details.
Given two right $\mathcal{C}$-module categories $\mathcal{M}$ and $\mathcal{N}$, a {\em lax right $\mathcal{C}$-module functor} from $\mathcal{M}$ to $\mathcal{N}$ is a functor $F : \mathcal{M} \to \mathcal{N}$ equipped with a natural transformation
\begin{equation*}
  \Psi^{F}_{M,X} : F(M) \catactr X \to F(M \catactr X)
  \quad (M \in \mathcal{M}, X \in \mathcal{C})
\end{equation*}
such that the equations
\begin{equation}
  \label{eq:lax-right-C-mod}
  \Psi^{F}_{M, \unitobj} = \id_{F(M)}
  \quad \text{and} \quad
  \Psi^{F}_{M, X \otimes Y} = \Psi^{F}_{M \catactr X, Y} \circ (\Psi^{F}_{M,X} \catactr \id_Y)
\end{equation}
hold for all objects $M \in \mathcal{M}$ and $X, Y \in \mathcal{C}$.
An {\em oplax right $\mathcal{C}$-module functor} from $\mathcal{M}$ to $\mathcal{N}$ is a functor $F : \mathcal{M} \to \mathcal{N}$ equipped with a natural transformation
\begin{equation*}
  F(M \catactr X) \to F(M) \catactr X
  \quad (M \in \mathcal{M}, X \in \mathcal{C})
\end{equation*}
satisfying similar equations as \eqref{eq:lax-right-C-mod}.
We omit the definition of morphisms of (op)lax $\mathcal{C}$-module functors.

A lax or oplax right $\mathcal{C}$-module functor is said to be {\em strong} if its structure morphism is invertible. In this paper, we only consider the case where $\mathcal{C}$ is rigid. If this is the case, then every (op)lax right $\mathcal{C}$-module functor is strong \cite[Lemma 2.10]{MR3934626} and hence we may call them simply a right $\mathcal{C}$-module functor.
Nevertheless, the adjective `op(lax)' is sometimes used to specify the direction of the structure morphism.

\subsection{Finite abelian categories}

Throughout this paper, we work over a fixed field $\bfk$ (this field is arbitrary except in Section~\ref{sec:applications-to-ftc}).
By an algebra, we mean an associative unital algebra over the field $\bfk$.
Given algebras $A$ and $B$, we denote by $\lmod{A}$, $\rmod{B}$ and $\bimod{A}{B}$
the category of finite-dimensional left $A$-modules, 
the category of finite-dimensional right $B$-modules
and the category of finite-dimensional $A$-$B$-bimodules, respectively.

Given a vector space $V$ over $\bfk$, we denote by $V^* := \Hom_{\bfk}(V, \bfk)$ the dual space of $V$.
An arbitrary element of $V^*$ is often written like $v^*$ (this symbol does not mean an element of $V^*$ obtained from an element $v$ by applying some operation $*$).
If $M$ and $N$ are a left and a right module over an algebra $A$, then $M^*$ and $N^*$ are a right and a left $A$-module by the actions defined by
\begin{equation}
  \label{eq:hit-actions}
  \langle m^* \leftharpoonup a, m \rangle = \langle m^*, a m \rangle
  \quad \text{and} \quad
  \langle a \rightharpoonup n^*, n \rangle = \langle n^*, n a \rangle,
\end{equation}
respectively, for $m^* \in M^*$, $m \in M$, $n^* \in N^*$, $n \in N$ and $a \in A$. The assignments $M \mapsto M^*$ and $N \mapsto N^*$ extend to anti-equivalences between $\lmod{A}$ and $\rmod{A}$.

A {\em finite abelian category} \cite{MR3242743,MR3934626} is a $\bfk$-linear category that is equivalent to $\lmod{A}$ for some finite-dimensional algebra $A$. The class of finite abelian categories is closed under taking the opposite category. Indeed, if $\mathcal{M} \approx \lmod{A}$ for some finite-dimensional algebra $A$, then we have $\mathcal{M}^{\op} \approx \rmod{A} \approx \lmod{A^{\op}}$.

Given two finite abelian categories $\mathcal{M}$ and $\mathcal{N}$, we denote by $\Rex(\mathcal{M}, \mathcal{N})$ the category of $\bfk$-linear right exact functors from $\mathcal{M}$ to $\mathcal{N}$. For finite-dimensional algebras $A$ and $B$, there is an equivalence
\begin{equation}
  \label{eq:EW-equivalence}
  \bimod{B}{A} \xrightarrow{\quad \approx \quad} \Rex(\lmod{A}, \lmod{B}),
  \quad M \mapsto M \otimes_{A} (-)
\end{equation}
of $\bfk$-linear categories, which we call {\em the Eilenberg-Watts equivalence}. By using this equivalence, we see that $\Rex(\mathcal{M}, \mathcal{N})$ is also a finite abelian category.

For $a \in A$, we define $r_a : A \to A$ by $r_a(b) = b a$ ($b \in A$).
A quasi-inverse of the equivalence \eqref{eq:EW-equivalence} is given by $F \mapsto F(A)$, where the right $A$-module structure of $F(A) \in \lmod{B}$ is given by $m \cdot a = F(r_a)(m)$ for $m \in F(A)$ and $a \in A$.
This means that an object $F \in \Rex(\lmod{A}, \lmod{B})$ is determined by its restriction to the full subcategory $\{ A \}$ of $\lmod{A}$.
Since $A$ is a projective object of $\lmod{A}$, we may say that $F$ is determined by its restriction to the full subcategory of projective objects of $\lmod{A}$.
For later use, we note this consequence as the following lemma:

\begin{lemma}
  \label{lem:restriction-to-proj}
  For finite abelian categories $\mathcal{M}$ and $\mathcal{N}$, the functor
  \begin{equation*}
    \Rex(\mathcal{M}, \mathcal{N}) \to \Fun_{\bfk}(\Proj(\mathcal{M}), \mathcal{N}),
    \quad F \mapsto F|_{\Proj(\mathcal{M})}
  \end{equation*}
  is fully faithful. Here, $\Proj(\mathcal{M})$ is the full subcategory projective objects of $\mathcal{M}$ and $\Fun_{\bfk}(\mathcal{A}, \mathcal{B})$ means the category of $\bfk$-linear functors from $\mathcal{A}$ to $\mathcal{B}$.
\end{lemma}

The functor $M \otimes_{A} (-)$ for $M \in \bimod{B}{A}$ has a right adjoint $\Hom_{B}(M, -)$. By the Eilenberg-Watts equivalence~\eqref{eq:EW-equivalence}, we see that a $\bfk$-linear functor $F : \mathcal{M} \to \mathcal{N}$ has a right adjoint if and only if $F$ is right exact \cite[Corollary 1.9]{MR3934626}. By applying this result to $F^{\op} : \mathcal{M}^{\op} \to \mathcal{N}^{\op}$, we also see that $F$ has a left adjoint if and only if $F$ is left exact.

\subsection{Canonical $\Vect$-action}
\label{subsec:cano-vect-act}

We denote by $\Vect := \lmod{\bfk}$ the category of finite-dimensional vector spaces over $\bfk$. Let $\mathcal{M}$ be a finite abelian category, and let $M$ be an object of $\mathcal{M}$. Since the functor $y_M := \Hom_{\mathcal{M}}(M, -) : \mathcal{M} \to \Vect$ is $\bfk$-linear and left exact, $y_M$ has a left adjoint. We denote the left adjoint of $y_M$ by $V \mapsto V \vectactl M$ ($V \in \Vect$). In other words, $V \vectactl M$ is defined so that there is a natural isomorphism
\begin{equation}
  \label{eq:cano-Vec-act}
  \Hom_{\mathcal{M}}(V \vectactl M, N) \cong \Hom_{\bfk}(V, \Hom_{\mathcal{M}}(M, N))
\end{equation}
for $V \in \Vect$ and $N \in \mathcal{M}$. The assignment $(V, M) \mapsto V \vectactl M$ extends to a functor $\Vect \times \mathcal{M} \to \mathcal{M}$ in such a way that the isomorphism \eqref{eq:cano-Vec-act} is also natural in $M \in \mathcal{M}$. We call this functor the {\em canonical action} of $\Vect$ on $\mathcal{M}$ as it makes $\mathcal{M}$ a left module category over $\Vect$.

Let $F : \mathcal{M} \to \mathcal{N}$ be a $\bfk$-linear functor between finite abelian categories $\mathcal{M}$ and $\mathcal{N}$. By the definition of a $\bfk$-linear functor, for each pair $(M, M')$ of objects of $\mathcal{M}$, there is a $\bfk$-linear map
\begin{equation*}
  F|_{M,M'} : \Hom_{\mathcal{M}}(M, M') \to \Hom_{\mathcal{N}}(F(M), F(M')),
  \quad f \mapsto F(f).
\end{equation*}
For $V \in \Vect$ and $M \in \mathcal{M}$, we define the morphism
\begin{equation}
  \label{eq:cano-Vec-mod-str}
  \Psi^{F}_{V,M} : V \vectactl F(M) \to F(V \vectactl M)
\end{equation}
in $\mathcal{N}$ to be the image of the identity morphism on $V \vectactl M$ under
\begin{equation*}
  \newcommand{\xarr}[1]{\xrightarrow{\makebox[7.3em]{$\scriptstyle{#1}$}}}
  \begin{aligned}
    \Hom_{\mathcal{M}}(V \vectactl M, V \vectactl M)
    & \xarr{\eqref{eq:cano-Vec-act}}
    \Hom_{\bfk}(V, \Hom_{\mathcal{M}}(M, V \vectactl M)) \\
    & \xarr{\Hom_{\bfk}(V, F|_{M, V \vectactl M})}
    \Hom_{\bfk}(V, \Hom_{\mathcal{N}}(F(M), F(V \vectactl M)) \\
    & \xarr{\eqref{eq:cano-Vec-act}}
    \Hom_{\mathcal{M}}(V \vectactl F(M), F(V \vectactl M)).
  \end{aligned}
\end{equation*}
In this way, $F$ becomes a left $\Vect$-module functor by $\Psi^F = \{ \Psi^{F}_{V,M} \}$. This construction is functorial in the following sense:

\begin{lemma}
  \label{lem:cano-Vect}
  The above construction establishes a 2-functor from the 2-category of finite abelian categories, $\bfk$-linear functors and natural transformations to the 2-category of $\Vect$-module categories, $\Vect$-module functors and their morphisms.
\end{lemma}

This lemma is only a part of the general result on the relation between categories enriched over a bicategory and categories being acted by the bicategory \cite[Section 3]{MR1466618}. For our applications, the above special case is enough.

\begin{example}
  Let $A$ be a finite-dimensional algebra. The canonical $\Vect$-action on $\lmod{A}$ is given by
  $V \vectactl M = V \otimes_{\bfk} M$,
  $a \cdot (v \otimes_{\bfk} m) = v \otimes_{\bfk} a m$
  for $v \in V \in \Vect$, $m \in M \in \lmod{A}$ and $a \in A$.
  For $P \in \lmod{A}$, the canonical $\Vect$-module structure of $\Hom_{A}(P, -) : \lmod{A} \to \Vect$ is given by
  \begin{equation}
    \label{eq:cano-Vect-act-Hom}
    V \otimes_{\bfk} \Hom_{A}(P, M) \to \Hom_{A}(P, V \vectactl M),
    \quad v \otimes f \mapsto
    \Big( p \mapsto v \otimes f(p) \Big).
  \end{equation}
\end{example}

\begin{example}
  Let $A$ and $B$ be finite-dimensional algebras. For $M \in \bimod{B}{A}$, the canonical $\Vect$-module structure of $M \otimes_{A} (-) : \lmod{A} \to \lmod{B}$ is given by
  \begin{equation}
    \label{eq:cano-Vect-act-tensor}
    V \otimes_{\bfk} (M \otimes_A X)
    \to M \otimes_A (V \otimes_{\bfk} X),
    \quad v \otimes_{\bfk} (m \otimes_A x)
    \mapsto m \otimes_A (v \otimes x).
  \end{equation}
\end{example}

\subsection{The Nakayama functor}
\label{subsec:Naka-fun}

We assume that the reader is familiar with dealing with ends and coends \cite[IX]{MR1712872}.
For a finite abelian category $\mathcal{M}$, it can be shown that the coend
\begin{equation}
  \label{eq:Naka-def}
  \Nak_{\mathcal{M}}(M) := \int^{X \in \mathcal{M}} \Hom_{\mathcal{M}}(M, X)^* \vectactl X
\end{equation}
exists for all objects $M \in \mathcal{M}$. The {\em Nakayama functor} \cite[Definition 3.14]{MR4042867} of $\mathcal{M}$ is the $\bfk$-linear right exact endofunctor on $\mathcal{M}$ defined by $M \mapsto \Nak_{\mathcal{M}}(M)$.

\begin{lemma}[{\cite[Lemma 3.15]{MR4042867}}]
  \label{lem:Nakayama-A-mod}
  For a finite-dimensional algebra $A$, we have
  \begin{equation*}
    \Nak_{\lmod{A}} = A^* \otimes_A (-),
  \end{equation*}
  where $A^*$ is made into an $A$-bimodule by the actions $\rightharpoonup$ and $\leftharpoonup$ defined by \eqref{eq:hit-actions}.
\end{lemma}

This lemma implies that there is a universal dinatural transformation
\begin{equation*}
  \mathbb{i}_{X,M} : \Hom_{A}(M, X)^* \otimes_{\bfk} X \to A^* \otimes_A M
  \quad (X, M \in \lmod{A})
\end{equation*}
making $A^* \otimes_A M$ into a coend of the form~\eqref{eq:Naka-def}. In \cite{MR4042867}, the above lemma is proved rather an indirect way and an explicit form of $\mathbb{i}_{X,M}$ is not mentioned. We will discuss it in Subsection~\ref{subsec:Naka-revisited}.

A {\em Frobenius form} on $A$ is a linear map $\lambda : A \to \bfk$ such that the associated pairing $\langle a, b \rangle := \lambda(a b)$ on $A$ is non-degenerate. Now we suppose that a Frobenius form $\lambda$ on $A$ is given. Then the {\em Nakayama automorphism} (associated to $\lambda$) is defined to be the unique algebra automorphism $\nu$ such that $\langle b, a \rangle = \langle \nu(a), b \rangle$ for all $a, b \in A$. Given a left $A$-module $M$, we denote by ${}_{\nu}M$ the left $A$-module obtained from $M$ by twisting the action of $A$ by $\nu$. There is an isomorphism $\theta : {}_{\nu}A \to A^*$ of $A$-bimodules given by $\theta(a) = \lambda \leftharpoonup a$ ($a \in A$) and thus we have natural isomorphisms
\begin{equation*}
  \Nak_{\lmod{A}}(M) = A^* \otimes_A M \cong {}_{\nu} A \otimes_A M \cong {}_{\nu} M
\end{equation*}
for $M \in \lmod{A}$. This is the reason why the Nakayama functor is called by such a name.

A Frobenius form on $A$ is said to be {\em symmetric} if the associated pairing is. It is well-known that $A$ admits a symmetric Frobenius form if and only if the Nakayama automorphism of $A$ is an inner automorphism. In terms of the Nakayama functor, this is also equivalent to that $\Nak_{\lmod{A}}$ is isomorphic to the identity functor \cite[Proposition 3.24]{MR4042867}.

The Nakayama functor will play a central role in this paper.
An important fact we will use is that the Nakayama functor between module categories has a structure of a `twisted' module functor \cite[Theorem 4.4]{MR4042867}.
This fact will be reviewed in Subsection~\ref{subsec:Naka-tw-module} in detail.

\section{Twisted traces and the Hochschild homology}
\label{sec:twist-trace-HH}

\subsection{Vector-valued twisted traces}

We first introduce the main subject of this paper:

\begin{definition}
  Let $\mathcal{M}$ be a $\bfk$-linear category, let $\mathcal{P}$ be a full subcategory of $\mathcal{M}$, and let $\Sigma : \mathcal{M} \to \mathcal{M}$ be a $\bfk$-linear endofunctor. Given a vector space $V$ over $\bfk$, a {\em $V$-valued $\Sigma$-twisted trace on $\mathcal{P}$} is a family
  \begin{equation*}
    \trace_{\bullet} = \{ \trace_P : \Hom_{\mathcal{M}}(P, \Sigma(P)) \to V \}_{P \in \mathcal{P}}
  \end{equation*}
  of $\bfk$-linear maps parametrized by the objects of $\mathcal{P}$ such that the following $\Sigma$-cyclicity condition is satisfied:
  \begin{description}
  \item [$\Sigma$-cyclicity] For all objects $P, Q \in \mathcal{P}$ and morphisms $f : P \to Q$ and $g : Q \to \Sigma(P)$ in $\mathcal{M}$, the following equation holds:
    \begin{equation}
      \label{eq:Sigma-cyclicity}
      \trace_{P}\Big(
        P \xrightarrow{\ f \ } Q \xrightarrow{\ g \ } \Sigma(P)
      \Big)
      = \trace_{Q}\Big(
        Q \xrightarrow{\ g \ } \Sigma(P) \xrightarrow{\ \Sigma(f) \ } \Sigma(Q)
      \Big).
    \end{equation}
  \end{description}
  If $V = \bfk$, then a $V$-valued $\Sigma$-twisted trace on $\mathcal{P}$ is simply called a {\em $\Sigma$-twisted trace on $\mathcal{P}$}. We say that a $\Sigma$-twisted trace $\trace_{\bullet}$ on $\mathcal{P}$ is {\em non-degenerate} if the pairing
  \begin{equation*}
    \Hom_{\mathcal{M}}(M, \Sigma(P)) \times \Hom_{\mathcal{M}}(P, M) \to \bfk,
    \quad (f, g) \mapsto \trace_{P}(f g)
  \end{equation*}
  is non-degenerate for all objects $M \in \mathcal{M}$ and $P \in \mathcal{P}$.
\end{definition}


Similar notions have been considered in, {\it e.g.}, \cite{2018arXiv180901122F,MR3910068}. Unlike these papers, we do not require that $\Sigma$ is an endofunctor on $\mathcal{P}$. The reason is that the case where $\mathcal{M}$ is a finite abelian category, $\mathcal{P} = \Proj(\mathcal{M})$ and $\Sigma = \Nak_{\mathcal{M}}$ will be especially important in this paper. In this case, the full subcategory $\mathcal{P}$ is not closed under $\Sigma$ in general. Indeed, if $\mathcal{M} = \lmod{A}$ for some finite-dimensional algebra $A$ which is not self-injective, then $A$ belongs to $\Proj(\mathcal{M})$ but $\Nak_{\mathcal{M}}(A) \cong A^*$ does not.

We introduce the following notation:

\begin{definition}
  Let $\mathcal{M}$ be a finite abelian category, and let $\Sigma: \mathcal{M} \to \mathcal{M}$ be a $\bfk$-linear endofunctor on $\mathcal{M}$.
  For a vector space $V$, we denote by $\TwTr(\Sigma, V)$ the class of $V$-valued $\Sigma$-twisted traces on $\Proj(\mathcal{M})$. For the case $V = \bfk$, we write $\TwTr(\Sigma) := \TwTr(\Sigma, \bfk)$.
\end{definition}

\subsection{Dual bases}
\label{subsec:dual-bases}

Let $A$ be a finite-dimensional algebra. We denote by $\lproj{A} := \Proj(\lmod{A})$ the full subcategory of projective objects of $\lmod{A}$. The aim of this section is to classify $\Sigma$-twisted traces on $\lproj{A}$ in the case where $\Sigma = M \otimes_A (-)$ for some $M \in \bimod{A}{A}$. For this purpose, we first recall basic results on finite-dimensional projective modules.

For $M \in \lmod{A}$, we set $M^{\dagger} := \Hom_{A}(M, A)$ and make it a right $A$-module by
\begin{equation*}
  (m^{\dagger} \cdot a)(m) = m^{\dagger}(m) a
  \quad (m^{\dagger} \in M^{\dagger}, m \in M, a \in A).
\end{equation*}

\begin{lemma}[the dual basis lemma; see {\cite[\S2B]{MR1653294}}]
  For every $P \in \lproj{A}$, there are finite number of elements $p_1, \cdots, p_n \in P$ and $p^1, \cdots, p^n \in P^{\dagger}$ such that the equation $p^i(p) \cdot p_i = p$ holds for all elements $p \in P$, where the Einstein convention is used to suppress the sum over $i$.
\end{lemma}

We call such a system $\{ p_i, p^i \}$ a {\em pair of dual bases} for $P \in \lproj{A}$.

For $P, M \in \lmod{A}$, there is a natural transformation $\theta$ defined by
\begin{equation}
  \label{eq:nat-trans-theta}
  \theta_{P,M} : P^{\dagger} \otimes_A M \to \Hom_{A}(P, M),
  \quad \theta_{P,M}(p^{\dagger} \otimes_A m)(p) = p^{\dagger}(p) \cdot m
\end{equation}
for $p^{\dagger} \in P$, $p \in P$ and $m \in M$. If $P \in \lproj{A}$, the map $\theta_{P,M}$ is invertible with the inverse
\begin{equation}
  \label{eq:nat-trans-theta-inv}
  \Hom_{A}(P, M) \to P^{\dagger} \otimes_A M,
  \quad \xi \mapsto p^i \otimes_A \xi(p_i)
  \quad (\xi \in \Hom_{A}(P, M)),
\end{equation}
where $\{ p_i, p^i \}$ is a pair of dual bases for $P$. The following lemma is well-known:

\begin{lemma}
  \label{lem:dual-basis-coev}
  Let $P$ and $Q$ be objects of $\lproj{A}$, and let $\{ p_i, p^i \}$ and $\{ q_j, q^j \}$ be pairs of dual bases for $P$ and $Q$, respectively. For every morphism $f : P \to Q$ of left $A$-modules, we have
  $p^i \otimes_A f(p_i) = (q^j \circ f) \otimes_A q_j$
  in $P^{\dagger} \otimes_A Q$.
\end{lemma}
\begin{proof}
  The bijection $\theta_{P,Q} : P^{\dagger} \otimes_A Q \to \Hom_A(P, Q)$ maps both sides to $f$.
\end{proof}

\subsection{Twisted traces and the Hochschild homology}
\label{subsec:tw-tr-HH0}

Let $A$ be a finite-\hspace{0pt}dimensional algebra, and let $\Sigma$ be a finite-\hspace{0pt}dimensional $A$-bimodule. By abuse of notation, we denote the endofunctor $\Sigma \otimes_A (-)$ on $\lmod{A}$ by the same symbol $\Sigma$. Thus, $\Sigma(M) = \Sigma \otimes_A M$ and $\Sigma(f) = \id_{\Sigma} \otimes_A f$ for an object $M \in \lmod{A}$ and a morphism $f$ in $\lmod{A}$.

We define the category $\mathfrak{T}$ (just for this subsection) as follows: An object of this category is a pair $(V, \trace_{\bullet})$ consisting of a vector space $V$ and $\trace_{\bullet} \in \TwTr(\Sigma, V)$. A morphism $f : (V, \trace_{\bullet}) \to (V', \trace'_{\bullet})$ in this category is a $\bfk$-linear map $f : V \to V'$ such that the equation $f \circ \trace_{P} = \trace'_{P}$ holds for all objects $P \in \lproj{A}$.

For the classification of vector-valued $\Sigma$-twisted traces on $\lproj{A}$, we construct an initial object of the category $\mathfrak{T}$ as follows: For an $A$-bimodule $M$, the 0-th Hochschild homology of $M$ is the vector space defined by
\begin{equation}
  \label{eq:HH0-def}
  \HH_0(M) := M / \mathrm{span}_{\bfk} \{ a m - m a \mid a \in A, m \in M \}.
\end{equation}
For $M \in \lmod{A}$, there is a well-defined $\bfk$-linear map
\begin{equation}
  \label{eq:HH0-trace-def-1}
  \begin{aligned}
    M^{\dagger} \otimes_A \Sigma(M)
    = M^{\dagger} \otimes_A \Sigma \otimes_A M
    & \to \HH_0(\Sigma), \\
    m^{\dagger} \otimes_A s \otimes_A m
    & \mapsto [m^{\dagger}(m) s],
  \end{aligned}
\end{equation}
where $[x]$ for $x \in \Sigma$ expresses the equivalence class of $x$ in $\HH_0(\Sigma)$. Now, for $P \in \lproj{A}$, we define the $\bfk$-linear map $\mathtt{T}_P$ to be the composition
\begin{equation}
  \label{eq:HH0-trace-def-2}
  \mathtt{T}_P = \Big(
  \Hom_{A}(P, \Sigma(P))
  \xrightarrow{\theta_{P,\Sigma(P)}^{-1}}
  P^{\dagger} \otimes_A \Sigma(P)
  \xrightarrow{\text{\eqref{eq:HH0-trace-def-1} with $M = P$}}
  \HH_0(\Sigma)
  \Big).
\end{equation}
Given a pair of dual bases $\{ p_i, p^i \}$ for $P$, the map $\mathtt{T}_P$ is expressed as
\begin{equation}
  \label{eq:HH0-trace-formula}
  \mathtt{T}_P(f)
  = \Big[ p^i(f_P(p_i)) \cdot f_{\Sigma}(p_i) \Big]
  \quad (f \in \Hom_A(P, \Sigma(P))),
\end{equation}
where $f(p) \in \Sigma \otimes_A P$ for $p \in P$ is expressed symbolically as $f(p) = f_{\Sigma}(p) \otimes_A f_{P}(p)$, although it may not be a single tensor in general. It follows from this formula and Lemma~\ref{lem:dual-basis-coev} that $\mathtt{T}_{\bullet}$ is an $\HH_0(\Sigma)$-valued $\Sigma$-twisted trace on $\lproj{A}$.

The $\Sigma$-twisted trace $\mathtt{T}_{\bullet}$ specializes to the Hattori-Stallings trace \cite{MR175950,MR202807} if $\Sigma = \id_{\mathcal{M}}$. Therefore $\mathtt{T}_{\bullet}$ is called the {\em twisted Hattori-Stallings trace} in \cite[Subsection 2.4.6]{MR3910068}. A universal property of $\mathtt{T}_{\bullet}$ mentioned in \cite{MR3910068} is, in our notation, rephrased as follows:

\begin{theorem}
  \label{thm:initial-twisted-trace}
  The pair $\mathbf{H} := (\HH_0(\Sigma), \mathtt{T}_{\bullet})$ is an initial object of the category $\mathfrak{T}$.
\end{theorem}
\begin{proof}
  We include a detailed proof since some parts of the proof will be used in later.
  We first introduce the following two isomorphisms of vector spaces:
  \begin{gather}
    \label{eq:HH0-initial-pf-iso-1}
    \phi : \Sigma \to \Hom_{A}(A, \Sigma(A)),
    \quad \phi(s)(a) = a s \otimes_A 1 \quad (s \in \Sigma, a \in A), \\
    \label{eq:HH0-initial-pf-iso-2}
    \phi' : A \to \Hom_{A}(A, A),
    \quad \phi'(a)(b) = b a \quad (a, b \in A).
  \end{gather}
  The following equations are easily verified:
  \begin{equation}
    \label{eq:HH0-initial-pf-iso-3}
    \phi(s) \circ \phi'(a) = \phi(a s),
    \quad (\id_{\Sigma} \otimes_A \phi'(a)) \circ \phi(s) = \phi(s a)
    \quad (s \in \Sigma, a \in A).
  \end{equation}
  By \eqref{eq:HH0-trace-formula} and that $\{ 1_A, \id_A \}$ is a pair of dual bases for $A \in \lproj{A}$, we have $\mathtt{T}_{A}(\phi(s)) = [s]$ for all $s \in \Sigma$. Thus, if $f$ is a morphism from $\mathbf{H}$ to $\mathbf{V} = (V, \trace_{\bullet})$ in $\mathfrak{T}$, then we have
  $f([s]) = f(\mathtt{T}_{A}(\phi(s))) = \trace_{A}(\phi(s))$
  for all $s \in \Sigma$. This means that a morphism $\mathbf{H} \to \mathbf{V}$ in $\mathfrak{T}$ is unique if it exists.

  To complete the proof, we show that there indeed exists a morphism $\mathbf{H} \to \mathbf{V}$ in $\mathfrak{T}$. In view of the above discussion, we define the linear map $\widetilde{\lambda} : \Sigma \to V$ by $\widetilde{\lambda}(s) = \trace_A(\phi(s))$ for $s \in \Sigma$. Then we have
  \begin{equation*}
    \widetilde{\lambda}(a s) 
    \mathop{=}^{\eqref{eq:HH0-initial-pf-iso-3}} \trace_A(\phi(s) \circ \phi'(a))
    \mathop{=}^{\eqref{eq:Sigma-cyclicity}}
    \trace_A(\Sigma(\phi'(a)) \circ \phi(s))
    \mathop{=}^{\eqref{eq:HH0-initial-pf-iso-3}}
    \widetilde{\lambda}(s a)
  \end{equation*}
  for $s \in \Sigma$ and $a \in A$. This means that the linear map
  \begin{equation*}
    \lambda : \HH_0(\Sigma) \to V,
    \quad \lambda([s]) = \trace_A(\phi(s))
    \quad (s \in \Sigma)
  \end{equation*}
  is well-defined. Now let $f : P \to \Sigma(P)$ be a morphism in $\lmod{A}$ with $P \in \lproj{A}$, and let $\{ p_i, p^i \}$ be a pair of dual bases for $P$. For each $i$, we define
  \begin{equation}
    \label{eq:HH0-initial-pf-p-tilde}
    \widetilde{p}_i : A \to P,
    \quad \widetilde{p}_i(a) = a p_i
    \quad (a \in A).
  \end{equation}
  Then the equation $\id_P = \widetilde{p}_i \circ p^i$ holds by the definition of a pair of dual bases (where the Einstein convention is used). Putting $\widetilde{f} := \Sigma(p^i) \circ f \circ \widetilde{p}_i$ and $s := \phi^{-1}(\widetilde{f})$, we compute
  \begin{gather*}
    \trace_P(f)
    = \trace_P(f \circ \widetilde{p}_i \circ p^i)
    \mathop{=}^{\eqref{eq:Sigma-cyclicity}}
    \trace_A(\Sigma(p^i) \circ f \circ \widetilde{p}_i) \\
    = \trace_A(\widetilde{f})
    = \trace_A(\phi(s))
    = \lambda([s])
    \mathop{=}^{\eqref{eq:HH0-trace-formula}} \lambda(\mathtt{T}_{P}(f)).
  \end{gather*}
  This means that $\lambda$ is a morphism $\mathbf{H} \to \mathbf{V}$ in $\mathfrak{T}$. The proof is done.
\end{proof}

The initial object $(\HH_0(\Sigma), \mathtt{T}_{\bullet})$ of Theorem~\ref{thm:initial-twisted-trace} may be called the `universal' vector-valued $\Sigma$-twisted trace on $\lproj{A}$. Indeed, we easily deduce from the theorem that for all vector spaces $V$, the map
\begin{equation}
  \label{eq:V-valued-trace}
  \Hom_{\bfk}(\HH_0(\Sigma), V)
  \to \TwTr(\Sigma, V),
  \quad f \mapsto f \circ \mathtt{T}_{\bullet}
\end{equation}
is bijective. Letting $V = \bfk$, we have:

\begin{corollary}
  \label{cor:tw-tr-A-mod}
  There is a bijection between $\TwTr(\Sigma)$ and $\HH_0(\Sigma)^*$.
\end{corollary}

\subsection{Non-degeneracy of twisted traces}
\label{subsec:non-degen-twist}

Let $A$ and $\Sigma$ be as in the previous subsection. We discuss which an element of $\HH_0(\Sigma)^*$ corresponds to a non-degenerate $\Sigma$-twisted trace on $\lproj{A}$ under the bijection \eqref{eq:V-valued-trace} with $V = \bfk$. By the definition of the 0-th Hochschild homology, we regard $\HH_0(\Sigma)^*$ as a subspace of $\Sigma^*$ by
\begin{equation*}
  \HH_0(\Sigma)^* = \{ \lambda \in \Sigma^* \mid \text{$\lambda(s a) = \lambda(a s)$ for all $s \in \Sigma$ and $a \in A$} \}.
\end{equation*}
Given $\lambda \in \Sigma^*$, we define the $\bfk$-linear map
\begin{equation}
  \label{eq:lambda-nat}
  \lambda^{\natural} : \Sigma \to A^*,
  \quad \langle \lambda^{\natural}(s), a \rangle = \langle \lambda, a s \rangle
  \quad (s \in \Sigma, a \in A).
\end{equation}
It is easy to verify that the map
\begin{equation}
  \label{eq:HH0-and-bimodules}
  \HH_0(\Sigma)^* \to \Hom_{A|A}(\Sigma, A^*),
  \quad \lambda \mapsto \lambda^{\natural}
\end{equation}
is an isomorphism of vector spaces, where $\Hom_{A|A}$ is the Hom functor of $\bimod{A}{A}$.

\begin{theorem}
  \label{thm:tw-tr-A-mod-non-deg}
  Given $\lambda \in \HH_0(\Sigma)^*$, we denote by $\trace^{\lambda}_{\bullet}$ the $\Sigma$-twisted trace on $\lproj{A}$ corresponding to $\lambda$ through the bijection given by Corollary~\ref{cor:tw-tr-A-mod}. We also define
  \begin{equation*}
    \beta^{\lambda}_{M,P}
    : \Hom_{A}(M, \Sigma(P))
    \times \Hom_{A}(P, M) \to \bfk,
    \quad (f, g) \mapsto \trace^{\lambda}_{P}(f g)
  \end{equation*}
  for $M \in \lmod{A}$ and $P \in \lproj{A}$.
  Then the following assertions are equivalent:
  \begin{itemize}
  \item [(1)] The $\Sigma$-twisted trace $\trace_{\bullet}^{\lambda}$ on $\lproj{A}$ is non-degenerate.
  \item [(2)] The pairing $\beta^{\lambda}_{P,Q}$ is non-degenerate for all $P, Q \in \lproj{A}$.
  \item [(3)] The map $\lambda^{\natural} : \Sigma \to A^*$ is an isomorphism.
  \end{itemize}
  Thus, in particular, a non-degenerate $\Sigma$-twisted trace on $\lproj{A}$ exists only if $\Sigma$ is isomorphic to $A^*$ as an $A$-bimodule.
\end{theorem}
\begin{proof}
  Let $\phi$ and $\phi'$ be the maps defined by \eqref{eq:HH0-initial-pf-iso-1} and~\eqref{eq:HH0-initial-pf-iso-2}, respectively

  \smallskip \noindent \underline{(2) $\Rightarrow$ (3)}.
  Set $\beta = \beta^{\lambda}_{A,A} \circ (\phi \times \phi')$.
  Since $\{ 1_A, \id_A \}$ is a pair of dual bases of $A$, we have
  \begin{equation}
    \label{eq:tw-tr-A-mod-pairing-formula}
    \beta (s, a)
    = \trace^{\lambda}_A(\phi(s) \circ \phi'(a))
    \mathop{=}^{\eqref{eq:HH0-initial-pf-iso-3}}
    \trace^{\lambda}_A(\phi(a s))
    \mathop{=}^{\eqref{eq:HH0-trace-formula}} \lambda(a s)
    \mathop{=}^{\eqref{eq:lambda-nat}} \langle \lambda^{\natural}(s), a \rangle
  \end{equation}
  for all $s \in \Sigma$ and $a \in A$. Suppose that (2) holds. Then, in particular, the pairing $\beta^{\lambda}_{A,A}$ is non-degenerate. Since $\phi$ and $\phi'$ are isomorphisms, the pairing $\beta : \Sigma \times A \to \bfk$ is also non-degenerate. By \eqref{eq:tw-tr-A-mod-pairing-formula}, the map $\lambda^{\natural} : \Sigma \to A^*$ is an isomorphism. Thus (3) holds.

  \smallskip \noindent
  \underline{(3) $\Rightarrow$ (2)}. We suppose that (3) holds. Let $P$ and $Q$ be objects of $\lproj{A}$, and let $\{ p^i, p_i \}$ and $\{ q_j, q^j \}$ be pairs of dual bases for $P$ and $Q$, respectively. For each $i$, we define $\widetilde{p}_i \in \Hom_A(A, P)$ by~\eqref{eq:HH0-initial-pf-p-tilde}. Then the equation $\widetilde{p}_i \circ p^i = \id_P$ holds (with the Einstein convention). We also define $\widetilde{q}_j \in \Hom_{A}(A, Q)$ in the same way so that the equation $\widetilde{q}_j \circ q^j = \id_Q$ holds.

  We verify the left non-degeneracy of the pairing $\beta^{\lambda}_{P,Q}$. Let $f : P \to \Sigma(Q)$ be a non-zero morphism in $\lmod{A}$. Then there are indices $u$ and $v$ such that $\Sigma(q^u) \circ f \circ \widetilde{p}_v \ne 0$, since, otherwise, we have
  \begin{equation*}
    f = \Sigma(\id_Q) \circ f \circ \id_P
    = \Sigma(\widetilde{q}_j) \circ \Sigma(q^j) \circ f \circ \widetilde{p}_i \circ p^i
    = 0,
  \end{equation*}
  a contradiction. We fix such indices $u$ and $v$ and set $s := \phi^{-1}(\Sigma(q^u) \circ f \circ \widetilde{p}_v)$. Since $s \ne 0$, and since $\lambda^{\natural}$ is assumed to be an isomorphism, there exists an element $a \in A$ such that $\lambda(s a) \ne 0$. By using such $a$, we define $g \in \Hom_A(Q, P)$ by $g := \widetilde{p}_v \circ \phi'(a) \circ q^u$. Then we have
  \begin{gather*}
    \beta^{\lambda}_{P,Q}(f, g)
    = \trace_Q^{\lambda}(f \circ \widetilde{p}_v \circ \phi'(a) \circ q^u)
    \mathop{=}^{\eqref{eq:Sigma-cyclicity}}
    \trace_A^{\lambda}(\Sigma(q^u) \circ f \circ \widetilde{p}_v \circ \phi'(a)) \\
    = \trace_A^{\lambda}(\phi(s) \circ \phi'(a))
    \mathop{=}^{\eqref{eq:tw-tr-A-mod-pairing-formula}} \lambda(a s) \ne 0
  \end{gather*}
  and therefore conclude that $\beta^{\lambda}_{P,Q}$ is left non-degenerate.

  The right non-degeneracy is proved in a similar way: Given a non-zero element $g \in \Hom_A(Q, P)$, we choose indices $u$ and $v$ such that $p^u \circ g \circ \widetilde{q}_v \ne 0$ and set $a := \phi'{}^{-1}(p^u \circ g \circ \widetilde{q}_v)$. Since $a \ne 0$, there is an element $s \in \Sigma$ such that $\lambda(a s) \ne 0$ and set $f := \Sigma(\widetilde{q}_v) \circ \phi(s) \circ p^u$. Then we have
  \begin{gather*}
    \beta^{\lambda}_{P,Q}(f, g)
    = \trace^{\lambda}_Q(\Sigma(\widetilde{q}_v) \circ \phi(s) \circ p^u \circ g)
    \mathop{=}^{\eqref{eq:Sigma-cyclicity}}
    \trace^{\lambda}_A(\phi(s) \circ p^u \circ g \circ \widetilde{q}_v) \\
    = \trace^{\lambda}_A(\phi(s) \circ \phi'(a))
    \mathop{=}^{\eqref{eq:tw-tr-A-mod-pairing-formula}} \lambda(a s) \ne 0
  \end{gather*}
  and therefore conclude that $\beta^{\lambda}_{P,Q}$ is right non-degenerate. Thus (2) holds.

  \smallskip \noindent
  \underline{(1) $\Rightarrow$ (2)}. This implication is trivial from the definition.

  \smallskip \noindent
  \underline{(2) $\Rightarrow$ (1)}. To complete the proof, we show that (2) implies (1).
  We assume that (2) holds (remark that this implies (3) as we have already proved in the above). We fix an object $P \in \lproj{A}$ and consider the following two $\bfk$-linear {\em covariant} functors:
  \begin{equation*}
    h_1, h_2 : \lmod{A} \to \Vect^{\op},
    \quad h_1 := \Hom_{A}(-, \Sigma(P)),
    \quad h_2 := \Hom_{A}(P, -)^*.
  \end{equation*}
  As is well-known, by applying the functor $\Hom_{A}(-, \Sigma(P))$ to a short exact sequence $0 \to X \to Y \to Z \to 0$ in $\lmod{A}$, we obtain an exact sequence
  \begin{equation*}
    0 \to \Hom_{A}(Z, \Sigma(P)) \to \Hom_{A}(Y, \Sigma(P)) \to \Hom_{A}(X, \Sigma(P))
  \end{equation*}
  in $\Vect$. This means that $h_1 : \lmod{A} \to \Vect^{\op}$ is {\em right} exact (as the target category is the opposite category of $\Vect$). The functor $h_2$ is also right exact as $P$ is projective.
  
  For $M \in \lmod{A}$, we define the $\bfk$-linear map $\alpha_M : h_1(M) \to h_2(M)$ by
  \begin{equation*}
    \langle \alpha_M(f), g \rangle = \trace^{\lambda}_{P}(f g)
    \quad (f \in \Hom_{A}(M, \Sigma(P)), g \in \Hom_{A}(P, M)).
  \end{equation*}
  It is easy to see that $\alpha = \{ \alpha_M \}_{M \in \lmod{A}}$ defines a natural transformation from $h_1$ to $h_2$. The assertion (2) is equivalent to that $\alpha$ induces a natural isomorphism between $h_1|_{\lproj{A}}$ and $h_2|_{\lproj{A}}$. Thus, by Lemma~\ref{lem:restriction-to-proj}, $\alpha : h_1 \to h_2$ is a natural isomorphism. This means that (3) holds. The proof is done.
\end{proof}

\section{Category-theoretical reformulation}
\label{sec:cat-theo-refo}

\subsection{The universal property of the Nakayama functor}
\label{subsec:Naka-revisited}

Let $\mathcal{M}$ be a finite abelian category, and let $\Sigma : \mathcal{M} \to \mathcal{M}$ be a $\bfk$-linear right exact endofunctor. We aim to give a `canonical' description of $\TwTr(\Sigma)$ that, unlike Theorem~\ref{thm:tw-tr-A-mod-non-deg}, does not reference a finite-dimensional algebra $A$ such that $\mathcal{M} \approx \lmod{A}$.

To state our result, the Nakayama functor $\Nak := \Nak_{\mathcal{M}}$ is essential. We will, in fact, establish a bijection between the space $\TwTr(\Sigma)$ and the space $\Nat(\Sigma, \Nak_{\mathcal{M}})$ of natural transformations from $\Sigma$ to $\Nak_{\mathcal{M}}$ in terms of the universal property of the Nakayama functor as a coend.

For a while, we consider the case where $\mathcal{M} = \lmod{A}$ for some finite-dimensional algebra $A$. In this case, the Nakayama functor is given by $\Nak_{\lmod{A}} = A^* \otimes_A (-)$. This means that, by the definition of $\Nak_{\mathcal{M}}$ as a coend, there is a universal dinatural transformation
\begin{equation}
  \label{eq:Nak-A-mod-univ}
  \mathbb{i}_{X,M} : \Hom_{A}(M, X)^* \otimes_{\bfk} X \to A^* \otimes_A M
  \quad (M, X \in \lmod{A}).
\end{equation}

For later use, we give an explicit description of the universal dinatural transformation~\eqref{eq:Nak-A-mod-univ}. We begin with the following observation: For each $X \in \lmod{A}$, there is the `action' map
\begin{equation}
  \rho_{X} : A \to \Hom_{\bfk}(X, X),
  \quad \rho_X(a)(x) = a x
  \quad (a \in A, x \in X).
\end{equation}
If we regard $\Hom_{\bfk}(X,X)$ as an $A$-bimodule by
\begin{equation}
  \label{eq:A-actions-on-Hom}
  (a \cdot \xi \cdot b)(x) = a \xi(b x)
  \quad (a, b \in A, \xi \in \Hom_{\bfk}(X,X), x \in X),
\end{equation}
then $\rho_X$ is a morphism of $A$-bimodules. It is easy to see that the family $\rho = \{ \rho_X \}_{X \in \lmod{A}}$ of morphisms of $A$-bimodules makes $A$ the end
\begin{equation*}
  A = \int_{X \in \lmod{A}} \Hom_{\bfk}(X, X)
\end{equation*}
in $\bimod{A}{A}$.

Since the duality functor $(-)^* : \bimod{A}{A} \to \bimod{A}{A}$ is an anti-equivalence of $\bfk$-linear categories, it turns an end into a coend and thus, in particular, we have
\begin{equation*}
  A^* = \int^{X \in \lmod{A}} \Hom_{\bfk}(X, X)^*
\end{equation*}
with the universal dinatural transformation $\rho_X^* : \Hom_{\bfk}(X, X)^* \to A^*$, the linear dual of $\rho_X$. We pick $M \in \lmod{A}$ and apply the functor $(-) \otimes_A M : \bimod{A}{A} \to \lmod{A}$ to the above equation. Since this functor has a right adjoint, it preserves colimits and therefore $A^* \otimes_A M$ is the coend
\begin{equation*}
  A^* \otimes_A M = \int^{X \in \lmod{A}} \Hom_{\bfk}(X, X)^* \otimes_A M
\end{equation*}
in $\lmod{A}$ with the universal dinatural transformation $\rho_X^* \otimes_A \id_M$. Now we discuss the `integrand' of this formula. We first remark:

\begin{lemma}
  For $X, M \in \lmod{A}$, there is an isomorphism
  \begin{equation}
    \label{eq:Nak-A-mod-lem-2}
    \varphi_{X,M} : X^* \otimes_A M \to \Hom_{A}(M, X)^*,
    \quad x^* \otimes_A m \mapsto \Big( f \mapsto \langle x^*, f(m) \rangle \Big)
  \end{equation}
  of vector spaces.
\end{lemma}
\begin{proof}
  The desired isomorphism is obtained as the dual of the composition
  \begin{gather*}
    (X^* \otimes_A M)^*
    = \Hom_{\bfk}(X^* \otimes_A M, \bfk)
    \cong \Hom_{A}(M, \Hom_{\bfk}(X^*, \bfk))
    \cong \Hom_{A}(M, X),
  \end{gather*}
  where the first isomorphism is the tensor-Hom adjunction and the second one is induced by the canonical isomorphism $X^{**} \cong X$ of left $A$-modules.
\end{proof}

\begin{lemma}
  For $X, Y \in \lmod{A}$, there is an isomorphism
  \begin{equation}
    \label{eq:Nak-A-mod-lem-1}
    \varphi'_{X,Y} : Y \otimes_{\bfk} X^* \to \Hom_{\bfk}(Y, X)^*,
    \quad y \otimes_{\bfk} x^* \mapsto \Big( f \mapsto \langle x^*, f(y) \rangle \Big)
  \end{equation}
  of $A$-bimodules, where $\Hom_{\bfk}(Y, X)$ is regarded as an $A$-bimodule by the same way as~\eqref{eq:A-actions-on-Hom}.
\end{lemma}
\begin{proof}
  The bijectivity of this map follows from the previous lemma applied to $A = \bfk$.
  It is straightforward to check that $\varphi'_{X,Y}$ is a homomorphism of $A$-bimodules. 
\end{proof}

Thus we have a chain of natural isomorphisms
\begin{equation*}
  \newcommand{\xarr}[1]{\xrightarrow{\makebox[2em]{$\scriptstyle{#1}$}}}
  \begin{aligned}
    \Hom_{A}(M, X)^* \otimes_{\bfk} Y
    & \xarr{\eqref{eq:Nak-A-mod-lem-2}}
    (X^* \otimes_A M) \otimes_{\bfk} Y \\
    & \xarr{\cong}
    (Y \otimes_{\bfk} X^*) \otimes_A M
    \xarr{\eqref{eq:Nak-A-mod-lem-1}}
    \Hom_{\bfk}(Y, X)^* \otimes_A M
  \end{aligned}
\end{equation*}
for $X, Y, M \in \lmod{A}$, where the second one just changes the order of tensorands. Now the universal dinatural transformation \eqref{eq:Nak-A-mod-univ} is obtained by composing $\rho_X^* \otimes_A \id_M$ and the above chain of isomorphisms with $X = Y$.
It seems to be difficult to write down the map $\mathbb{i}_{X,M}$ explicitly. Instead of doing so, we provide the following useful identity:

\begin{lemma}
  For $x^* \in X^*$, $x \in X$ and $m \in M$, we have
  \begin{equation}
    \label{eq:Nak-A-mod-lem-3}
    \mathbb{i}_{X,M}(\varphi_{X,M}(x^* \otimes_A m) \otimes_{\bfk} x)
    = \langle x^*, ? x \rangle \otimes_A m,
  \end{equation}
  where $\langle x^*, ? x \rangle \in A^*$ is the linear map defined by $a \mapsto \langle x^*, a x \rangle$.
\end{lemma}
\begin{proof}
  By the definition of $\mathbb{i}_{X,M}$, we have
  \begin{equation*}
    \mathbb{i}_{X,M}(\varphi_{X,M}(x^* \otimes_A m) \otimes_{\bfk} x)
    = \rho_X^*(\varphi'_{X,X}(x \otimes_{\bfk} x^*)) \otimes_A m
    = \langle x^*, ? x \rangle \otimes_A m.
  \end{equation*}
\end{proof}

\subsection{The canonical Nakayama-twisted trace}
\label{subsec:cano-Naka-tw-tr}

Let $\mathcal{M}$ be a finite abelian category, and let $\Nak := \Nak_{\mathcal{M}}$. We point out that there is a canonical $\Nak$-twisted trace on $\Proj(\mathcal{M})$, which is `universal' in a different meaning as the universal vector-valued twisted trace mentioned in Subsection~\ref{subsec:tw-tr-HH0}.

We first recall from Subsection~\ref{subsec:cano-vect-act} that there is a canonical isomorphism
\begin{equation}
  \label{eq:cano-Vect-act-Hom-2}
  V \vectactl \Hom_{\mathcal{M}}(M, M')
  \xrightarrow{\ \cong \ } \Hom_{\mathcal{M}}(M, V \vectactl M')
  \quad (V \in \Vect, M, M' \in \mathcal{M})
\end{equation}
as the functor $\Hom_{\mathcal{M}}(M, -)$ is $\bfk$-linear. Let $\mathbb{i}_{X,M} : \Hom_{\mathcal{M}}(M, X)^* \vectactl X \to \Nak(M)$ ($M, X \in \mathcal{M}$) be the universal dinatural transformation for the coend $\Nak(M)$. If $P$ is a projective object of $\mathcal{M}$, then the functor $\Hom_{\mathcal{M}}(P, -)$ preserves colimits and, in particular, we have
\begin{equation*}
  \Hom_{\mathcal{M}}(P, \Nak(P))
  = \int^{X \in \mathcal{M}} \Hom_{\mathcal{M}}(P, \Hom_{\mathcal{M}}(P, X)^* \vectactl X)
\end{equation*}
with the universal dinatural transformation $\Hom_{\mathcal{M}}(P, \mathbb{i}_{X,P})$. Now we define
\begin{equation*}
  \nakac{\trace}_{\bullet} := \{ \nakac{\trace}_{P} : \Hom_{\mathcal{M}}(P, \Nak(P)) \to \bfk \}_{P \in \Proj(\mathcal{M})},
\end{equation*}
where $\nakac{\trace}_P$ for $P \in \Proj(\mathcal{M})$ is the unique $\bfk$-linear map such that the diagram
\begin{equation}
  \label{eq:can-Naka-tw-tr-def}
  \begin{tikzcd}[column sep = 8em]
    \Hom_{\mathcal{M}}(P, \Nak(P))
    \arrow[d, "\nakac{\trace}_P"']
    & \Hom_{\mathcal{M}}(P, \Hom_{\mathcal{M}}(P, X)^* \vectactl X)
    \arrow[d, leftarrow, "\eqref{eq:cano-Vect-act-Hom-2}"]
    \arrow[l, "{\Hom_{\mathcal{M}}(P, \, \mathbb{i}_{X,P})}"'] \\
    \bfk
    & \Hom_{\mathcal{M}}(P, X)^* \otimes_{\bfk} \Hom_{\mathcal{M}}(P, X)
    \arrow[l, "\text{evaluation}"]
  \end{tikzcd}
\end{equation}
commutes. By the naturality of \eqref{eq:cano-Vect-act-Hom-2} and the dinaturality of the evaluation map, we easily verify that $\nakac{\trace}_{\bullet}$ is a $\Nak$-twisted trace on $\Proj(\mathcal{M})$.

\begin{definition}
  We call $\nakac{\trace}_{\bullet}$ the {\em canonical Nakayama-twisted trace} on $\Proj(\mathcal{M})$.
\end{definition}

Let $A$ be a finite-dimensional algebra. We consider the case where $\mathcal{M} = \lmod{A}$.

\begin{lemma}
  \label{lem:cano-Naka-tw-tr-A-mod}
  For every $P \in \lproj{A}$, we have
  \begin{equation}
    \label{eq:cano-Naka-tw-tr-A-mod-1}
    \nakac{\trace}_{P} = \epsilon_P \circ \theta_{P,\Nak(P)}^{-1},
  \end{equation}
  where $\theta$ is the natural transformation given by~\eqref{eq:nat-trans-theta} and $\epsilon_{P}$ is the $\bfk$-linear map
  \begin{equation}
    \label{eq:cano-Naka-tw-tr-A-mod-2}
    \epsilon_P : P^{\dagger} \otimes_A A^* \otimes_A P \to \bfk,
    \quad p^{\dagger} \otimes_A a^* \otimes_A p \mapsto \langle a^*, p^{\dagger}(p) \rangle.
  \end{equation}
\end{lemma}
\begin{proof}
  For $P \in \lproj{A}$, we define $\trace'_P$ to be the right-hand side of \eqref{eq:cano-Naka-tw-tr-A-mod-1}. By the definition \eqref{eq:can-Naka-tw-tr-def} of the canonical Nakayama-twisted trace, this lemma is proved if we prove the equation
  \begin{equation}
    \label{eq:cano-Naka-tw-tr-A-mod-pf-1}
    \mathsf{t}'_P(\mathbb{i}_{X,P} \circ \Lambda(\xi,f)) = \langle \xi, f \rangle
  \end{equation}
  for all $P \in \lproj{A}$, $X \in \lmod{A}$, $\xi \in \Hom_A(P, X)^*$ and $f \in \Hom_A(P,X)$, where
  \begin{equation*}
    \Lambda(\xi, f) : P \to \Hom_A(P,X)^* \otimes_{\bfk} X,
    \quad p' \mapsto \xi \otimes_{\bfk} f(p')
    \quad (p' \in P).
  \end{equation*}
  Let $\varphi_{X,M} : X^* \otimes_A M \to \Hom_A(M, X)^*$ be the isomorphism given by \eqref{eq:Nak-A-mod-lem-2}.
  To prove \eqref{eq:cano-Naka-tw-tr-A-mod-pf-1}, it is enough to consider the case where $\xi = \varphi_{X,P}(x^* \otimes_A p)$ and $f = \theta_{P,X}(p^{\dagger} \otimes_A x)$ for some $x \in X$, $x^* \in X^*$, $p \in P$ and $p^{\dagger} \in P^{\dagger}$.
  The right-hand side of \eqref{eq:cano-Naka-tw-tr-A-mod-pf-1} is then computed as follows:
  \begin{align*}
    \langle \xi, f \rangle
    \mathop{=}^{\eqref{eq:Nak-A-mod-lem-2}} \langle x^*, \theta_{P,X}(p^{\dagger} \otimes_A x)(p) \rangle
    \mathop{=}^{\eqref{eq:nat-trans-theta}} \langle x^*, p^{\dagger}(p) x \rangle.
  \end{align*}  
  We fix a pair $\{ p_i, p^i \}$ of dual bases for $P$. Then the left-hand side of \eqref{eq:cano-Naka-tw-tr-A-mod-pf-1} is
  \begin{gather*}
    \mathsf{t}'_P(\mathbb{i}_{X,P} \circ \Lambda(\xi,f))
    \mathop{=}^{\eqref{eq:nat-trans-theta-inv}}
    \epsilon_P(p^i \otimes_A \mathbb{i}_{X,P}(\xi \otimes_{\bfk} f(p_i)))
    \mathop{=}^{\eqref{eq:Nak-A-mod-lem-3}}
    \epsilon_P(p^i \otimes_A \langle x^*, ? f(p_i) \rangle \otimes_A p) \\
    \mathop{=}^{\eqref{eq:cano-Naka-tw-tr-A-mod-2}}
    \langle x^*, p^i(p) f(p_i) \rangle
    \mathop{=}^{\eqref{eq:nat-trans-theta}}
    \langle x^*, p^i(p) \cdot p^{\dagger}(p_i) x \rangle
    = \langle x^*, p^{\dagger}(p) x \rangle,
  \end{gather*}
  where the summation over $i$ is understood.
  Thus \eqref{eq:cano-Naka-tw-tr-A-mod-pf-1} is verified. The proof is done.
\end{proof}

We fix a finite-dimensional $A$-bimodule $\Sigma$ and regard it as a $\bfk$-linear right exact endofunctor on $\lmod{A}$ as in Subsection~\ref{subsec:tw-tr-HH0}.
A bijection between $\TwTr(\Sigma)$ and $\HH_0(\Sigma)^*$ has been established there.
By using the canonical Nakayama-twisted trace, the bijection is expressed as follows:

\begin{lemma}
  \label{lem:tw-tr-Nakayama-ring-theoretical}
  The $\Sigma$-twisted trace $\trace^{\lambda}_{\bullet}$ associated to $\lambda \in \HH_0(\Sigma)^*$ is given by
  \begin{equation}
    \label{eq:tw-trace-Nakayama-ring-theoretical}
    \trace^{\lambda}_{P}(f)
    = \nakac{\trace}_P\Big(
      P \xrightarrow{\quad f \quad} 
      \Sigma \otimes_A P
      \xrightarrow{\quad \lambda^{\natural} \otimes_A \id_{P} \quad}
      A^* \otimes_A P
    \Big)
  \end{equation}
  for $P \in \lproj{A}$ and $f \in \Hom_{A}(P, \Sigma(P))$, where $\lambda^{\natural} : \Sigma \to A^*$ is given by \eqref{eq:lambda-nat}.
\end{lemma}
\begin{proof}
  Given a morphism $f : P \to \Sigma(P)$ in $\lmod{A}$ with $P \in \lproj{A}$, we denote by $\trace'_P(f)$ the right-hand side of \eqref{eq:tw-trace-Nakayama-ring-theoretical}. We take a pair $\{ p_i, p^i \}$ of dual bases for $P$ and compute
  \begin{gather*}
    \trace'_P(f)
    = \epsilon_P(\theta_{P,\Nak(P)}^{-1}((\lambda^{\natural} \otimes_A \id_P) \circ f))
    \mathop{=}^{\eqref{eq:nat-trans-theta-inv}}
    \epsilon_P(p^i \otimes_A \lambda^{\natural}(f_{\Sigma}(p_i)) \otimes_A f_P(p_i)) \\
    \mathop{=}^{\eqref{eq:cano-Naka-tw-tr-A-mod-2}}
    \langle \lambda^{\natural}(f_{\Sigma}(p_i)), p^i (f_P(p_i)) \rangle
    \mathop{=}^{\eqref{eq:lambda-nat}}
    \langle \lambda, p^i (f_P(p_i)) \cdot f_{\Sigma}(p_i) \rangle
    \mathop{=}^{\eqref{eq:HH0-trace-formula}}
    \trace^{\lambda}_{P}(f),
  \end{gather*}
  where $f(p)$ is expressed as $f(p) = f_{\Sigma}(p) \otimes_A f_{P}(p)$ as in \eqref{eq:HH0-trace-formula}. The proof is done.
\end{proof}

We define the linear map $\lambda : A^* \to \bfk$ by $\lambda(a^*) = a^*(1_A)$.
It is easy to verify that $\lambda$ is the element of $\HH_0(A^*)^*$ such that $\lambda^{\natural} = \id_{A^*}$. The above lemma shows that the canonical Nakayama-twisted trace for $\lmod{A}$ is the $\Nak$-twisted trace associated to $\lambda$ under the bijection given by Theorem~\ref{thm:tw-tr-A-mod-non-deg}. Since $\id_{A^*}$ is an isomorphism, we have the following consequence:

\begin{lemma}
  The canonical Nakayama-twisted trace is non-degenerate.
\end{lemma}

Let $\mathcal{M}$ be a finite abelian category.
The above lemma means that the pairing
\begin{equation}
  \nakac{\beta}_{M, P} : \Hom_{\mathcal{M}}(M, \Nak(P)) \times \Hom_{\mathcal{M}}(P, M) \to \bfk,
  \quad (f, g) \mapsto \nakac{\trace}_P(f g)
\end{equation}
is non-degenerate for all $P \in \Proj(\mathcal{M})$ and $M \in \mathcal{M}$. For later use, we denote by
\begin{equation}
  \nakac{\alpha}_{M,P} : \Hom_{\mathcal{M}}(M, \Nak(P)) \to \Hom_{\mathcal{M}}(P, M)^*,
  \quad f \mapsto \nakac{\beta}_{M,P}(f, -)
\end{equation}
the $\bfk$-linear map induced by $\beta_{M,P}$. By the $\Nak$-cyclicity of $\nakac{\trace}_{\bullet}$, we see that $\nakac{\alpha}_{M,P}$ is natural in $M$ and $P$. Furthermore, $\nakac{\alpha}_{M,P}$ is an isomorphism since $\nakac{\beta}_{M,P}$ is non-degenerate.

\subsection{Category-theoretical reformulation}
\label{subsec:cat-theo-reform}

Let $\mathcal{M}$ be a finite abelian category, and let $\Sigma$ be a $\bfk$-linear right exact endofunctor on $\mathcal{M}$. For simplicity, we write $\mathcal{P} = \Proj(\mathcal{M})$. Given a natural transformation $\xi : \Sigma \to \Nak$, we define
\begin{equation}
  \label{eq:tw-trace-xi-def}
  \trace^{\xi}_{P} : \Hom_{\mathcal{M}}(P, \Sigma(P)) \to \bfk,
  \quad \trace_P^{\xi}(f) := \nakac{\trace}_{P}(\xi_P \circ f)
\end{equation}
for $P \in \mathcal{P}$ and $f \in \Hom_{\mathcal{M}}(P, \Sigma(P))$. By the fact that $\nakac{\trace}_{\bullet}$ is an $\Nak$-twisted trace, one can easily verify that $\trace^{\xi}_{\bullet} := \{ \trace^{\xi}_P \}$ is a $\Sigma$-twisted trace on $\mathcal{P}$. Thus we have the map
\begin{equation}
  \Phi_1 : \Nat(\Sigma, \Nak) \to \TwTr(\Sigma),
  \quad \Phi_1(\xi) = \trace^{\xi}_{\bullet}.
\end{equation}
If $\mathcal{M} = \lmod{A}$ for some finite-dimensional algebra $A$, then, by the Eilenberg-Watts equivalence, we may assume that the functor $\Sigma$ is given by tensoring a finite-dimensional $A$-bimodule. Abusing notation, we denote that bimodule by $\Sigma$ (thus, as in Subsection~\ref{subsec:tw-tr-HH0}, we have $\Sigma(M) = \Sigma \otimes_A M$). By Lemma~\ref{lem:tw-tr-Nakayama-ring-theoretical}, we see that the composition
\begin{equation}
  \HH_0(\Sigma)^*
  \xrightarrow[\approx]{\ \eqref{eq:HH0-and-bimodules} \ }
  \Hom_{A|A}(\Sigma, A^*)
  \xrightarrow[\approx]{\ \eqref{eq:EW-equivalence} \ }
  \Nat(\Sigma, \Nak)
  \xrightarrow{\  \Phi_1 \ }
  \TwTr(\Sigma)
\end{equation}
agrees with the bijection given in Theorem~\ref{thm:tw-tr-A-mod-non-deg}.

The map $\Phi_1$ is described without referencing an algebra $A$ such that $\mathcal{M} \approx \lmod{A}$. Therefore the map $\Phi_1$ is a desired categorical interpretation of the bijection given in Theorem~\ref{thm:tw-tr-A-mod-non-deg}. We also desire to describe the inverse of $\Phi_1$ in a categorical way. For this purpose, we introduce two functors
$\mathbb{h}^*, \mathbb{h}_{\Sigma} : \mathcal{M}^{\op} \times \mathcal{P} \to \Vect$ by
\begin{equation*}
  \mathbb{h}^*(M, P) = \Hom_{\mathcal{M}}(P, M)^*,
  \quad \mathbb{h}_{\Sigma}(M, P) = \Hom_{\mathcal{M}}(M, \Sigma(P))
\end{equation*}
and define two maps $\Phi_2 : \TwTr(\Sigma) \to \Nat(\mathbb{h}_{\Sigma}, \mathbb{h}^*)$ and $\Phi_3 : \Nat(\mathbb{h}_{\Sigma}, \mathbb{h}^*) \to \Nat(\Sigma, \Nak)$ as follows:
\begin{itemize}
\item Given $\trace_{\bullet} \in \TwTr(\Sigma)$, $M \in \mathcal{M}$ and $P \in \mathcal{P}$, we define $\alpha_{M,P} : \mathbb{h}_{\Sigma}(M, P) \to \mathbb{h}^*(M, P)$ by
  \begin{equation*}
    \langle \alpha_{M,P}(f), g \rangle = \trace_{P}(f g)
    \quad (f \in \mathbb{h}_{\Sigma}(M, P), g \in \Hom_{\mathcal{M}}(P, M)).
  \end{equation*}
  It follows from the $\Sigma$-cyclicity of $\trace_{\bullet}$ that $\alpha = \{ \alpha_{M,P} \}_{M \in \mathcal{M}, P \in \mathcal{P}}$ is a natural transformation from $\mathbb{h}_{\Sigma}$ to $\mathbb{h}^*$. We now define $\Phi_2(\trace_{\bullet}) = \alpha$.
\item Given $\alpha \in \Nat(\mathbb{h}_{\Sigma}, \mathbb{h}^*)$, we consider the natural transformation
  \begin{equation*}
    (\nakac{\alpha}_{M,P})^{-1} \circ \alpha_{M,P} : \Hom_{\mathcal{M}}(M, \Sigma(P)) \to \Hom_{\mathcal{M}}(M, \Nak(P))
  \end{equation*}
  for $M \in \mathcal{M}$ and $P \in \mathcal{P}$. By the Yoneda lemma, there is a unique element $\xi \in \Nat(\Sigma|_{\mathcal{P}}, \Nak|_{\mathcal{P}})$ such that the following equation holds:
  \begin{equation*}
    (\nakac{\alpha}_{M,P})^{-1} \circ \alpha_{M,P} = \Hom_{\mathcal{M}}(M, \xi_P)
    \quad (M \in \mathcal{M}, P \in \mathcal{P}).
  \end{equation*}
  By Lemma~\ref{lem:restriction-to-proj}, $\xi : \Sigma|_{\mathcal{P}} \to \Nak|_{\mathcal{P}}$ extends to a natural transformation from $\Sigma$ to $\Nak$, which we denote by the same symbol $\xi$. We now define $\Phi_3(\alpha) = \xi$.
\end{itemize}
Now we state the following main theorem of this section:

\begin{theorem}
  The maps $\Phi_1$, $\Phi_2$ and $\Phi_3$ are bijections such that $\Phi_{3} \Phi_{2} \Phi_{1}$ is the identity map. Furthermore, these bijections restrict to bijections between the following three sets:
  \begin{itemize}
  \item The set of non-degenerate $\Sigma$-twisted traces on $\Proj(\mathcal{M})$.
  \item The set of natural isomorphisms $\Sigma \to \Nak$.
  \item The set of natural isomorphisms $\mathbb{h}_{\Sigma} \to \mathbb{h}^*$.
  \end{itemize}
\end{theorem}
\begin{proof}
  The bijectivity of $\Phi_1$ has been proved by showing that it reduces the bijection of Theorem~\ref{thm:tw-tr-A-mod-non-deg} in the case where $\mathcal{M} = \lmod{A}$ for some finite-dimensional algebra $A$. To see that $\Phi_3$ is bijective, we remark that there are bijections
  \begin{equation*}
    \Nat(\mathbb{h}_{\Sigma}, \mathbb{h}^*)
    \cong \Nat(\mathbb{h}_{\Sigma}, \mathbb{h}_{\Nak})
    \cong \Nat(\Sigma|_{\mathcal{P}}, \Nak|_{\mathcal{P}})
    \cong \Nat(\Sigma, \Nak)
  \end{equation*}
  by the Yoneda lemma and Lemma~\ref{lem:restriction-to-proj}. One can check that the map $\Phi_3$ is actually the composition of these bijections and, in particular, is bijective.

  Now we show that $\Phi_1 \Phi_3 \Phi_2$ is the identity map. Given $\trace_{\bullet} \in \TwTr(\Sigma)$, we set $\alpha = \Phi_2(t_{\bullet})$, $\xi = \Phi_3(\alpha)$ and $\trace'_{\bullet} = \Phi_1(\xi)$. By the definition of the map $\Phi_2$, we have $\nakac{\alpha}_{M,P}(\xi_P \circ f) = \alpha_{M,P}(f)$ for all objects $M \in \mathcal{M}$ and $P \in \mathcal{P}$ and all morphisms $f \in \Hom_{\mathcal{M}}(M, \Sigma(P))$. Hence we compute
  \begin{align*}
    \trace'_P(f)
    = \nakac{\trace}_{P}(\xi_P \circ f)
    = \langle \nakac{\alpha}_{P, P}(\xi_P \circ f), \id_{P} \rangle
    = \langle \alpha_{P, P}(f), \id_{P} \rangle
    = \trace_{P}(f)
  \end{align*}
  for all morphisms $f : P \to \Sigma(P)$ in $\mathcal{M}$ with $P \in \mathcal{P}$. Thus $\trace'_{\bullet} = \trace_{\bullet}$. This means that $\Phi_1 \Phi_3 \Phi_2$ is the identity map.

  Since we have already proved that $\Phi_1$ and $\Phi_3$ are bijective, we obtain $\Phi_2 = \Phi_3^{-1} \Phi_1^{-1}$. This implies the bijectivity of $\Phi_2$ and that $\Phi_3 \Phi_2 \Phi_1$ is the identity map.

  For simplicity, we denote by $\mathfrak{X}_1$, $\mathfrak{X}_2$ and $\mathfrak{X}_3$
  the set of natural isomorphisms from $\Sigma$ to $\Nak$,
  the set of non-degenerate $\Sigma$-twisted traces on $\mathcal{P}$,
  and the set of natural isomorphisms from $\mathbb{h}_{\Sigma}$ to $\mathbb{h}^*$, respectively.
  To complete the proof, we shall show that $\Phi_1$, $\Phi_2$ and $\Phi_3$ induce bijections between $\mathfrak{X}_1$, $\mathfrak{X}_2$ and $\mathfrak{X}_3$.
  To achieve this, it suffices to show
  $\Phi_1(\mathfrak{X}_1) \subset \mathfrak{X}_2$,
  $\Phi_2(\mathfrak{X}_2) \subset \mathfrak{X}_3$ and
  $\Phi_3(\mathfrak{X}_3) \subset \mathfrak{X}_1$.
  The first one follows from the non-degeneracy of the canonical Nakayama-twisted trace, the second from the definition of non-degeneracy, and the third from the Yoneda lemma. The proof is done.
\end{proof}

\section{Compatibility with the module structure}
\label{sec:compati-module}

\subsection{Closing operator and the module-compatibility}

Let $\mathcal{M}$ be a finite abelian category, and let $\mathcal{C}$ be a rigid monoidal category. We say that $\mathcal{C}$ {\em acts linearly} on $\mathcal{M}$ from the right if $\mathcal{M}$ is equipped with a structure of a right $\mathcal{C}$-module category such that the functor $\mathcal{M} \to \mathcal{M}$ given by $M \mapsto M \catactr X$ is $\bfk$-linear for all objects $X \in \mathcal{C}$.
Given finite abelian categories $\mathcal{M}$ and $\mathcal{N}$ on which $\mathcal{C}$ acts linearly, we denote by $\Rex_{\mathcal{C}}(\mathcal{M}, \mathcal{N})$ the category of $\bfk$-linear right exact right $\mathcal{C}$-module functors from $\mathcal{M}$ to $\mathcal{N}$.

Now let $\mathcal{M}$ be a finite abelian category on which $\mathcal{C}$ acts linearly from the right.
Then, for every object $X \in \mathcal{C}$, the endofunctor $(-) \catactr X$ on $\mathcal{M}$ is $\bfk$-linear and exact. The $\bfk$-linearity is a part of our assumption. The exactness follows from that $(-) \catactr X$ has a left and a right adjoint. Indeed, as in the case of rigid monoidal categories, there are natural isomorphisms
\begin{align}
  \label{eq:action-adj-1}
  (-)^{\flat} : \Hom_{\mathcal{M}}(M \catactr X, M')
  & \to \Hom_{\mathcal{M}}(M, M' \catactr X^*), \\
  \label{eq:action-adj-2}
  (-)^{\sharp} : \Hom_{\mathcal{M}}(M, M' \catactr X)
  & \to \Hom_{\mathcal{M}}(M \catactr {}^* \! X, M')
\end{align}
for $M, M' \in \mathcal{M}$ and $X \in \mathcal{C}$ defined by
\begin{equation*}
  f^{\flat} = (f \catactr \id_{X^*}) \circ (\id_M \catactr \coev_X),
  \quad
  g^{\sharp} = (\id_{M'} \catactr \eval_{{}^* \! X}) \circ (g \catactr \id_{{}^* \! X}).
\end{equation*}

We are interested in when the twisted trace is `compatible' with the right action of $\mathcal{C}$ on $\mathcal{M}$. To give a rigorous formulation of the compatibility, we use the linear map
\begin{gather*}
  \close_{M,N | X} : \Hom_{\mathcal{M}}(M \catactr X, N \catactr X^{**}) \to \Hom_{\mathcal{M}}(M, N)
  \quad (M, N \in \mathcal{M}, X \in \mathcal{C}), \\[3pt]
  \close_{M,N|X}(f) = (\id_N \catactr \eval_{X^*}) \circ (f \catactr \id_{X^*}) \circ (\id_M \catactr \coev_{X}).
\end{gather*}
The linear map $\close_{M,N | X}$ is called the partial trace in literature (at least when $\mathcal{M} = \mathcal{C}$ and $\mathcal{C}$ is pivotal), however, we call it the {\em closing operator} to prevent the proliferation of trace-like terms. This terminology is justified by the following graphical expression of the closing operator:
\begin{equation*}
  \close_{M,N|X}\left(
    \begin{array}{c}
      \begin{tikzpicture}[x=1em,y=1em,baseline=0pt,thick]
        \node (B1) at (0, 0) [draw, rectangle] {\makebox[3em]{$f$}};
        \draw let \p1 = ($(B1.north)-(1,0)$)
        in (\p1) to (\x1, +2) node [above] {$M$};
        \draw let \p1 = ($(B1.north)+(1,0)$)
        in (\p1) to (\x1, +2) node [above] {$X$};
        \draw let \p1 = ($(B1.south)-(1,0)$)
        in (\p1) to (\x1, -2) node [below] {$N$};
        \draw let \p1 = ($(B1.south)+(1,0)$)
        in (\p1) to (\x1, -2) node [below] {$\phantom{{}^{**}}X^{**}$};
      \end{tikzpicture}
    \end{array}\!\!\!\!
  \right) =
  \begin{array}{c}
    \begin{tikzpicture}[x=1em,y=1em,baseline=0pt,thick]
      \node (B1) at (0, 0) [draw, rectangle] {\makebox[3em]{$f$}};
      \draw let \p1 = ($(B1.north)-(1,0)$)
      in (\p1) to (\x1, +2) node [above] {$M$};
      \draw let \p1 = ($(B1.south)-(1,0)$)
      in (\p1) to (\x1, -2) node [below] {$N$};
      \draw let \p1 = ($(B1.north)+(1,0)$), \p2 = ($(B1.south)+(1,0)$)
      in (\p1)
      to [out=90, in=90, looseness=2] ($(\p1)+(2.5,0)$) coordinate (Q1)
      to ($(\p2)+(2.5,0)$) coordinate (Q2)
      to [out=-90, in=-90, looseness=2] (\p2);
      \node at ($(Q1)!.5!(Q2)$) [right] {$X^*$};
    \end{tikzpicture}
  \end{array}
  \quad (f \in \Hom_{\mathcal{M}}(M \catactr X, N \catactr X^{**})).
\end{equation*}

We denote by $\mathcal{M}_{\DLD}$ the category $\mathcal{M}$ equipped with the right $\mathcal{C}$-action twisted by the double left dual functor $\DLD(X) = X^{**}$ ($X \in \mathcal{C}$). Let $\Sigma \in \Rex(\mathcal{M}, \mathcal{M}_{\DLD})$. In other words, $\Sigma$ is a $\bfk$-linear right exact endofunctor on $\mathcal{M}$ equipped with a natural isomorphism
\begin{equation*}
  \Psi^{\Sigma}_{M,X} : \Sigma(M) \catactr X^{**} \to \Sigma(M \catactr X)
  \quad (M \in \mathcal{M}, X \in \mathcal{C})
\end{equation*}
such that the equations $\Psi^{\Sigma}_{M,\unitobj} = \id_{\Sigma(M)}$ and $\Psi^{\Sigma}_{M, X \otimes Y} = \Psi^{\Sigma}_{M \catactr X, Y} \circ (\Psi^{\Sigma}_{M, X} \catactr \id_{Y^{**}})$ hold for all objects $M \in \mathcal{M}$ and $X, Y \in \mathcal{C}$. Now we clarify what `compatible' means:

\begin{definition}
  A $\Sigma$-twisted trace $\trace_{\bullet}$ on $\Proj(\mathcal{M})$ is said to be {\em compatible with the right $\mathcal{C}$-module structure} (or {\em module-compatible} for short) if the equation
  \begin{equation}
    \label{eq:tw-tr-module-compatibitlity}
    \trace_{P}(\close_{P, \Sigma(P)|X}(f)) = \trace_{P \catactr X}(\Psi^{\Sigma}_{P,X} \circ f)
  \end{equation}
  holds for all $P \in \Proj(\mathcal{M})$, $X \in \mathcal{C}$ and $f \in \Hom_{\mathcal{M}}(P \catactr X, \Sigma(P) \catactr X^{**})$.
\end{definition}

Graphically, equation~\eqref{eq:tw-tr-module-compatibitlity} is expressed as follows:
\begin{equation*}
  \trace_{P} \left(
    \begin{array}{c}
      \begin{tikzpicture}[x=1em,y=1em,baseline=0pt,thick]
        \node (B1) at (0, 0) [draw, rectangle] {\makebox[3em]{$f$}};
        \draw let \p1 = ($(B1.north)-(1,0)$)
        in (\p1) to (\x1, +2) node [above] {$P$};
        \draw let \p1 = ($(B1.south)-(1,0)$)
        in (\p1) to (\x1, -2) node [below] {$\Sigma(P)$};
        \draw let \p1 = ($(B1.north)+(1,0)$), \p2 = ($(B1.south)+(1,0)$)
        in (\p1)
        to [out=90, in=90, looseness=2] ($(\p1)+(2.5,0)$) coordinate (Q1)
        to ($(\p2)+(2.5,0)$) coordinate (Q2)
        to [out=-90, in=-90, looseness=2] (\p2);
      \end{tikzpicture}
    \end{array}
  \right)
  = \trace_{P \catactr X} \left(
    \begin{array}{c}
      \begin{tikzpicture}[x=1em, y=1em, baseline=0pt, thick]
      \node (B1) at (0, 0) [draw, rectangle] {\makebox[4em]{$\Psi_{P,X}^{\Sigma} \circ f$}};
      \draw let \p1 = ($(B1.north)-(1.25,0)$)
      in (\p1) to (\x1, +2) node [above] {$P$};
      \draw let \p1 = ($(B1.north)+(1.25,0)$)
      in (\p1) to (\x1, +2) node [above] {$X$};
      \draw let \p1 = (B1.south)
      in (\p1) to (\x1, -2) node [below] {$\Sigma(P \catactr X)$};
    \end{tikzpicture}
    \end{array}
  \right)
\end{equation*}

If $P \in \Proj(\mathcal{M})$ and $X \in \mathcal{C}$, then the functor $\Hom_{\mathcal{M}}(P \catactr X, -)$ is exact since there is a natural isomorphism $\Hom_{\mathcal{M}}(P \catactr X, M) \cong \Hom_{\mathcal{M}}(P, M \catactr X^*)$ for $M \in \mathcal{M}$.
Thus $P \catactr X \in \Proj(\mathcal{M})$.
Thanks to this, the symbol $\trace_{P \catactr X}$ in equation \eqref{eq:tw-tr-module-compatibitlity} makes sense.

\begin{remark}
  \label{rem:pivotal-case}
  A {\em pivotal structure} of $\mathcal{C}$ is an isomorphism $\mathfrak{p} : \id_{\mathcal{C}} \to \DLD$ of monoidal functors. We suppose that a pivotal structure $\mathfrak{p}$ of $\mathcal{C}$ is given. Then an object $\Sigma \in \Rex_{\mathcal{C}}(\mathcal{M}, \mathcal{M}_{\DLD})$ is made into a right $\mathcal{C}$-module endofunctor on $\mathcal{M}$ by replacing its structure morphism $\Psi^{\Sigma}$ with
  \begin{equation*}
    \Sigma(M) \catactr X
    \xrightarrow{\ \id_{\Sigma(M)} \catactr \mathfrak{p}_{X} \ }
    \Sigma(M) \catactr X^{**}
    \xrightarrow{\ \Psi^{\Sigma}_{M,X} \ }
    \Sigma(M \catactr X),
  \end{equation*}
  and, in this way, one can identify $\Rex_{\mathcal{C}}(\mathcal{M}, \mathcal{M}_{\DLD})$ with $\Rex_{\mathcal{C}}(\mathcal{M}, \mathcal{M})$.

  Now let $\Sigma$ be an object of the category $\Rex_{\mathcal{C}}(\mathcal{M}, \mathcal{M}_{\DLD}) \cong \Rex_{\mathcal{C}}(\mathcal{M}, \mathcal{M})$. If $\Proj(\mathcal{M})$ is closed under the functor $\Sigma$, then our notion of a module-compatible $\Sigma$-twisted trace on $\Proj(\mathcal{M})$ is precisely a right module trace on $(\Proj(\mathcal{M}), \Sigma)$ in the sense of \cite{2018arXiv180901122F}.
  Note, in general, an object of $\Rex_{\mathcal{C}}(\mathcal{M}, \mathcal{M})$ does not have a structure of a $\DLD$-twisted module functor in absence of a pivotal structure of $\mathcal{C}$.

  If $\mathcal{M} = \mathcal{C}$ and $\Sigma = \id_{\mathcal{C}}$, then our notion further specializes to modified traces as mentioned in Introduction. Applications of results of this section to modified traces will be given in the next section.
\end{remark}

\subsection{The twisted module structure of the Nakayama functor}
\label{subsec:Naka-tw-module}

Let $\mathcal{M}$ and $\mathcal{C}$ be as in the previous subsection, and let $\Sigma \in \Rex_{\mathcal{C}}(\mathcal{M}, \mathcal{M}_{\DLD})$.
Then there is the bijection between $\TwTr(\Sigma)$ and $\Nat(\Sigma, \Nak)$ established in the previous section.
We recall that the Nakayama functor $\Nak := \Nak_{\mathcal{M}}$ belongs to $\Rex_{\mathcal{C}}(\mathcal{M}, \mathcal{M}_{\DLD})$ \cite[Theorem 4.4]{MR4042867}. The goal of this section is to prove that the set of module-compatible $\Sigma$-twisted traces on $\Proj(\mathcal{M})$ is in bijection with the set of morphisms of right $\mathcal{C}$-module functors from $\Sigma$ to $\Nak$ (Theorem~\ref{thm:tw-tr-module-compati}).

As a preparation, we recall from \cite{MR4042867} how the structure morphism
\begin{equation}
  \label{eq:Naka-twisted-mod-str}
  \Psi^{\Nak}_{M,X} : \Nak_{\mathcal{M}}(M) \catactr X^{**} \to \Nak_{\mathcal{M}}(M \catactr X)
  \quad (M \in \mathcal{M}, X \in \mathcal{C})
\end{equation}
of the Nakayama functor is given. The following formula for coends is useful:

\begin{lemma}[{\cite[Lemma 3.9]{MR2869176}}]
  \label{lem:coends-adjunction}
  Let $\mathcal{A}$, $\mathcal{B}$ and $\mathcal{V}$ be categories, let $F : \mathcal{A} \to \mathcal{B}$ be a functor admitting a right adjoint $G$, and let $H : \mathcal{B}^{\op} \times \mathcal{A} \to \mathcal{V}$ be a functor. If either of coends
  \begin{equation*}
    C_1 = \int^{X \in \mathcal{A}} H(F(X), X)
    \quad \text{or} \quad
    C_2 = \int^{Y \in \mathcal{B}} H(Y, G(Y))
  \end{equation*}
  exists, then both exist and they are canonically isomorphic.
\end{lemma}

For reader's convenience, we include how the isomorphism $C_1 \cong C_2$ is constructed.
Let $i$ and $j$ be the universal dinatural transformation for the coend $C_1$ and $C_2$, respectively. We define $\phi : C_1 \to C_2$ and $\overline{\phi} : C_2 \to C_1$ to be the unique morphisms in $\mathcal{V}$ such that the equations
\begin{equation*}
  \phi \circ i_X = j_{F(X)} \circ H(F(X), \eta_X)
  \quad \text{and} \quad
  \overline{\phi} \circ j_X = i_{G(X)} \circ H(\varepsilon_Y, G(Y))
\end{equation*}
hold for all objects $X \in \mathcal{A}$ and $Y \in \mathcal{B}$, where $\eta$ and $\varepsilon$ are the unit and the counit for the adjunction $F \dashv G$, respectively. Then $\phi$ and $\overline{\phi}$ are mutually inverse to each other.

Given a functor $F$, we denote its right adjoint by $F^{\radj}$ (if it exists). Let $F : \mathcal{M} \to \mathcal{M}$ be a $\bfk$-linear right exact functor such that $F^{\radj}$ is right exact and the double right adjoint $F^{\rradj} := (F^{\radj})^{\radj}$ is again right exact. We then have natural isomorphisms
\begin{align}
  \label{eq:Naka-adj-1}
  F^{\rradj} \Nak(M)
  & \textstyle \cong \int^{X \in \mathcal{M}} F^{\rradj}(\Hom_{\mathcal{M}}(M, X)^* \vectactl X) \\
  \label{eq:Naka-adj-2}
  & \textstyle \cong \int^{X \in \mathcal{M}} \Hom_{\mathcal{M}}(M, X)^* \vectactl F^{\rradj}(X) \\
  \label{eq:Naka-adj-3}
  & \textstyle \cong \int^{X \in \mathcal{M}} \Hom_{\mathcal{M}}(M, F^{\radj}(X))^* \vectactl X \\
  \label{eq:Naka-adj-4}
  & \textstyle \cong \int^{X \in \mathcal{M}} \Hom_{\mathcal{M}}(F(M), X)^* \vectactl X
    = \Nak F(M)
\end{align}
for $M \in \mathcal{M}$, where the first isomorphism is due to the assumption that $F^{\rradj}$ is right exact, the second to the $\Vect$-module structure of the $\bfk$-linear functor $F^{\rradj}$, the third to Lemma~\ref{lem:coends-adjunction} and the last to the adjunction $F \dashv F^{\radj}$.
In summary, we have an isomorphism
\begin{equation}
  \label{eq:Naka-dbl-right-adj}
  F^{\rradj} \circ \Nak \cong \Nak \circ F,
\end{equation}
as stated in \cite[Theorem 3.18]{MR4042867}. If $F = (-) \catactr V$ for some $V \in \mathcal{C}$, then we have $F^{\radj} = (-) \catactr V^*$ and $F^{\rradj} = (-) \catactr V^{**}$ in view of \eqref{eq:action-adj-1}. Thus a natural isomorphism
\begin{equation}
  \label{eq:Naka-twisted-mod-str-def}
  \Nak(M) \catactr V^{**}
  = F^{\rradj}(\Nak(M))
  \xrightarrow{\ \eqref{eq:Naka-dbl-right-adj} \ } \Nak(F(M))
  = \Nak(M \catactr V)
  \quad (M \in \mathcal{M})
\end{equation}
is induced. This makes $\Nak: \mathcal{M} \to \mathcal{M}_{\DLD}$ a right $\mathcal{C}$-module functor.

For later use, we remark the following relation between the right $\mathcal{C}$-module structure of $\Nak$ and the closing operator.

\begin{lemma}
  \label{lem:Naka-tw-mod-lem-1}
  For $M, X \in \mathcal{M}$ and $V \in \mathcal{C}$, the following diagram commutes:
  \begin{equation*}
    \begin{tikzcd}[column sep = 8em]
      \Nak(M) \catactr V^{**}
      \arrow[dd, "\Psi^{\Nak}_{M,V}"']
      & (\Hom_{\mathcal{M}}(M, X)^* \vectactl X) \catactr V^{**}
      \arrow[d, "\text{\eqref{eq:cano-Vec-mod-str} with $F = (-) \catactr V^{**}$}"]
      \arrow[l, "\mathbb{i}_{X,M} \catactr \id_{V^{**}}"'] \\
      & \Hom_{\mathcal{M}}(M, X)^* \vectactl (X \catactr V^{**})
      \arrow[d, "(\close_{M, X | V})^* \vectactl \id_{X \catactr V^{**}}"]\\
      \Nak(M \catactr V)
      & \Hom_{\mathcal{M}}(M \catactr V, X \catactr V^{**})^* \vectactl (X \catactr V^{**})
      \arrow[l, "\mathbb{i}_{X \catactr V^{**}, M \catactr V}"]
    \end{tikzcd}
  \end{equation*}
\end{lemma}
\begin{proof}
  We fix $M, X \in \mathcal{M}$ and $V \in \mathcal{C}$ and introduce the $\bfk$-linear map
  \begin{equation*}
    c : \Hom_{\mathcal{M}}(M, X \catactr V^{**} \catactr V^{*}) \to \Hom_{\mathcal{M}}(M, X),
    \quad f \mapsto (\id_{X} \catactr \eval_{V^*}) \circ f.
  \end{equation*}
  By the definition \eqref{eq:Naka-adj-1}--\eqref{eq:Naka-twisted-mod-str-def} of the right $\mathcal{C}$-module structure of the Nakayama functor, the following diagram is commutative:
  \begin{equation*}
    \begin{tikzcd}[column sep = 8em]
      \Nak(M) \catactr V^{**}
      \arrow[ddd, "{\Psi^{\Nak}_{M,V}}"']
      & (\Hom_{\mathcal{M}}(M, X)^* \vectactl X) \catactr V^{**}
      \arrow[d, "\text{\eqref{eq:cano-Vec-mod-str} with $F = (-) \catactr V^{**}$}"]
      \arrow[l, "\mathbb{i}_{X,M} \catactr \id_{V^{**}}"'] \\
      & \Hom_{\mathcal{M}}(M, X)^* \vectactl (X \catactr V^{**})
      \arrow[d, "c^* \vectactl \id_{X \catactr V^{**}}"] \\
      & \Hom_{\mathcal{M}}(M, X \catactr V^{**} \catactr V^{*})^* \vectactl (X \catactr V^{**})
      \arrow[d, "\eqref{eq:action-adj-1}^* \vectactl \id_{X \catactr V^{**}}"] \\
      \Nak(M \catactr V)
      & \Hom_{\mathcal{M}}(M \catactr V, X \catactr V^{**})^* \vectactl (X \catactr V^{**})
      \arrow[l, "\mathbb{i}_{X \catactr V^{**}, M \catactr V}"]
    \end{tikzcd}
  \end{equation*}
  One can verify that the composition
  \begin{equation*}
    \Hom_{\mathcal{M}}(M \catactr V, X \catactr V^{**})
    \xrightarrow{\ \eqref{eq:action-adj-1} \ }
    \Hom_{\mathcal{M}}(M, X \catactr V^{**} \catactr V^{*})
    \xrightarrow{\ c \ } \Hom_{\mathcal{M}}(M, X)
  \end{equation*}
  is equal to $\close_{M,X | V}$. The proof is done.
\end{proof}

\subsection{Module-compatibility of the canonical Nakayama-twisted trace}

As we have recalled, the Nakayama functor $\Nak = \Nak_{\mathcal{M}}$ belongs to the category $\Rex_{\mathcal{C}}(\mathcal{M}, \mathcal{M}_{\DLD})$.
The canonical $\Nak$-twisted trace $\nakac{\trace}_{\bullet}$, introduced in Subsection~\ref{subsec:cano-Naka-tw-tr}, has the following important property:

\begin{lemma}
  \label{lem:Naka-tw-mod-lem-2}
  $\nakac{\trace}_{\bullet}$ is module-compatible.
\end{lemma}

\begin{proof}
  We write the Hom functor of $\mathcal{M}$ as $[\,,\,]$ to save spaces in the proof of this lemma.
  Let $P \in \Proj(\mathcal{M})$ and $V \in \mathcal{C}$.
  Since $P \catactr V$ is projective, we have
  \begin{equation*}
    [P \catactr V, \Nak(P) \catactr V^{**}]
    = \int^{X \in \mathcal{M}} [P \catactr V, ([X, P]^* \vectactl X) \catactr V^{**}]
  \end{equation*}
  with the universal dinatural transformation $\mathbb{j}_X := [P \catactr V, \mathbb{i}_{X,P} \catactr \id_{V^{**}}]$.
  Thus, to prove this lemma, it suffices to show that the equation
  \begin{equation}
    \label{eq:Naka-tw-mod-lem-2-proof-1}
    \nakac{\trace}_{P \catactr V} \circ [P \catactr V, \Psi^{\Nak}_{P,V}] \circ \mathbb{j}_X
    = \nakac{\trace}_P \circ \close_{P,\Nak(P) | V} \circ \mathbb{j}_X
  \end{equation}
  holds for all objects $X \in \mathcal{M}$.

  We first verify that the diagram given as Figure~\ref{fig:proof-Naka-tw-mod-lem-2} is commutative.
  To see the commutativity of the cell labeled (C1), we define three $\bfk$-linear functors $E : \mathcal{M} \to \Vect$, $F : \mathcal{M} \to \mathcal{M}$ and $G: \mathcal{M} \to \Vect$ by
  \begin{gather*}
    E(M) = [P \catactr V, M],
    \quad F(M) = M \catactr V^{**},
    \quad G(M) = [P, M]
  \end{gather*}
  for $M \in \mathcal{M}$, respectively.
  Every $\bfk$-linear functor between finite abelian categories has a canonical structure of a $\Vect$-module functor.
  Lemma~\ref{lem:cano-Vect} implies that the natural transformation $\close_{P, M | V} : E F(M) \to G(M)$ ($M \in \mathcal{M}$) gives rise to a morphism of $\Vect$-module functors. Namely, the diagram
  \begin{equation*}
    \begin{tikzcd}[column sep = 4em]
      E F(W \otimes M)
      \arrow[d, "\close_{P, W \otimes X | V}"']
      \arrow[r, "\eqref{eq:cano-Vec-mod-str}"]
      & E(W \otimes F(M))
      \arrow[r, "\eqref{eq:cano-Vect-act-Hom-2}"]
      & W \otimes E F(M)
      \arrow[d, "\id_W \otimes \close_{P, X | V}"] \\
      G(W \otimes M)
      \arrow[rr, "\eqref{eq:cano-Vect-act-Hom-2}"]
      & & W \otimes G(M)
    \end{tikzcd}
  \end{equation*}
  commutes for all $M \in \mathcal{M}$ and $W \in \Vect$.
  Letting $W = [P, X]^*$ and $M = X$, we find that the cell (C1) is commutative.
  The cells (C2), (C3) and (C4) are also commutative by the naturality of \eqref{eq:cano-Vect-act-Hom-2}, the definition of the dual map, and the commutative diagram \eqref{eq:can-Naka-tw-tr-def} defining $\nakac{\trace}$, respectively.
  Thus the diagram of Figure~\ref{fig:proof-Naka-tw-mod-lem-2} is commutative.

  By Lemma~\ref{lem:Naka-tw-mod-lem-1}, the left row of the diagram is equal to the map
  \begin{equation*}
    [P \catactr V, \Psi^{\Nak}_{P,V} (\mathbb{i}_{X,P} \catactr \id_{V^{**}})]:
    [P \catactr V, ([P, X]^* \otimes X) \catactr V^{**}]
    \to [P \catactr V, \Nak(P \catactr V)].
  \end{equation*}
  By \eqref{eq:can-Naka-tw-tr-def}, the right row is equal to $\nakac{\trace}_{P} \circ [P, \mathbb{i}_{X,P}] : [P, [P, X]^* \otimes X] \to \bfk$. Hence, by the commutativity of the diagram, we have
  \begin{equation*}
    \nakac{\trace}_{P \catactr V} \circ [P \catactr V, \Psi^{\Nak}_{P,V} (\mathbb{i}_{X,P} \catactr \id_{V^{**}})]
    = \nakac{\trace}_{P} \circ [P, \mathbb{i}_{X,P}] \circ \close_{P, [P, X]^* \vectactl X | V}.
  \end{equation*}
  The left-hand side of this equation is equal to that of \eqref{eq:Naka-tw-mod-lem-2-proof-1}. The right-hand side is equal to that of \eqref{eq:Naka-tw-mod-lem-2-proof-1} by the naturality of $\close_{P, - | V}$. The proof is done.
\end{proof}

\begin{figure}
  \centering
  \makebox[\textwidth][c]{
    \begin{tikzcd}[ampersand replacement=\&, row sep = 36pt, column sep = 24pt]
      {[P \catactr V, ([P, X]^* \otimes X) \catactr V^{**}]}
      \arrow[d, "{\scriptstyle \text{\eqref{eq:cano-Vec-mod-str} with $F = (-) \catactr V^{**}$}}"']
      \arrow[rr, "{\close_{P, [P, X]^* \vectactl X | V}}"]
      \arrow[rrd, phantom, "\text{(C1)}"]
      \& \& {[P, [P, X]^* \otimes X]}
      \arrow[d, "\eqref{eq:cano-Vect-act-Hom-2}"]
      \\ 
      {[P \catactr V, [P, X]^* \vectactl (X \catactr V^{**})]}
      \arrow[r, "\eqref{eq:cano-Vect-act-Hom-2}"]
      \arrow[d, "{[P \catactr V, (\close_{P,X | V})^* \vectactl (X \catactr V^{**})]}"']
      \arrow[rd, phantom, "\text{(C2)}"]
      \& {[P, X]^* \otimes_{\bfk} [P \catactr V, X \catactr V^{**}]}
      \arrow[d, "{(\close_{P,X | V})^* \otimes \id}"]
      \arrow[r, "{\id \otimes \close_{P,X | V}}"]
      \arrow[rd, phantom, "\text{(C3)}"]
      \& {[P, X]^* \otimes_{\bfk} [P, X]}
      \arrow[d, "{\eval_{[P,X]}}"]
      \\ 
      {[P \catactr V, [P \catactr V, X \catactr V^{**}]^* \vectactl (X \catactr V^{**})]}
      \arrow[r, "\eqref{eq:cano-Vect-act-Hom-2}"]
      \arrow[d, "{[P \catactr V, \mathbb{i}_{X \catactr V^{**}, P \catactr V}]}"']
      \& {[P \catactr V, X \catactr V^{**}]^* \otimes_{\bfk} [P \catactr V, X \catactr V^{**}]}
      \arrow[r, "{\eval_{[P \catactr V, X \catactr V^{**}]}}" {yshift={2pt}}]
      \& \bfk
      \arrow[d, equal]
      \\ 
      {[P \catactr V, \Nak(P \catactr V)]}
      \arrow[rr, "{\nakac{\trace}_{P \catactr V}}"]
      \arrow[rru, phantom, "\text{(C4)}"]
      \& \& \bfk
    \end{tikzcd}}
  \caption{Proof of Lemma \ref{lem:Naka-tw-mod-lem-2}}
  \label{fig:proof-Naka-tw-mod-lem-2}
\end{figure}

\subsection{Module-compatibility of twisted traces (the general case)}

Let $\Sigma$ be an object of the category $\Rex_{\mathcal{C}}(\mathcal{M}, \mathcal{M}_{\DLD})$. We note that the Nakayama functor $\Nak := \Nak_{\mathcal{M}}$ also belongs to this category.

\begin{definition}
  A natural transformation $\xi : \Sigma \to \Nak$ is said to be {\em compatible with the right $\mathcal{C}$-module structure} (or {\em module-compatible} for short) if it is a morphism in the category $\Rex_{\mathcal{C}}(\mathcal{M}, \mathcal{M}_{\DLD})$.
\end{definition}

We define $\bfk$-linear functors $\mathbb{h}^*, \mathbb{h}_{\Sigma} : \mathcal{M}^{\op} \times \Proj(\mathcal{M}) \to \Vect$ as in Subsection~\ref{subsec:cat-theo-reform}.

\begin{definition}
  A natural transformation $\alpha : \mathbb{h}_{\Sigma} \to \mathbb{h}^*$ is said to be {\em compatible with the right $\mathcal{C}$-module structure} (or {\em module-compatible} for short) if, for all objects $M \in \mathcal{M}$, $P \in \Proj(\mathcal{M})$ and $X \in \mathcal{C}$, the following diagram commutes:
  \begin{equation*}
    \begin{tikzcd}[column sep = 8em, row sep = 1.5em]
      \Hom_{\mathcal{M}}(M \catactr X^*, \Sigma(P))
      \arrow[d, "\eqref{eq:action-adj-1}"']
      \arrow[r, "\alpha_{M \catactr X^*, P}"]
      & \Hom_{\mathcal{M}}(P, M \catactr X^*)^*
      \arrow[dd, "\eqref{eq:action-adj-1}^*"] \\
      \Hom_{\mathcal{M}}(M, \Sigma(P) \catactr X^{**})
      \arrow[d, "{\Hom_{\mathcal{M}}(M, \Psi^{\Sigma}_{P,X})}"'] \\
      \Hom_{\mathcal{M}}(M, \Sigma(P \catactr X))
      \arrow[r, "\alpha_{M, P \catactr X}"]
      & \Hom_{\mathcal{M}}(P \catactr X, M)^*
    \end{tikzcd}
  \end{equation*}
\end{definition}

The commutativity of the above diagram is equivalent to that the equation
\begin{equation}
  \label{eq:mod-compati-def-alpha}
  \langle \alpha_{M, P \catactr X}(\Psi^{\Sigma}_{P,X} \circ f^{\flat}), g \rangle
  = \langle \alpha_{M \catactr X^*, P}(f), g^{\flat} \rangle
\end{equation}
holds for all $f \in \Hom_{\mathcal{M}}(M \catactr X^*, \Sigma(P))$
and $g \in \Hom_{\mathcal{M}}(P \catactr X, M)$.

Now the main result of this section is stated as follows:

\begin{theorem}
  \label{thm:tw-tr-module-compati}
  The three bijections
  \begin{equation*}
    \Nat(\Sigma, \Nak)
    \xrightarrow{\quad \Phi_1 \quad}
    \TwTr(\Sigma)
    \xrightarrow{\quad \Phi_2 \quad}
    \Nat(\mathbb{h}_{\Sigma}, \mathbb{h}^*)
    \xrightarrow{\quad \Phi_3 \quad}
    \Nat(\Sigma, \Nak)
  \end{equation*}
  introduced in Subsection~\ref{subsec:cat-theo-reform} restrict to bijections between the following three sets:
  \begin{itemize}
  \item The set of module-compatible $\Sigma$-twisted traces on $\Proj(\mathcal{M})$.
  \item The set of module-compatible natural transformations $\Sigma \to \Nak$.
  \item The set of module-compatible natural transformations $\mathbb{h}_{\Sigma} \to \mathbb{h}^*$.
  \end{itemize}
\end{theorem}
\begin{proof}
  For simplicity, we denote by $\mathfrak{Y}_i$ ($i = 1, 2, 3$) the set of module-compatible elements belonging to the source of the map $\Phi_i$. To prove this theorem, it suffices to show $\Phi_i(\mathfrak{Y}_i) \subset \mathfrak{Y}_{i+1}$ for each $i = 1, 2, 3$, where $\mathfrak{Y}_{4} = \mathfrak{Y}_1$.

  \smallskip
  \underline{(1) $\Phi_1(\mathfrak{Y}_1) \subset \mathfrak{Y}_{2}$}.
  Let $\xi : \Sigma \to \Nak$ be a natural transformation, and let $\trace^{\xi}_{\bullet}$ be the $\Sigma$-twisted trace corresponding to $\xi$.
  Suppose that $\xi$ is module-compatible.
  Then, for all objects $P \in \Proj(\mathcal{M})$ and $X \in \mathcal{C}$ and all morphisms $f : P \catactr X \to \Sigma(P) \catactr X^{**}$ in $\mathcal{M}$, we have
  \begin{gather*}
    \trace^{\xi}_{P \catactr X}(\Psi^{\Sigma}_{P,X} \circ f)
    \mathop{=}^{\eqref{eq:tw-trace-xi-def}}
    \nakac{\trace}_{P \catactr X}(\xi_{P \catactr X} \circ \Psi^{\Sigma}_{P,X} \circ f)
    = \nakac{\trace}_{P \catactr X}(\Psi^{\Nak}_{P,X} \circ (\xi_P \catactr \id_{X^{**}}) \circ f) \\
    = \nakac{\trace}_{P}(\close_{P,\Nak(P) | X}((\xi_P \catactr \id_{X^{**}}) \circ f))
    = \nakac{\trace}_{P}(\xi_P \circ \close_{P, \Sigma(P) | X}(f))
    \mathop{=}^{\eqref{eq:tw-trace-xi-def}}
    \trace^{\xi}_{P}(\close_{P, \Sigma(P) | X}(f)).
  \end{gather*}
  Here, the second equality follows from the module-compatibility of $\xi$, the third from Lemma~\ref{lem:Naka-tw-mod-lem-2}, and the fourth from the naturality of the closing operator. Hence $\Phi_1(\mathfrak{Y}_1) \subset \mathfrak{Y}_{2}$ is proved.

  \smallskip
  \underline{(2) $\Phi_2(\mathfrak{Y}_2) \subset \mathfrak{Y}_{3}$}.
  Let $\trace_{\bullet} \in \TwTr(\Sigma)$ be a module-compatible $\Sigma$-twisted trace on $\Proj(\mathcal{M})$, and set $\alpha = \Phi_2(\trace_{\bullet})$. We fix objects $M \in \mathcal{M}$, $P \in \Proj(\mathcal{M})$ and $X \in \mathcal{C}$. For all morphisms $f : M \catactr X^* \to \Sigma(P)$ and $g : P \catactr X \to M$ in $\mathcal{M}$, we have
  \begin{equation}
    \label{eq:tw-tr-module-compati-pf-2}
    \close_{P,\Sigma(P) | X}(f^{\flat} \circ g) =
    \begin{array}{c}
      \begin{tikzpicture}[x=1em,y=1em,baseline=0pt,thick]
        \node (B1) at (0, -2.25) [draw, rectangle] {\makebox[2em]{$f$}};
        \draw let \p1 = ($(B1.south)$)
        in (\p1) -- (\x1, -4) node [below] {$\Sigma(P)$};
        \draw let \p1 = ($(B1.north)-(.75,0)$)
        in (\p1) -- ++(0, 3) coordinate (T2);
        \node (B2) at (T2) [above, draw, rectangle] {\makebox[2em]{$g$}};
        \node at (B2.south) [below left] {$M$};
        \draw let \p1 = ($(B2.north)-(.75,0)$)
        in (\p1) -- (\x1, 4) node [above] {$P$};
        \coordinate (X1) at ($(B1.north west)$);
        \coordinate (X2) at ($(B1.south)+(2.25, -.5)$);
        \draw let \p1 = ($(B1.north)+(.75,0)$),
        \p2 = (X2),
        \p3 = ($(B2.north)+(.75, 0)$),
        \p4 = ($(\p3) + (3.5,0)$)
        in (\p1)
        to [out=90, in=90, looseness=2] (\x2, \y1)
        to (\p2)
        to [out=-90, in=-90, looseness=2] (\x4, \y2)
        to (\p4) node [right] {$X^*$}
        to [out=90, in=90, looseness=1.75] (\p3);
        \draw [dotted] let \p1 = ($(X1)+(-.5, 1.1)$), \p2 = ($(X2)+(.5, 0)$)
        in (\p1) -- (\x1, \y2) -- (\p2)
        -- (\x2, \y1) node [above left] {$f^{\flat}$}
        -- cycle;
      \end{tikzpicture}
    \end{array}
    \!\!\!\!=
    \begin{array}{c}
      \begin{tikzpicture}[x=1em,y=1em,baseline=0pt,thick]
        \node (B1) at (0, -2.25) [draw, rectangle] {\makebox[2em]{$f$}};
        \draw let \p1 = ($(B1.south)$)
        in (\p1) -- (\x1, -4) node [below] {$\Sigma(P)$};
        \draw let \p1 = ($(B1.north)-(.75,0)$)
        in (\p1) -- ++(0, 3) coordinate (T2);
        \node (B2) at (T2) [above, draw, rectangle] {\makebox[1.25em]{$g$}};
        \node at (B2.south) [below left] {$M$};
        \draw let \p1 = ($(B2.north)-(.5,0)$)
        in (\p1) -- (\x1, 4) node [above] {$P$};
        \draw let \p1 = ($(B2.north)+(.5,0)$),
        \p2 = ($(B1.north)+(.75,0)$)
        in (\p1) to [out=90, in=90, looseness=2] (\x2, \y1) -- (\p2);
      \end{tikzpicture}
    \end{array}
    = f \circ g^{\flat}.
  \end{equation}
  Now the equation~\eqref{eq:mod-compati-def-alpha} is verified as follows:
  \begin{gather*}
    \langle \alpha_{M, P \catactr X}(\Psi^{\Sigma}_{P,X} \circ f^{\flat}), g \rangle
    = \trace_{P \catactr X}(\Psi^{\Sigma}_{P,X} \circ f^{\flat} \circ g) \\
    \mathop{=}^{\eqref{eq:tw-tr-module-compatibitlity}}
    \trace_{P}(\close_{P,\Sigma(P) | X}(f^{\flat} \circ g))
    \mathop{=}^{\eqref{eq:tw-tr-module-compati-pf-2}}
    \trace_{P}(f \circ g^{\flat}) = \langle \alpha_{M \catactr X^*, P}(f), g^{\flat} \rangle.
  \end{gather*}
  Hence $\Phi_2(\mathfrak{Y}_2) \subset \mathfrak{Y}_{3}$ is proved.

  \smallskip
  \underline{(3) $\Phi_3(\mathfrak{Y}_3) \subset \mathfrak{Y}_{1}$}.
  Let $\alpha : \mathbb{h}_{\Sigma} \to \mathbb{h}^*$ be a module-compatible natural transformation, and let $\xi = \Phi_3(\alpha)$ be the corresponding natural transformation. For all objects $M \in \mathcal{M}$, $P \in \Proj(\mathcal{M})$ and $X \in \mathcal{C}$, there is the commutative diagram given as Figure \ref{fig:tw-tr-module-compati-3}. By the definition of $\xi$ and the naturality of \eqref{eq:action-adj-1}, the composition along the left column is
  \begin{equation*}
    \Hom_{\mathcal{M}}(M, \Sigma(P) \catactr X^{**})
    \to \Hom_{\mathcal{M}}(M, \Nak(P) \catactr X^{**}),
    \quad f \mapsto (\xi_P \catactr \id_{X^{**}}) \circ f.
  \end{equation*}
  By chasing the identity map $\id_{M} \in \Hom_{\mathcal{M}}(M, \Sigma(P) \catactr X^{**})$ with $M = \Sigma(P) \catactr X^{**}$ around the above diagram, we see that the equation
  \begin{equation}
    \label{eq:tw-tr-module-compati-pf-1}
    \xi_{P \catactr X} \circ \Psi^{\Sigma}_{P,X} = \Psi^{\Nak}_{P,X} \circ (\xi_{P} \catactr \id_{X^{**}})
  \end{equation}
  holds for all $P \in \Proj(\mathcal{M})$ and $X \in \mathcal{C}$.

  To show that $\xi$ is module-compatible, we shall show that the equation \eqref{eq:tw-tr-module-compati-pf-1} holds for all objects $P \in \mathcal{M}$ and $X \in \mathcal{C}$. We fix an object $X \in \mathcal{C}$ and introduce two $\bfk$-linear functors
  \begin{equation*}
    F, G : \mathcal{M} \to \mathcal{M},
    \quad F(M) = \Sigma(M) \catactr X^{**},
    \quad G(M) = \Nak(M \catactr X)
    \quad (M \in \mathcal{M}).
  \end{equation*}
  There is a natural transformation
  \begin{equation*}
    z_M : F(M) \to G(M),
    \quad z_M = \xi_{M \catactr X} \circ \Psi^{\Sigma}_{M,X} - \Psi^{\Nak}_{M,X} \circ (\xi_{M} \catactr \id_{X^{**}}).
  \end{equation*}
  Equation~\eqref{eq:tw-tr-module-compati-pf-1} implies that the natural transformation $z|_{\Proj(\mathcal{M})} : F|_{\Proj(\mathcal{M})} \to G|_{\Proj(\mathcal{M})}$ is zero. Since $\Sigma$ and $\Nak$ are right exact, so are $F$ and $G$. Thus, by Lemma~\ref{lem:restriction-to-proj}, $z = 0$. This means that \eqref{eq:tw-tr-module-compati-pf-1} holds for all $P \in \mathcal{M}$ and $X \in \mathcal{C}$. The proof is done.
\end{proof}

\begin{figure}
  \centering
  \begin{equation*}
    \begin{tikzcd}[column sep = 12em, row sep = 2em]
      \Hom_{\mathcal{M}}(M, \Sigma(P) \catactr X^{**})
      \arrow[ddr, phantom, "\scriptstyle \text{(module-compatibility of $\alpha$)}"]
      \arrow[r, "{\Hom_{\mathcal{M}}(M, \Psi^{\Sigma}_{P,X})}"]
      \arrow[d, "\eqref{eq:action-adj-1}"']
      & \Hom_{\mathcal{M}}(M, \Sigma(P \catactr X))
      \arrow[dd, "\alpha_{M, P \catactr X}"] \\
      \Hom_{\mathcal{M}}(M \catactr X^*, \Sigma(P))
      \arrow[d, "\alpha_{M \catactr X^*, P}"']
      \\
      \Hom_{\mathcal{M}}(P, M \catactr X^*)^*
      \arrow[r, "\eqref{eq:action-adj-1}^*"]
      \arrow[ddr, phantom, "\scriptstyle \text{(module-compatibility of $\nakac{\alpha}$)}"]
      & \Hom_{\mathcal{M}}(P \catactr X, M)^*
      \\
      \Hom_{\mathcal{M}}(M \catactr X^*, \Nak(P))
      \arrow[u, leftarrow, "(\nakac{\alpha}_{M \catactr X^*, P})^{-1}"] \\
      \Hom_{\mathcal{M}}(M, \Nak(P) \catactr X^{**})
      \arrow[u, leftarrow, "\text{\eqref{eq:action-adj-1}}"]
      \arrow[r, "{\Hom_{\mathcal{M}}(P, \Psi^{\Nak}_{P,X})}"]
      & \Hom_{\mathcal{M}}(M, \Nak(P \catactr X))
      \arrow[uu, leftarrow, "(\nakac{\alpha}_{M, P \catactr X})^{-1}"']
    \end{tikzcd}
  \end{equation*}
  \caption{Proof of $\Phi_3(\mathfrak{Y}_3) \subset \mathfrak{Y}_{1}$}
  \label{fig:tw-tr-module-compati-3}
\end{figure}

\begin{remark}[compatibility with left module structures]
  \label{rem:left-module-case}
  Let $\mathcal{M}$ be a finite abelian category, and let $\mathcal{C}$ be a rigid monoidal category acting linearly on $\mathcal{M}$ from the {\em left} (that is, $\mathcal{M}$ is equipped with structure of a left $\mathcal{C}$-module category and the functor $X \catactl \id_{\mathcal{M}}$ is $\bfk$-linear for all objects $X \in \mathcal{C}$).
  This assumption is equivalent to that $\mathcal{C}^{\rev}$ acts linearly on $\mathcal{M}$ from the {\em right}. Thus, by applying our results, we can know when twisted traces are compatible with the left $\mathcal{C}$-module structure.

  A technical remark is that the double left dual functor of $\mathcal{C}^{\rev}$ is the double right dual functor $\DRD = {}^{**}(-)$ of $\mathcal{C}$. Thus, in our framework, $\DRD$-twisted left $\mathcal{C}$-module functors are considered. They are functors $\Sigma$ equipped with structure morphism of the form ${}^{**}\!X \catactl \Sigma(M) \to \Sigma(X \catactl M)$.
\end{remark}

\section{Applications to finite tensor categories}
\label{sec:applications-to-ftc}

\subsection{Twisted traces for finite module categories}

In this section, we assume that the base field $\bfk$ is algebraically closed (since we will use results of \cite{MR2119143}, where the base field is assumed to be algebraically closed). A {\em finite tensor category} \cite{MR2119143,MR3242743} is a finite abelian category $\mathcal{C}$ equipped with a structure of a rigid monoidal category such that the tensor product $\mathcal{C} \times \mathcal{C} \to \mathcal{C}$ is $\bfk$-bilinear and the unit object $\unitobj \in \mathcal{C}$ is a simple object. We give some applications of our results to finite tensor categories.

Let $\mathcal{C}$ be a finite tensor category.
An {\em exact} right $\mathcal{C}$-module category \cite{MR2119143,MR3242743} is a finite abelian category $\mathcal{M}$ equipped with a structure of a right $\mathcal{C}$-module category such that the action $\catactr : \mathcal{M} \times \mathcal{C} \to \mathcal{M}$ is $\bfk$-bilinear and $M \catactr P$ is projective for all $M \in \mathcal{M}$ and $P \in \Proj(\mathcal{C})$.
An easy, but important observation is:

\begin{lemma}
  Let $\mathcal{C}$ be a finite tensor category, and let $\mathcal{M}$ be an indecomposable exact right $\mathcal{C}$-module category. Then the Nakayama functor is a simple object in the finite abelian category $\Rex_{\mathcal{C}}(\mathcal{M}, \mathcal{M}_{\DLD})$. 
\end{lemma}
\begin{proof}
  Since $\mathcal{M}$ is assumed be exact, any projective object of $\mathcal{M}$ is injective and vice versa \cite{MR2119143}, and therefore the Nakayama functor $\Nak_{\mathcal{M}}$ is an equivalence \cite[Proposition 3.24]{MR4042867}. Hence the functor
  \begin{equation*}
    \Rex_{\mathcal{C}}(\mathcal{M}, \mathcal{M}) \to \Rex_{\mathcal{C}}(\mathcal{M}, \mathcal{M}_{\DLD}),
    \quad F \mapsto \Nak_{\mathcal{M}} \circ F
  \end{equation*}
  is an equivalence.
  The category $\mathcal{E} := \Rex_{\mathcal{C}}(\mathcal{M}, \mathcal{M})$ is known as the dual of $\mathcal{C}$ with respect to $\mathcal{M}$ \cite{MR3242743}.
  By the exactness and the indecomposability of $\mathcal{M}$, $\mathcal{E}$ is in fact a finite tensor category. The unit object of $\mathcal{E}$, namely $\id_{\mathcal{M}}$, is a simple object of $\mathcal{E}$. Thus so is the corresponding object $\Nak_{\mathcal{M}}$.
\end{proof}

As an application of our results, we prove:

\begin{theorem}
  \label{thm:tw-tr-indec-exact-mod-cat}
  Let $\mathcal{C}$ be a finite tensor category, let $\mathcal{M}$ be an indecomposable exact right $\mathcal{C}$-module category, and let $\Sigma : \mathcal{M} \to \mathcal{M}_{\DLD}$ be a $\bfk$-linear right exact right $\mathcal{C}$-module functor. Then there exists a non-degenerate module-compatible $\Sigma$-twisted trace on $\Proj(\mathcal{M})$ if and only if $\Sigma \cong \Nak_{\mathcal{M}}$ as right $\mathcal{C}$-module functors. If this is the case, a non-zero module-compatible $\Sigma$-twisted trace on $\Proj(\mathcal{M})$ is unique up to scalar multiple and every such trace is non-degenerate.
\end{theorem}
\begin{proof}
  The claim about the existence follows immediately from Theorems~\ref{thm:tw-tr-A-mod-non-deg} and~\ref{thm:tw-tr-module-compati}. Now we suppose that there is an isomorphism $\xi : \Sigma \to \Nak_{\mathcal{M}}$ of right $\mathcal{C}$-module functors. Then we have
  $\Hom_{\mathcal{E}}(\Sigma, \Nak_{\mathcal{M}})
  \cong \Hom_{\mathcal{E}}(\Nak_{\mathcal{M}}, \Nak_{\mathcal{M}})
  \cong \bfk$,
  where $\mathcal{E} := \Rex_{\mathcal{C}}(\mathcal{M},\mathcal{M}_{\DLD})$ and the second isomorphism follows from Schur's lemma. This implies that every non-zero module-compatible $\Sigma$-twisted trace on $\Proj(\mathcal{M})$ is a scalar multiple of the $\Sigma$-twisted trace associated to the isomorphism $\xi$. The proof is done.
\end{proof}

Now we consider the case where $\mathcal{M} = \mathcal{C}$.

\begin{definition}
  \label{def:modular-object}
  For a finite tensor category $\mathcal{C}$, we set $\alpha_{\mathcal{C}} := \Nak_{\mathcal{C}}(\unitobj)$.
  We say that $\mathcal{C}$ is {\em unimodular} if the object $\alpha_{\mathcal{C}}$ is isomorphic to the unit object.
\end{definition}

The distinguished invertible object of $\mathcal{C}$ is introduced in \cite{MR2097289} as an analogue of the modular function on a finite-dimensional Hopf algebra (also called the distinguished grouplike element). By the discussion of \cite[Subsection 4.3]{MR4042867}, the object $\alpha_{\mathcal{C}}$ is isomorphic to the dual of the distinguished invertible object of \cite{MR2097289}.
Thus the above definition of unimodularity agrees with that of \cite{MR2097289}.

\begin{theorem}
  \label{thm:FTC-tw-tr-1}
  We fix an object $\sigma \in \mathcal{C}$,
  define $\Sigma : \mathcal{C} \to \mathcal{C}$ by $\Sigma(X) = \sigma \otimes X^{**}$ for $X \in \mathcal{C}$,
  and make $\Sigma$ an object of the category $\Rex_{\mathcal{C}}(\mathcal{C}, \mathcal{C}_{\DLD})$ in an obvious way.
  Then the space of module-compatible $\Sigma$-twisted traces on $\Proj(\mathcal{C})$ can be identified with the space $\Hom_{\mathcal{C}}(\sigma, \alpha_{\mathcal{C}})$.
  Furthermore, a module-compatible $\Sigma$-twisted traces on $\Proj(\mathcal{C})$ is non-degenerate if and only if the corresponding morphism $\sigma \to \alpha_{\mathcal{C}}$ in $\mathcal{C}$ is an isomorphism.
\end{theorem}
\begin{proof}
  We write $\mathcal{E} := \Rex_{\mathcal{C}}(\mathcal{C}, \mathcal{C}_{\DLD})$ for simplicity. By Theorem~\ref{thm:tw-tr-module-compati}, the space of  module-compatible $\Sigma$-twisted traces on $\Proj(\mathcal{C})$ is identified with the space $\Hom_{\mathcal{E}}(\Sigma, \Nak_{\mathcal{C}})$.
  Since the functor $\mathcal{E} \to \mathcal{C}$ defined by $F \mapsto F(\unitobj)$ is an equivalence, the space is isomorphic to $\Hom_{\mathcal{C}}(\sigma, \alpha_{\mathcal{C}})$.
  Since an element of $\Hom_{\mathcal{E}}(\Sigma, \Nak_{\mathcal{C}})$ is invertible if and only if the corresponding element of $\Hom_{\mathcal{C}}(\sigma, \alpha_{\mathcal{C}})$ is invertible, the proof is now completed by  Theorem~\ref{thm:tw-tr-module-compati}.
\end{proof}

Corollary \ref{cor:FTC-tw-tr-1} below is just the case where $\sigma = \alpha_{\mathcal{C}}$ in the above theorem. We note that $\alpha_{\mathcal{C}}$ is isomorphic to the dual of the socle of the projective cover of $\unitobj$ \cite[Theorem 6.1]{MR2097289}.
The following corollary is closely related to \cite[Corollary 5.6]{2018arXiv180900499G}.

\begin{corollary}
  \label{cor:FTC-tw-tr-1}
  We define $\Sigma : \mathcal{C} \to \mathcal{C}$ by $\Sigma(X) = \alpha_{\mathcal{C}} \otimes X^{**}$ for $X \in \mathcal{C}$ and make $\Sigma$ an object of the category $\Rex_{\mathcal{C}}(\mathcal{C}, \mathcal{C}_{\DLD})$ in an obvious way.
  Then a non-zero module-compatible $\Sigma$-twisted trace on $\Proj(\mathcal{C})$ exists and is unique up to scalar multiple.
  Furthermore, every such trace is non-degenerate.
\end{corollary}

The following corollary is also interesting:

\begin{corollary}
  \label{cor:FTC-tw-tr-2}
  For a finite tensor category $\mathcal{C}$, there is a natural isomorphism
  \begin{equation*}
    \Hom_{\mathcal{C}}(P, M)^*
    \cong \Hom_{\mathcal{C}}(M, \alpha_{\mathcal{C}} \otimes P^{**})
    \quad (M \in \mathcal{C}, P \in \Proj(\mathcal{C})).
  \end{equation*}
\end{corollary}

\subsection{Pivotal case}
\label{subsec:FTC-pivotal-case}

\newcommand{\MT}{\mathcal{T}}
\newcommand{\tracemigi}{\MT^{(r)}}
\newcommand{\tracehidari}{\MT^{(\ell)}}
\newcommand{\traceryogawa}{\MT^{(\ell r)}}

Let $\mathcal{C}$ be a pivotal finite tensor category with pivotal structure $\mathfrak{p}$.
As before, we denote by $\DLD$ the double left dual functor. The double right dual functor will be denoted by $\DRD$.
The identity functor on $\mathcal{C}$ is made into a $\DLD$-twisted right $\mathcal{C}$-module functor, as well as a $\DRD$-twisted left $\mathcal{C}$-module functor, through the pivotal structure $\mathfrak{p}$ (see Remarks~\ref{rem:pivotal-case} and~\ref{rem:left-module-case}).

\begin{definition}
  A {\em left} ({\it resp}.~{\em right}) {\em modified trace} on $\Proj(\mathcal{C})$ is an $\id_{\mathcal{C}}$-twisted trace on $\Proj(\mathcal{C})$ compatible with the left ({\it resp}.~right) $\mathcal{C}$-module structure.  
  We denote by
  $\tracehidari_{\mathcal{C}}$
  and $\tracemigi_{\mathcal{C}}$
  the spaces of left and right modified traces on $\Proj(\mathcal{C})$, respectively.
\end{definition}

The above definition agrees with those considered in \cite{MR2998839,MR2803849,MR3068949,2018arXiv180100321B}.
Now we give some applications of our results to modified traces.
We first establish the following criterion for the existence of non-zero modified traces:

\begin{theorem}
  \label{thm:FTC-mod-tr-1}
  For a pivotal finite tensor category $\mathcal{C}$, the following assertions are equivalent:
  \begin{itemize}
  \item [(1)] $\mathcal{C}$ is unimodular, that is, there is an isomorphism $\alpha_{\mathcal{C}} \cong \unitobj$ in $\mathcal{C}$.
  \item [(2)] There is a non-zero right modified trace on $\Proj(\mathcal{C})$.
  \item [(3)] There is a non-zero left modified trace on $\Proj(\mathcal{C})$.
  \end{itemize}
  If these equivalent conditions are satisfied, then we have
  $\dim_{\bfk} \tracemigi_{\mathcal{C}} = \dim_{\bfk} \tracehidari_{\mathcal{C}} = 1$
  and every non-zero left or right modified trace on $\Proj(\mathcal{C})$ is non-degenerate.
\end{theorem}

This theorem has been proved in \cite[Corollary 5.6]{2018arXiv180900499G} except the implications (2) $\Rightarrow$ (1) and (3) $\Rightarrow$ (1). We give a proof of the whole part of this theorem based on our approach, which is different to \cite{2018arXiv180900499G}.

\begin{proof}
  \underline{(1) $\Rightarrow$ (2)}.
  Suppose that $\mathcal{C}$ is unimodular.
  We fix an isomorphism $f : \unitobj \to \alpha_{\mathcal{C}}$ in $\mathcal{C}$ and then define the natural isomorphism $\xi$ by
  \begin{equation}
    \xi_X = \Big( X \xrightarrow{\quad f \otimes \mathfrak{p}_{X} \quad}
    \alpha_{\mathcal{C}} \otimes X^{**}
    \xrightarrow{\quad \Psi^{\Nak}_{\unitobj,X} \quad}
    \Nak_{\mathcal{C}}(X) \Big)
  \end{equation}
  for $X \in \mathcal{C}$. It is easy to see that $\xi = \{ \xi_X \}$ is in fact an isomorphism of $\DLD$-twisted right $\mathcal{C}$-module functors. Thus, by Theorem~\ref{thm:tw-tr-indec-exact-mod-cat}, there is a non-zero right modified trace on $\Proj(\mathcal{C})$. The uniqueness (up to scalar) and non-degeneracy of non-zero right modified trace on $\Proj(\mathcal{C})$ also follow from that theorem.

  \smallskip
  \underline{(2) $\Rightarrow$ (1)}.
  Suppose that there is a non-zero right modified trace on $\Proj(\mathcal{C})$. Then, again by  Theorem~\ref{thm:tw-tr-indec-exact-mod-cat}, the right $\DLD$-twisted $\mathcal{C}$-module functor $\id_{\mathcal{C}}$ is isomorphic to $\Nak_{\mathcal{C}}$. In particular, we have
  $\alpha_{\mathcal{C}} = \Nak_{\mathcal{C}}(\unitobj) \cong \id_{\mathcal{C}}(\unitobj) = \unitobj$
  and thus $\mathcal{C}$ is unimodular.

  \smallskip
  \underline{(1) $\Leftrightarrow$ (3)}. Apply the proof of (1) $\Leftrightarrow$ (2) to $\mathcal{C}^{\rev}$.
\end{proof}

\subsection{Existence of two-sided modified traces}

Let $\mathcal{C}$ be a pivotal finite tensor category.
A {\em two-sided modified trace} on $\Proj(\mathcal{C})$ is an element of $\traceryogawa_{\mathcal{C}} := \tracehidari_{\mathcal{C}} \cap \tracemigi_{\mathcal{C}}$. Theorem~\ref{thm:FTC-mod-tr-1} implies that a non-zero two-side modified trace on $\Proj(\mathcal{C})$ exists only if $\mathcal{C}$ is unimodular. The converse does not hold in general. For the existence of such a trace, we shall impose an extra condition for the pivotal structure.

\begin{definition}
  Let $\mathcal{C}$ be a finite tensor category. We denote by
  \begin{equation*}
    \Phi^{\Nak}_{X,Y} : {}^{**}X \otimes \Nak(Y) \to \Nak(X \otimes Y)
    \quad \text{and} \quad
    \Psi^{\Nak}_{X,Y} : \Nak(X) \otimes Y^{**} \to \Nak(X \otimes Y)
  \end{equation*}
  the left and the right twisted module structures of the Nakayama functor $\Nak = \Nak_{\mathcal{C}}$, respectively. The {\em Radford isomorphism} ({\it cf}. \cite{MR2097289} and \cite[Corollary 4.12]{MR4042867}) for $\mathcal{C}$ is defined to be the composition
  \begin{equation*}
    g_X := \Big(
    {}^{**}X \otimes \alpha_{\mathcal{C}}
    \xrightarrow{\quad \Phi^{\Nak}_{X,\unitobj} \quad}
    \Nak(X)
    \xrightarrow{\quad (\Psi^{\Nak}_{\unitobj,X})^{-1} \quad}
    \alpha_{\mathcal{C}} \otimes X^{**} \Big)
  \end{equation*}
  for $X \in \mathcal{C}$. Now we suppose that $\mathcal{C}$ is a pivotal finite tensor category with pivotal structure $\mathfrak{p}$. We say that $\mathcal{C}$ is {\em spherical} \cite[Definition 3.5.2]{2013arXiv1312.7188D} if $\mathcal{C}$ is unimodular and the diagram
  \begin{equation}
    \label{eq:def-sphericity}
    \begin{tikzcd}[column sep = 5em]
      {}^{**} X \otimes \unitobj
      \arrow[d, "\id_{{}^{**}\!X} \otimes f"']
      \arrow[r, "\mathfrak{p}_{\, {}^{**} \! X}"]
      & X \arrow[r, "\mathfrak{p}_X"]
      & \unitobj \otimes X^{**}
      \arrow[d, "f \otimes \id_{X^{**}}"] \\
      {}^{**} X \otimes \alpha_{\mathcal{C}}
      \arrow[rr, "g_X"]
      & & \alpha_{\mathcal{C}} \otimes X^{**}
    \end{tikzcd}
  \end{equation}
  commutes for all objects $X \in \mathcal{C}$, where $f : \unitobj \to \alpha_{\mathcal{C}}$ is an arbitrary isomorphism in $\mathcal{C}$ (since $\unitobj \in \mathcal{C}$ is a simple object, such an isomorphism $f$ is unique up to scalar and therefore this definition does not depend on the choice of $f$).
\end{definition}

\begin{remark}
  Let $\mathcal{C}$ be a finite tensor category, and set $D = \alpha_{\mathcal{C}}^*$ (this is the distinguished invertible object of \cite{MR2097289}).
  The {\em left exact Nakayama functor} of $\mathcal{C}$ is defined by
  \begin{equation*}
    \Nak^{\ell}_{\mathcal{C}}: \mathcal{C} \to \mathcal{C}, \quad
    V \mapsto \int_{X \in \mathcal{C}} \Hom_{\mathcal{C}}(X, V) \vectactl X
    \quad (V \in \mathcal{C})
  \end{equation*}
  \cite[Definition 3.14]{MR4042867}.
  A relationship between $\Nak^{\ell}_{\mathcal{C}}$ and a result of \cite{MR2097289} has been pointed out in \cite{2017arXiv170709691S}. Noting that $\Nak^{\ell}_{\mathcal{C}}$ is the inverse of the (right exact) Nakayama functor $\Nak_{\mathcal{C}}$ used in this paper, we see from \cite[Remark 4.11]{2017arXiv170709691S} that the natural isomorphism
  \begin{equation*}
    X^{**} \xrightarrow{\  \coev_D \otimes \id_{X^{**}} \ } 
    D \otimes D^* \otimes X^{**}
    \xrightarrow{\  \id_D \otimes g_X^{-1} \ }
    D \otimes {}^{**}\!X \otimes D^*
    \quad (X \in \mathcal{C})
  \end{equation*}
  coincides with that given in \cite[Theorem 3.3]{MR2097289}.
\end{remark}

Now we prove the following criterion for the existence of two-sided modified traces:

\begin{corollary}
  \label{cor:FTC-mod-tr-2}
  Let $\mathcal{C}$ be a pivotal finite tensor category. Then there exists a non-zero two-sided modified trace on $\Proj(\mathcal{C})$ if and only if $\mathcal{C}$ is spherical. If this is the case, a non-zero two-sided modified trace on $\Proj(\mathcal{C})$ is unique up to scalar multiple and every such trace is non-degenerate.
\end{corollary}

If $\mathcal{C} = \lmod{H}$ for some finite-dimensional Hopf algebra $H$, then the Radford isomorphism corresponds to the distinguished grouplike element of $H$. Thus, in the case where $H$ is pivotal, $\lmod{H}$ is spherical if and only if $H$ is unimodular and unibalanced in the sense of \cite{2018arXiv180100321B}. This corollary can be thought of a generalization of \cite[Theorem 1]{2018arXiv180100321B} to the setting of finite tensor categories.

\begin{proof}
  We make $\Sigma := \id_{\mathcal{C}}$ a $\mathcal{C}$-bimodule functor from $\mathcal{C}$ to ${}_{\DRD}\mathcal{C}_{\DLD}^{}$ by the structure morphisms
  \begin{gather*}
    \Phi_{X,Y}^{\Sigma} : {}^{**}\!X \otimes \Sigma(Y)
    \xrightarrow{\ \mathfrak{p}_{{}^{**} \! X} \otimes \id_Y \ } \Sigma(X \otimes Y), \\
    \Psi_{X,Y}^{\Sigma} : \Sigma(X) \otimes Y^{**}
    \xrightarrow{\ \id_X \otimes \mathfrak{p}_Y^{-1} \ } \Sigma(X \otimes Y).
  \end{gather*}
  We also define $\mathfrak{N}(X) = \alpha_{\mathcal{C}} \otimes X^{**}$ for $X \in \mathcal{C}$. There is an isomorphism
  \begin{equation*}
    \mathfrak{N}(X)
    = \Nak_{\mathcal{C}}(\unitobj) \otimes X^{**}
    \xrightarrow{\quad \Psi^{\Nak}_{\unitobj, X} \quad}
    \Nak_{\mathcal{C}}(X)
    \quad (X \in \mathcal{C})
  \end{equation*}
  of right $\mathcal{C}$-module functors from $\mathcal{C}$ to $\mathcal{C}_{\DLD}$.
  We make $\mathfrak{N}$ a $\mathcal{C}$-bimodule functor from $\mathcal{C}$ to ${}_{\DRD}\mathcal{C}_{\DLD}^{}$ in such a way that the above isomorphism $\mathfrak{N} \cong \Nak_{\mathcal{C}}$ is in fact an isomorphism of $\mathcal{C}$-bimodule functors. By the definition of the Radford isomorphism, we see that the resulting structure morphism of $\mathfrak{N}$ as a twisted left $\mathcal{C}$-module functor is given by
  \begin{equation*}
    \Phi^{\mathfrak{N}}_{X,Y} : {}^{**}\!X \otimes \mathfrak{N}(Y)
    \xrightarrow{\quad g_X \otimes \id_{Y^{**}} \quad}
    \mathfrak{N}(X \otimes Y)
  \end{equation*}
  for $X, Y \in \mathcal{C}$.

  We first prove the `if' part.
  Suppose that $\mathcal{C}$ is spherical. Then, by the definition of the sphericity, there is an isomorphism $f : \unitobj \to \alpha_{\mathcal{C}}$ in $\mathcal{C}$. For $X \in \mathcal{C}$, we set $\xi_X := f \otimes \mathfrak{p}_X$.
  Then $\xi = \{ \xi_X \}$ is an isomorphism of right $\mathcal{C}$-module functors from $\Sigma$ to $\mathfrak{N}$.
  By the sphericity, we have
  \begin{gather*}
    \xi_{X \otimes Y} \circ \Phi^{\Sigma}_{X,Y}
    = (f \otimes \mathfrak{p}_{X \otimes Y}) \circ (\mathfrak{p}_{{}^{**}X} \otimes \id_Y)
    = (f \otimes \mathfrak{p}_X \otimes \mathfrak{p}_Y) \circ (\mathfrak{p}_{{}^{**}X} \otimes \id_Y) \\
    \mathop{=}^{\eqref{eq:def-sphericity}}
    g_X \otimes f \otimes \mathfrak{p}_Y
    = \Phi_{X,Y}^{\mathfrak{N}} \circ (\id_{{}^{**}\!X} \otimes \xi_Y)
  \end{gather*}
  for all $X, Y \in \mathcal{C}$. Thus $\xi : \Sigma \to \mathfrak{N}$ is also an isomorphism of left $\mathcal{C}$-module functors. Hence the $\Sigma$-twisted trace $\trace^{\xi}_{\bullet}$ associated to $\xi$ is a non-degenerate two-sided modified trace on $\Proj(\mathcal{C})$. By Theorem~\ref{thm:FTC-mod-tr-1}, every non-zero two-sided modified trace on $\Proj(\mathcal{C})$ is a scalar multiple of $\trace^{\xi}_{\bullet}$ and thus non-degenerate.

  We prove the converse. Suppose that there is a non-zero two-sided modified trace $\trace_{\bullet}$ on $\Proj(\mathcal{C})$. Then $\trace_{\bullet}$ is non-degenerate by Theorem~\ref{thm:FTC-mod-tr-1} and therefore there is an isomorphism $\xi : \Sigma \to \mathfrak{N}$ of $\mathcal{C}$-bimodule functors. Since, in particular, $\xi$ is an isomorphism of right $\mathcal{C}$-module functors, $f := \xi_{\unitobj}$ is an isomorphism in $\mathcal{C}$ and the equation $\xi_X = f \otimes \mathfrak{p}_X$ holds for all objects $X \in \mathcal{C}$. Now the sphericity follows from that $\mathcal{C}$ is a morphism of left $\mathcal{C}$-module functors. The proof is done.
\end{proof}

\subsection{Ribbon finite tensor categories}

We give some applications of our results to ribbon finite tensor categories.
Let $\mathcal{B}$ be a braided rigid monoidal category with braiding $\sigma$.
The Drinfeld isomorphism ({\it cf}. \cite[Section 5]{MR2097289}) of $\mathcal{B}$ is the natural transformation $u: \id_{\mathcal{B}} \to \DLD$ defined by
\begin{align*}
  u_X
  & = (\eval_X \otimes \id_{X^{**}})
    \circ (\sigma_{X,X^*} \otimes \id_{X^{**}})
    \circ (\id_X \otimes \coev_{X^*})
\end{align*}
for $X \in \mathcal{B}$. The map $\mathfrak{p} \mapsto u^{-1} \circ \mathfrak{p}$ gives a bijection between the set of pivotal structures of $\mathcal{B}$ and the set of natural isomorphisms $\theta : \id_{\mathcal{B}} \to \id_{\mathcal{B}}$ satisfying
\begin{equation}
  \label{eq:def-ribbon-1}
  \theta_{X \otimes Y} = \sigma_{Y,X}\sigma_{X,Y} (\theta_X \otimes \theta_Y)
  \quad (X, Y \in \mathcal{B}).
\end{equation}
A ribbon category \cite{MR1321145} is a braided rigid monoidal category $\mathcal{B}$ equipped with a twist, that is, a natural isomorphism $\theta : \id_{\mathcal{B}} \to \id_{\mathcal{B}}$ satisfying \eqref{eq:def-ribbon-1} and
\begin{equation}
  \label{eq:def-ribbon-2}
  \quad (\theta_{X})^* = \theta_{X^*} \quad (X \in \mathcal{B}).
\end{equation}
Equivalently, a ribbon category is a braided pivotal monoidal category with pivotal structure $\mathfrak{p}$ such that the natural isomorphism $\theta := u^{-1} \circ \mathfrak{p}$ satisfies \eqref{eq:def-ribbon-2}. Such a pivotal structure is called a {\em ribbon pivotal structure} in \cite{2017arXiv170709691S}.

We are interested in when a ribbon finite tensor category admits a non-degenerate two-sided modified trace. Suppose that $\mathcal{B}$ is a braided pivotal unimodular finite tensor category.
By \cite[Lemma 5.9]{2017arXiv170709691S}, $\mathcal{B}$ is a ribbon finite tensor category if and only if $\mathcal{B}$ is spherical. This fact and the previous corollary yield the following consequence:

\begin{corollary}
  \label{cor:ribbon-two-sided-trace}
  Let $\mathcal{C}$ be a ribbon finite tensor category. Then there exists a non-degenerate two-sided modified trace on $\Proj(\mathcal{C})$ if and only if $\mathcal{C}$ is unimodular.
\end{corollary}

\begin{remark}
  To prove the above corollary, we have used the fact that a unimodular ribbon finite tensor category is spherical \cite[Lemma 5.9]{2017arXiv170709691S}. One can give an alternative proof of this fact from the viewpoint of modified trace theory as follows:
  Let $\mathcal{C}$ be a unimodular ribbon finite tensor category.
  Then a non-degenerate right modified trace $\trace_{\bullet}$ on $\Proj(\mathcal{C})$ exists by Theorem~\ref{thm:FTC-mod-tr-1}.
  One can prove that $\trace_{\bullet}$ is two-sided by the graphical technique in ribbon categories as demonstrated in \cite[Theorem 3.3.1]{MR2803849} (note, however, the definition of an ambidextrous trace \cite[Definition 3.2.3]{MR2803849} is slightly different from that of our two-sided modified trace). Thus, by Corollary~\ref{cor:FTC-mod-tr-2}, $\mathcal{C}$ is spherical.
\end{remark}

A {\em modular tensor category} \cite{MR1862634,MR3996323} is a ribbon finite tensor category that is {\em factorisable} in the sense of \cite[Definition 8.6.2]{MR3242743}. Gainutdinov and Runkel proved that a modular tensor category admits a non-degenerate two-sided modified trace under a technical assumption called Condition P \cite[Corollary 4.7]{MR4079628}. Slightly generalizing their result, we have:

\begin{corollary}
  A modular tensor category admits a non-degenerate two-sided modified trace.
\end{corollary}
\begin{proof}
  This follows Corollary~\ref{cor:ribbon-two-sided-trace} and the fact that a factorisable braided finite tensor category is unimodular \cite[Proposition 8.10.10]{MR3242743}.
\end{proof}

\section{Hopf algebras and comodule algebras}
\label{sec:comod-algebras}

\subsection{Module categories and comodule algebras}

Throughout this section, we work over a general field $\bfk$.
Let $H$ be a finite-dimensional Hopf algebra with comultiplication $\Delta$, counit $\varepsilon$ and antipode $S$. The Sweedler notation, such as $\Delta(h) = h_{(1)} \otimes h_{(2)}$, will be used to express the comultiplication. Given a right $H$-comodule $M$, we denote the coaction by $\delta_M : M \to M \otimes_{\bfk} H$ and express it as $\delta_M(m) = m_{(0)} \otimes m_{(1)}$.

A right $H$-comodule algebra is an algebra $A$ equipped with a right $H$-comodule structure $\delta_A$ such that the equations $\delta_A(1_A) = 1_A \otimes 1_H$ and $\delta_A(a b) = a_{(0)} b_{(0)} \otimes a_{(1)} b_{(1)}$ hold for all $a, b \in A$. If $A$ is a right $H$-comodule algebra, then $\lmod{A}$ is a finite right module category over the finite tensor category $\mathcal{C} := \lmod{H}$ by $M \catactr X = M \otimes_{\bfk} X$ ($M \in \lmod{A}$, $X \in \mathcal{C}$), where $A$ acts on $M \catactr X$ from the left by
\begin{equation*}
  a \cdot (m \otimes x) = a_{(0)} m \otimes a_{(1)} x
  \quad (a \in A, m \in M, x \in X).
\end{equation*}
It is known that every finite right $\mathcal{C}$-module category is equivalent to $\lmod{A}$ for some finite-dimensional right $H$-comodule algebra $A$ \cite{MR2331768}. The aim of this section is to describe and compute module-compatible twisted traces for such right $\mathcal{C}$-module categories.

From now on, $H$-comodule algebras are always assumed to be finite-dimensional. Let $A$ be a right $H$-comodule algebra. We pick a finite-dimensional $A$-bimodule $\Sigma$ and identify it with a $\bfk$-linear right exact endofunctor on $\lmod{A}$. In our framework, the functor $\Sigma$ is required to be a `twisted' right $\mathcal{C}$-module functor in order that the module-compatibility of a $\Sigma$-twisted trace makes sense. Thus we first discuss when the functor $\Sigma$ has such a structure.

Let $A$ and $B$ be right $H$-comodule algebras. We define the category $(\bimod{B}{A})^H$ of finite-dimensional $H$-equivariant $B$-$A$-bimodules ({\it cf}. \cite[Definition 1.22]{MR2331768}) as follows: An object of this category is a finite-dimensional $B$-$A$-bimodule $M$ equipped with a right $H$-coaction $\delta_M$ such that the equation
\begin{equation}
  \label{eq:H-eq-bimod-def}
  \delta_{M}(b m a) = b_{(0)} m_{(0)} a_{(0)} \otimes b_{(1)} m_{(1)} a_{(1)}
\end{equation}
holds for all $m \in M$, $a \in A$ and $b \in B$. A morphism of this category is a left $B$-linear, right $A$-linear and right $H$-colinear map.

We recall that an object $T \in \bimod{B}{A}$ defines a $\bfk$-linear right exact functor $T \otimes_{A}(-)$ from $\lmod{A}$ to $\lmod{B}$. If $T$ belongs to the category $(\bimod{B}{A})^H$, then the functor $T \otimes_A (-)$ is an oplax right $\mathcal{C}$-module functor by the structure map given by
\begin{equation}
  \label{eq:H-eq-bimod-oplax}
  \begin{aligned}
    T \otimes_A (M \catactr X) & \to (T \otimes_A M) \catactr X, \\
    t \otimes_A (m \otimes x) & \mapsto (t_{(0)} \otimes_A m) \otimes t_{(1)} x
  \end{aligned}
\end{equation}
for $m \in M \in \lmod{A}$, $x \in X \in \mathcal{C}$ and $t \in T$. Since $\mathcal{C}$ is rigid, the oplax right $\mathcal{C}$-module structure \eqref{eq:H-eq-bimod-oplax} is invertible whose inverse is given by
\begin{equation}
  \label{eq:H-eq-bimod-lax}
  \begin{aligned}
    (T \otimes_A M) \catactr X
    & \to T \otimes_A (M \catactr X), \\
    (t \otimes_A m) \otimes x
    & \mapsto t_{(0)} \otimes_A (m \otimes S(t_{(1)})x).
  \end{aligned}
\end{equation}
The Eilenberg-Watts equivalence \eqref{eq:EW-equivalence} induces an equivalence
\begin{equation}
  \label{eq:H-eq-EW}
  (\bimod{B}{A})^H \to \Rex_{\mathcal{C}}(\lmod{A}, \lmod{B}),
  \quad T \mapsto T \otimes_A (-)
\end{equation}
of $\bfk$-linear categories \cite{MR2331768}, which we call {\em the equivariant Eilenberg-Watts equivalence}.

What we actually want to know is the category $\Rex_{\mathcal{C}}(\lmod{A}, (\lmod{A})_{\DLD})$. To describe this category, we introduce the right $H$-comodule algebra $A^{(S^2)}$ as follows: As an algebra, $A^{(S^2)} = A$. The right $H$-coaction is twisted by the square of the antipode, as $\delta_{A^{(S^2)}} = (\id_A \otimes S^2) \circ \delta_A$. Now we prove:

\begin{lemma}
  The identity functor on $\lmod{A}$ induces an isomorphism $\lmod{A^{(S^2)}} \cong (\lmod{A})_{\DLD}$ of right $\mathcal{C}$-module categories.
\end{lemma}
\begin{proof}
  Let $F : \lmod{A^{(S^2)}} \to (\lmod{A})_{\DLD}$ be the identity functor. For a finite-dimensional vector space $V$, we denote by $\phi_V : V \to V^{**}$ the canonical isomorphism of vector spaces defined by $\phi_V(v) = \langle -, v\rangle$ for $v \in V$. If $X \in \mathcal{C}$, then we have
  \begin{equation}
    h \cdot \phi_X(x) = \phi_X(S^2(h) x)
    \quad (h \in H, x \in X)
  \end{equation}
  by the definition of a left dual $H$-module. Noting this equation, one can verify that $F$ is a right $\mathcal{C}$-module functor by the structure morphism defined by
  \begin{equation*}
    F(M) \catactr X^{**} \to F(M \catactr X),
    \quad m \otimes \phi_X(x) \mapsto m \otimes x
  \end{equation*}
  for $m \in M \in \lmod{A^{(S^2)}}$ and $x \in X \in \mathcal{C}$. The proof is done.
\end{proof}

By this lemma and the equivariant Eilenberg-Watts equivalence~\eqref{eq:H-eq-EW}, we obtain:

\begin{lemma}
  \label{lem:H-eq-EW-2}
  $\Rex(\lmod{A}, (\lmod{A})_{\DLD}) \approx (\bimod{A^{(S^2)}}{A})^H$.
\end{lemma}

Spelling out, an object of the category $(\bimod{A^{(S^2)}}{A})^H$ is a finite-dimensional $A$-bimodule $\Sigma$ equipped with a right $H$-comodule structure $\delta_{\Sigma}$ satisfying
\begin{equation}
  \label{eq:H-eq-bimod-twisted-1}
  \delta_{\Sigma}(a s a') = a_{(0)} s_{(0)} a'_{(0)} \otimes S^2(a_{(1)}) s_{(1)} a'_{(2)}
  \quad (a, a' \in A, s \in \Sigma).
\end{equation}
The equivalence of the lemma sends an object $\Sigma \in (\bimod{A^{(S^2)}}{A})^H$ to the functor $\Sigma \otimes_A (-)$, which is an oplax right $\mathcal{C}$-module functor from $\lmod{A}$ to $(\lmod{A})_{\DLD}$ by the structure morphism given by
\begin{equation}
  \label{eq:H-eq-bimod-twisted-2}
  \begin{aligned}
    \Sigma \otimes_A (M \catactr X)
    & \to (\Sigma \otimes_A M) \catactr X^{**}, \\
    s \otimes_A (m \otimes x)
    & \mapsto (s_{(0)} \otimes_A m) \otimes \phi_X(s_{(1)} x)
  \end{aligned}
\end{equation}
for $m \in M \in \lmod{A}$, $x \in X \in \mathcal{C}$ and $s \in \Sigma$.

\subsection{The Nakayama functor as an equivariant bimodule}

Let $A$ be a right $H$-comodule algebra. Then the Nakayama functor $\Nak := \Nak_{\lmod{A}}$ is given by tensoring the bimodule $A^*$. Since $\Nak$ is a right $\mathcal{C}$-module functor from $\lmod{A}$ to $(\lmod{A})_{\DLD}$, the bimodule $A^*$ should have a right $H$-comodule structure making it an object of $(\bimod{A^{(S^2)}}{A})^H$ in view of the equivalence given by Lemma~\ref{lem:H-eq-EW-2}.

There is a right $H$-comodule structure of $A^*$ given as follows:
Let $\{ a_i \}$ be a basis of $A$ (as a vector space), and let $\{ a^i \}$ be the basis of $A^*$ dual to $\{ a_i \}$. We define the linear map $\delta_{A^*} : A^* \to A^* \otimes_{\bfk} H$ by
\begin{equation}
  \label{eq:Naka-bimod-coaction-1}
  \delta_{A^*}(a^*) = \langle a^*, a_{i(0)} \rangle a^i \otimes S(a_{i(1)})
  \quad (a^* \in A^*),
\end{equation}
where the summation over $i$ is understood. The element $a^*_{(0)} \otimes a^*_{(1)} = \delta_{A^*}(a^*)$ is characterized by
\begin{equation}
  \label{eq:Naka-bimod-coaction-2}
  \langle a^*_{(0)}, a \rangle a^*_{(1)} = \langle a^*, a_{(0)} \rangle S(a_{(1)})
  \quad (a \in A).
\end{equation}
One can verify that $\delta_{A^*}$ makes the $A$-bimodule $A^*$ an object of $(\bimod{A^{(S^2)}}{A})^H$. This observation itself is not new as it has been remarked in \cite{MR3537815}. The point not mentioned in existing literature is that the $H$-comodule structure $\delta_{A^*}$ originates from the twisted module structure of the Nakayama functor. Namely, we have:

\begin{theorem}
  \label{thm:Nakayama-as-equivariant-bimodule}
  The object $A^* \in (\bimod{A^{(S^2)}}{A})^H$ corresponds to the Nakayama functor of $\lmod{A}$ under the equivalence of Lemma~\ref{lem:H-eq-EW-2}.
\end{theorem}
\begin{proof}
  We define the functor $\mathfrak{N} : \lmod{A} \to \lmod{A}$ by $\mathfrak{N}(M) := A^* \otimes_A M$ ($M \in \lmod{A}$) and equip it with a structure of a (lax) right $\mathcal{C}$-module functor from $\lmod{A}$ to $(\lmod{A})_{\DLD}$ by the inverse of \eqref{eq:H-eq-bimod-twisted-2}. Explicitly, the structure morphism of $\mathfrak{N}$ is given by
  \begin{equation}
    \begin{aligned}
      \label{eq:Naka-bimod-tw-mod}
      \Psi^{\mathfrak{N}}_{M,V} : \mathfrak{N}(M) \catactr V^{**}
      & \to \mathfrak{N}(M \catactr V), \\
      (a^* \otimes_{A} m) \otimes \phi_V(v)
      & \mapsto a^*_{(0)} \otimes_{A} (m \otimes S(a^*_{(1)})v) \\
    \end{aligned}
  \end{equation}
  for $m \in M \in \lmod{A}$, $v \in V \in \mathcal{C}$ and $a^* \in A^*$. As a functor, $\mathfrak{N}$ is identical to $\Nak$. We prove this theorem by showing that the structure morphism of $\mathfrak{N}$ is same as that of $\Nak$. By Lemma~\ref{lem:Naka-tw-mod-lem-1} and the universal property of $\Nak(M)$ as a coend, this is equivalent to that the diagram
  \begin{equation*}
    \begin{tikzcd}[column sep = 8em]
      \mathfrak{N}(M) \catactr V^{**}
      \arrow[d, "\Psi^{\mathfrak{N}}_{M,V}"']
      & (\Hom_{A}(M, X)^* \otimes_{\bfk} X) \catactr V^{**}
      \arrow[l, "\mathbb{i}_{X,M} \otimes \id_{V^{**}}"]
      \arrow[d, "(\close_{M, X | V})^* \otimes \id_{X} \otimes \id_{V^{**}}"] \\
      \mathfrak{N}(M \catactr V)
      & \Hom_{A}(M \catactr V^{**}, X \catactr V)^* \otimes_{\bfk} (X \catactr V^{**})
      \ar[l, "\mathbb{i}_{X \catactr V, M \catactr V^{**}}"]
    \end{tikzcd}
  \end{equation*}
  commutes for all objects $M, X \in \lmod{A}$ and $V \in \mathcal{C}$, where the associativity isomorphism for $\Vect$ is treated as the identity map.

  To prove the commutativity of the diagram, we require a technical identity \eqref{eq:Naka-bimod-pf-1} below. Let $\{ v_j \}$ be a basis of $V$, and let $\{ v^j \}$ be the basis of $V^*$ dual to $\{ v_j \}$. Then, as is well-known, the equation $h v_j \otimes v^j = v_j \otimes (v^j \leftharpoonup h)$ in $V \otimes_{\bfk} V^*$ holds for all $h \in H$, where the Einstein convention is used to suppress the sum over $j$. This allows us to define the $\bfk$-linear map
  \begin{align*}
    \Xi_{M, X | V} : X^* \otimes_A M
    & \to (X \catactr V^{**})^* \otimes_A (M \catactr V), \\
    x^* \otimes_A m
    & \mapsto (x^* \otimes \phi_{V^*}(v^j)) \otimes_A (m \otimes v_j),
  \end{align*}
  where $(X \catactr V^{**})^*$ is identified with $X^* \otimes V^{***}$ as a vector space. We note that the equation
  \begin{equation}
    \label{eq:Naka-bimod-pf-1}
    (\close_{M,X|V})^* \circ \varphi_{X, M}
    = \varphi_{X \catactr V^{**}, M \catactr V} \circ \Xi_{M, X | V}
  \end{equation}
  holds, where $\varphi_{-,-}$ is the natural isomorphism \eqref{eq:Nak-A-mod-lem-2}. Indeed, we have
  \begin{gather*}
    \langle (\close_{M,X|V})^* \varphi_{X,M}(x^* \otimes m), f \rangle
    = \langle \varphi_{X,M}(x^* \otimes m), \close_{M,X|V}(f) \rangle
    = \langle x^* , \close_{M,X|V}(f)(m) \rangle \\
    = \langle x^* \otimes \phi_{V^*}(v^j), f(m \otimes v_j) \rangle
    = \langle \varphi_{X \catactr V^{**}, M \catactr V} \Xi_{M, X | V} (x^* \otimes_A m), f \rangle
  \end{gather*}
  for $x^* \in X^*$, $m \in M$ and $f \in \Hom_{A}(M \catactr V, X \catactr V^{**})$.

  Now we pick elements $x^* \in X^*$, $x \in X$, $m \in M$ and $v \in V$ and chase the element
  \begin{equation*}
    w := \varphi_{X,M}(x^* \otimes_A m) \otimes x \otimes \phi_V(v)
    \in \Hom_{A}(M, X)^* \otimes_{\bfk} X \catactr V^{**}
  \end{equation*}
  around the diagram. For simplicity of notation, we write
  \begin{equation*}
    F_1 := \Psi^{\mathfrak{N}}_{M,V}(\mathbb{i}_{X,M} \otimes \id_{V^{**}})
    \quad \text{and} \quad
    F_2 := \mathbb{i}_{X \catactr V^{**}, M \catactr V} ((\close_{M,X|V})^* \otimes \id_X \otimes \id_{V^{**}}).
  \end{equation*}
  By \eqref{eq:Nak-A-mod-lem-3}, \eqref{eq:Naka-bimod-coaction-1}, \eqref{eq:Naka-bimod-tw-mod} and \eqref{eq:Naka-bimod-pf-1}, we compute
  \begin{align*}        
    F_1(w)
    & = \Psi^{\mathfrak{N}}_{M,V}((\langle x^*, ? x \rangle \otimes_A m) \otimes \phi_V(v)) \\
    & = \langle x^*, a_{i(0)} x \rangle a^i \otimes_{A} (m \otimes S^2(a_{i(1)}) v), \\
    F_2(w)
    & = \mathbb{i}_{X \catactr V^{**}, M \catactr V}
      (\varphi_{X \catactr V^{**}, M \catactr V}\Xi_{M, X | V}(x^* \otimes_A m)
      \otimes x \otimes \phi_V(v)) \\
    & = \mathbb{i}_{X \catactr V^{**}, M \catactr V}
      (\varphi_{X \catactr V^{**}, M \catactr V}((x^* \otimes \phi_{V^*}(v^j)) \otimes_A (m \otimes v_j))
      \otimes x \otimes \phi_V(v)) \\
    & = \langle x^* \otimes \phi_{V^*}(v^j), a_{i(0)} x \otimes a_{i(1)} \phi_V(v)) \rangle a^i
      \otimes_A (m \otimes v_j) \\
    & = \langle x^* \otimes \phi_{V^*}(v^j), a_{i(0)} x \otimes \phi_V(S^2(a_{i(1)})v)) \rangle a^i
      \otimes_A (m \otimes v_j) \\
    & = \langle x^*, a_{i(0)} x \rangle \langle v^j, S^2(a_{i(1)}) v) \rangle a^i
      \otimes_A (m \otimes v_j) \\
    & = \langle x^*, a_{i(0)} x \rangle a^i
      \otimes_A (m \otimes S^2(a_{i(1)}) v).
  \end{align*}
  Since $\varphi_{X,M}$ is an isomorphism, $F_1 = F_2$. The proof is done.
\end{proof}

\subsection{Twisted traces and cointegrals}

Let $A$ be a right $H$-comodule algebra, and let $\Sigma \in (\bimod{A^{(S^2)}}{A})^H$. We regard $\Sigma$ as a $\bfk$-linear right exact twisted $\mathcal{C}$-module functor by Lemma~\ref{lem:H-eq-EW-2}. As we have proved in Section~\ref{sec:twist-trace-HH}, the space $\TwTr(\Sigma)$ of $\Sigma$-twisted modified traces on $\lproj{A}$ is in bijection with $\HH_0(\Sigma)^*$. From the viewpoint of this bijection, module-compatible $\Sigma$-twisted traces are described as follows:

\begin{theorem}
  \label{thm:mod-tr-cointegral-1}
  The bijection $\TwTr(\Sigma) \cong \HH_0(\Sigma)^*$ restricts to a bijection between the following two sets:
  \begin{enumerate}
  \item The space of $\Sigma$-twisted module-compatible traces on $\lproj{A}$.
  \item The subspace of $\HH_0(\Sigma)^*$ consisting of all elements $\lambda$ satisfying
    \begin{equation}
      \label{eq:H-eq-bimod-cointegral}
      \lambda(s_{(0)}) s_{(1)} = \lambda(s) 1_H
      \quad (s \in \Sigma).
    \end{equation}
  \end{enumerate}
\end{theorem}
\begin{proof}
  \newcommand{\mTwTr}{\mathrm{Tr}_{\mathcal{C}}}
  For simplicity of notation, we denote by $\mTwTr(\Sigma)$ the space of $\Sigma$-twisted traces on $\lproj{A}$ compatible with the right $\mathcal{C}$-module structure. By our results, there are bijections
  \begin{equation}
    \TwTr(\Sigma)
    \xrightarrow[\cong]{\  \eqref{eq:V-valued-trace} \ }
    \HH_0(\Sigma)^*
    \xrightarrow[\cong]{\  \eqref{eq:HH0-and-bimodules} \ }
    \Hom_{A|A}(\Sigma, A^*)
    \xrightarrow[\cong]{\  \eqref{eq:EW-equivalence} \ }
    \Nat(\Sigma, \Nak_{\lmod{A}})
  \end{equation}
  and, under these bijections, the space $\mTwTr(\Sigma)$ corresponds to the space of module-compatible natural transformations from $\Sigma \to \Nak_{\lmod{A}}$.
  Thus the space $\mTwTr(\Sigma)$ is in bijection with the space
  \begin{equation*}
    \{ \lambda \in \HH_0(\Sigma)^* \mid \text{The map $\lambda^{\natural} : \Sigma \to A^*$ is $H$-colinear} \},
  \end{equation*}
  where $\lambda^{\natural}$ for $\lambda \in \Sigma^*$ is given by~\eqref{eq:lambda-nat}.
  In conclusion, to complete the proof, it suffices to show that $\lambda \in \HH_0(\Sigma)^*$ satisfies \eqref{eq:H-eq-bimod-cointegral} if and only if $\lambda^{\natural}$ is $H$-colinear. This can be verified directly. The proof is done.
\end{proof}

\subsection{Pivotal case}

A {\em pivotal element} of $H$ is a grouplike element $g \in H$ such that the equation $S^2(h) = g h g^{-1}$ holds for all $h \in H$. In this subsection, we assume that a pivotal element $g_{\piv} \in H$ is given. Thus $\mathcal{C} = \lmod{H}$ is pivotal with the pivotal structure $\mathfrak{p}$ defined by
\begin{equation*}
  \mathfrak{p}_X : X \to X^{**},
  \quad \langle \mathfrak{p}_X(x), x^* \rangle = \langle x^*, g_{\piv} \cdot x \rangle
  \quad (x \in X, x^* \in X^*)
\end{equation*}
for $X \in \mathcal{C}$.

Now let $A$ be a right $H$-comodule algebra. As in Subsection~\ref{subsec:FTC-pivotal-case}, we make $\id_{\lmod{A}}$ a $\DLD$-twisted right $\mathcal{C}$-module functor by the pivotal structure. By a {\em right modified trace} on $\lproj{A}$, we mean a module compatible $\Sigma$-twisted trace on $\lproj{A}$ with $\Sigma = \id_{\lmod{A}}$ ({\it cf}. Subsection~\ref{subsec:FTC-pivotal-case}).

\begin{theorem}
  \label{thm:mod-tr-cointegral-2}
  For a right $H$-comodule algebra $A$, there is a bijection between the following two spaces:
  \begin{enumerate}
  \item The space of right modified traces on $\lproj{A}$.
  \item The space of all linear maps $\lambda : A \to \bfk$ satisfying
    \begin{equation}
      \label{eq:g-sym-coint}
      \lambda(a b) = \lambda(b a), \quad
      \lambda(a_{(0)}) a_{(1)} = \lambda(a) g_{\piv}^{-1}
      \quad (a, b \in A).
    \end{equation}
  \end{enumerate}
\end{theorem}
\begin{proof}
  The $A$-bimodule $\Sigma := A$ becomes an object of $(\bimod{A^{(S^2)}}{A})^H$ by the coaction
  \begin{equation*}
    \delta_{\Sigma} : \Sigma \to \Sigma \otimes_{\bfk} H, \quad a \mapsto a_{(0)} \otimes g_{\piv} a_{(1)} \quad (a \in \Sigma).
  \end{equation*}
  The object $\Sigma$ corresponds, through the equivalence of Lemma \ref{lem:H-eq-EW-2}, to the identity functor on $\lmod{A}$ regarded as a $\DLD$-twisted right $\mathcal{C}$-module functor by the pivotal structure associated to $g_{\piv}$. Thus the proof is done by applying Theorem \ref{thm:mod-tr-cointegral-1} to $\Sigma$.
\end{proof}

\begin{remark}
  For the case where $A = H$, a linear map $\lambda$ satisfying \eqref{eq:g-sym-coint} is precisely a symmetrized right cointegral on $H$ in the sense of \cite{2018arXiv180100321B} and therefore this theorem specializes to the description of right modified traces on $\mathcal{C}$ given in \cite{2018arXiv180100321B}.
\end{remark}

\begin{remark}
  Let $g \in H$ be a grouplike element. A $g$-cointegral on $A$ \cite{2018arXiv181007114K} is a linear map $\lambda : A \to \bfk$ satisfying $\lambda(a_{(0)}) a_{(1)} = \lambda(a) g$ for all elements $a \in A$ (properties and examples of such `grouplike-cointegrals' are studied in \cite{2018arXiv181007114K,2019arXiv190400376S}). Theorem~\ref{thm:mod-tr-cointegral-2} implies that a modified trace on $\lmod{A}$ exists only if there is a non-zero $g_{\piv}^{-1}$-cointegral on $A$, but the converse does not hold as we will see in Example~\ref{ex:Taft-alg}.
\end{remark}

\begin{remark}
  By Theorem \ref{thm:mod-tr-cointegral-2}, the algebra $A$ is symmetric Frobenius if there is a non-degenerate modified trace on $\lproj{A}$. The converse does not hold in general; see Example~\ref{ex:book-Hopf-alg}.
\end{remark}

\subsection{Examples}

We close this paper by examining the existence of a non-zero right modified trace for some concrete examples. We fix an integer $N > 1$ and assume that the base field $\bfk$ has a primitive $N$-th root $\omega$ of unity.

\begin{example}
  \label{ex:Taft-alg}
  The Taft algebra $T_{\omega}$ is generated by $x$ and $g$ subject to the relations
  $g^N = 1$,
  $g x = \omega x g$ and
  $x^N = 0$.
  The Taft algebra $H := T_{\omega}$ has the unique Hopf algebra structure determined by $\Delta(g) = g \otimes g$ and $\Delta(x) = x \otimes g + 1 \otimes x$. The antipode is the anti-algebra automorphism given by $S(g) = g^{-1}$ and $S(x) = - x g^{-1}$. The Hopf algebra $H$ has a unique pivotal element $g_{\piv} := g$.

  Given a divisor $d$ of $N$, we define $A(d)$ to be the subalgebra of $H$ generated by $x$ and $g^{m}$, where $m = N/d$. It is easy to see that $A(d)$ is a right coideal subalgebra of $H$ and, in particular, it is a right $H$-comodule algebra by the comultiplication of $H$.

  The set $\{ x^i g^{m j} \mid i = 0, \cdots, N - 1; j = 0, \cdots, d - 1 \}$ is a basis of the vector space $A(d)$. With respect to this basis, we define the linear map $\lambda : A(d) \to \bfk$ by $\lambda(x^i g^{m j}) = \delta_{i, N - 1} \delta_{j,0}$. By the same way as \cite{2018arXiv181007114K} and \cite[Subsection 5.1]{2019arXiv190400376S}, we see that $\lambda$ is a $g_{\piv}^{-1}$-cointegral on $A(d)$ and every non-zero $g_{\piv}^{-1}$-cointegral on $A$ is a scalar multiple of $\lambda$.

  If $d = 1$, then the pairing on $A(d)$ associated to $\lambda$ is symmetric and non-degenerate. Thus, by Theorem~\ref{thm:mod-tr-cointegral-2}, there is a non-degenerate right modified trace on $\lproj{A(d)}$. On the other hand, if $d > 1$, then $a = g^{m}$ and $b = x g^{-m}$ does not satisfy $\lambda(a b) = \lambda(b a)$. Thus, by the theorem, we conclude that there are no non-zero right modified traces on $\lproj{A(d)}$ in this case.
\end{example}

\begin{example}
  \label{ex:book-Hopf-alg}
  We consider the Hopf algebra $H$ defined as follows: As an algebra, it is generated by $x$, $y$ and $g$ subject to the relations
  \begin{equation*}
    g^N = 1,
    \quad g x = \omega x g,
    \quad g y = \omega^{-1} y g,
    \quad y x = \omega x y
    \quad \text{and}
    \quad x^N = y^N = 0.
  \end{equation*}
  The Hopf algebra structure of $H$ is determined by
  \begin{equation*}
    \Delta(g) = g \otimes g,
    \quad \Delta(x) = x \otimes g + 1 \otimes x,
    \quad \Delta(y) = y \otimes g + 1 \otimes y.
  \end{equation*}
  The antipode is given by $S(g) = g^{-1}$, $S(x) = - x g^{-1}$ and $S(y) = - y g^{-1}$. It is easy to see that $g_{\piv} := g$ is a unique pivotal element of $H$.

  Given a divisor $d$ of $N$ and parameters $\xi, \mu \in \bfk$, we define $A(d; \xi, \mu)$ to be the right $H$-comodule algebra generated, as an algebra, by $G$, $X$ and $Y$ subject to the relations
  \begin{gather*}
    G^{d} = 1, \quad X^N = \xi 1, \quad Y^N = \mu 1, \\
    G X = \omega^{m} X G, \quad G Y = \omega^{-m} Y G, \quad Y X = \omega X Y.
  \end{gather*}
  The right $H$-comodule structure $\delta : A \to A \otimes H$ is determined by
  \begin{equation*}
    \delta(X) = X \otimes g + 1 \otimes x,
    \quad \delta(Y) = Y \otimes g + 1 \otimes y, \\
    \quad \delta(G) = G \otimes g^{m},
  \end{equation*}
  where $m = N/d$. This example is taken from \cite[Subsection 8.3]{MR2678630}, where indecomposable exact left comodule algebras over $H$ are classified up to equivariant Morita equivalence assuming that the base field is an algebraically closed field of characteristic zero.

  For simplicity, we write $A = A(d; \xi, \mu)$.
  We determine when $\lmod{A}$ admits a right non-zero modified trace.
  We first classify non-zero grouplike-cointegrals on $A$.
  The set
  $\{ X^r Y^s G^t \mid r, s = 0, \cdots, N - 1; t = 0, \cdots, d - 1 \}$
  is a basis of $A$. For $u \in \mathbb{Z}/d\mathbb{Z}$, we define the linear map $\lambda_u : A \to \bfk$ by
  \begin{equation*}
    \lambda_u(X^r Y^s G^t) = \delta_{r,N-1} \delta_{s,N-1} \delta_{t,u}
    \quad (r, s = 0, \cdots, N - 1; t = 0, \cdots, d - 1).
  \end{equation*}
  One can directly check that $\lambda_u$ is a Frobenius form on $A$ and the associated Nakayama automorphism $\nu_u$ is given by
  $\nu_u(X) = \omega^{m u - 1} X$,
  $\nu_u(Y) = \omega^{- m u + 1} Y$
  and $\nu_u(G) = G$.

  By the same way as \cite[Subsection 5.2]{2019arXiv190400376S}, we see that $\lambda_u$ is a $g^{m u - 2}$-cointegral on $A$ and every grouplike-cointegral on $A$ is a scalar multiple of $\lambda_u$ for some $u$. We note that the equation $g^{m u - 2} = g_{\piv}^{-1}$ holds if and only if $m u \equiv 1 \pmod{N}$. Thus a non-zero $g_{\piv}^{-1}$-cointegral on $A$ exists if and only if $m u \equiv 1 \pmod{N}$ for some $u$. Since $m = N/d$ is a divisor of $N$, this happens precisely if $d = N$.

  The above discussion and Theorem~\ref{thm:mod-tr-cointegral-2} show that there is a non-zero right modified trace on $\lproj{A}$ only if $d = N$. Now we suppose $d = N$.
  Then $\lambda_1$ is a $g_{\piv}^{-1}$-cointegral on $A$.
  Furthermore, the pairing on $A$ induced by $\lambda_1$ is symmetric and non-degenerate.
  Thus, by Theorem~\ref{thm:mod-tr-cointegral-2}, there is a non-degenerate right modified trace on $\lproj{A}$.

  Finally, we mention the case where $d < N$ and $\xi \mu \ne 0$. The elements $X$ and $Y$ are invertible in this case and $\nu_u$ is the inner automorphism implemented by $G^u X Y$. Hence, as a functor, $\id_{\lmod{A}}$ is isomorphic to the Nakayama functor $\Nak_{\lmod{A}}$.
  On the other hand, as we have seen in the above, there are no non-zero right modified traces on $\lproj{A}$.
  This implies that $\id_{\lmod{A}}$ is {\em not} isomorphic to $\Nak_{\lmod{A}}$ as a $\DLD$-twisted $\mathcal{C}$-module functor.
\end{example}

\def\cprime{$'$}

\end{document}